\documentclass[11pt]{amsart}

\newtheorem{thm}{Theorem}[section]
\newtheorem*{thm*}{Theorem}

\newtheorem{cor}[thm]{Corollary}

\newtheorem{lem}[thm]{Lemma}
\newtheorem*{lem*}{Lemma}
\newtheorem{mainthm}{Theorem}
\newtheorem*{mainthm*}{Theorem}
\newtheorem{maincor}[mainthm]{Corollary}

\newtheorem{prop}[thm]{Proposition}

\theoremstyle{definition}

\newtheorem*{case*}{Case}

\newtheorem{defn}[thm]{Definition}
\newtheorem*{defn*}{Definition}

\newtheorem*{exmp*}{Example}

\renewcommand{\thestep}{}

\theoremstyle{remark}

\renewcommand{\thecase}{}

\newtheorem{rmk}[thm]{Remark}
\newtheorem*{rmk*}{Remark}

\makeatletter
\def\alphenumi{
  \def\theenumi{\alph{enumi}}
  \def\p@enumi{\theenumi}
  \def\labelenumi{(\@alph\c@enumi)}}
\makeatother

\makeatletter
\def\thecase{\@arabic\c@case}
\makeatother

\makeatletter
\def\thestep{\@arabic\c@step}
\makeatother

\newcommand{\transv}{\mathrel{\text{\tpitchfork}}}
\makeatletter
\newcommand{\tpitchfork}{%
  \vbox{
    \baselineskip\z@skip
    \lineskip-.52ex
    \lineskiplimit\maxdimen
    \m@th
    \ialign{##\crcr\hidewidth\smash{$-$}\hidewidth\crcr$\pitchfork$\crcr}
  }%
}
\makeatother

\DeclareFontFamily{U}{mathx}{\hyphenchar\font45}
\DeclareFontShape{U}{mathx}{m}{n}{
      <5> <6> <7> <8> <9> <10>
      <10.95> <12> <14.4> <17.28> <20.74> <24.88>
      mathx10
      }{}
\DeclareSymbolFont{mathx}{U}{mathx}{m}{n}
\DeclareFontSubstitution{U}{mathx}{m}{n}
\DeclareMathAccent{\widecheck}{0}{mathx}{"71}
\DeclareMathAccent{\wideparen}{0}{mathx}{"75}

\newcount\hh
\newcount\mm
\mm=\time
\hh=\time
\divide\hh by 60
\divide\mm by 60
\multiply\mm by 60
\mm=-\mm
\advance\mm by \time
\def\hhmm{\number\hh:\ifnum\mm<10{}0\fi\number\mm}

\let\oldmarginpar\marginpar
\renewcommand\marginpar[1]{\-\oldmarginpar[\raggedleft\footnotesize #1]%
{\raggedright\footnotesize #1}}

\newcommand\CC{\mathbb{C}}

\newcommand\NN{\mathbb{N}}

\newcommand\RR{\mathbb{R}}

\newcommand\ZZ{\mathbb{Z}}

\newcommand\fg{{\mathfrak{g}}}

\newcommand\ft{{\mathfrak{t}}}

\newcommand\sA{{\mathscr{A}}}
\newcommand\sB{{\mathscr{B}}}
\newcommand\sC{{\mathscr{C}}}

\newcommand\sE{{\mathscr{E}}}

\newcommand\sG{{\mathscr{G}}}
\newcommand\sH{{\mathscr{H}}}
\newcommand\sI{{\mathscr{I}}}

\newcommand\sK{{\mathscr{K}}}
\newcommand\sL{{\mathscr{L}}}
\newcommand\sM{{\mathscr{M}}}
\newcommand\sN{{\mathscr{N}}}

\newcommand\sR{{\mathscr{R}}}

\newcommand\sU{{\mathscr{U}}}
\newcommand\sV{{\mathscr{V}}}

\newcommand\sX{{\mathscr{X}}}
\newcommand\sY{{\mathscr{Y}}}

\newcommand\bH{{\mathbf{H}}}

\newcommand\eps{\varepsilon}

\newcommand\SO{\operatorname{SO}}

\newcommand\SU{\operatorname{SU}}
\newcommand\U{\operatorname{U}}

\newcommand\less{\setminus}

\newcommand\ad{{\operatorname{ad}}}
\newcommand\Ad{{\operatorname{Ad}}}

\newcommand\Aut{\operatorname{Aut}}

\newcommand\Coker{\operatorname{Coker}}

\DeclareMathOperator{\Crit}{Crit}

\newcommand\dist{\operatorname{dist}}

\newcommand\End{\operatorname{End}}

\DeclareMathOperator{\genus}{genus}

\newcommand\grad{\operatorname{grad}}

\newcommand\Hess{\operatorname{Hess}}

\newcommand\Hom{\operatorname{Hom}}

\DeclareMathOperator{\Inj}{Inj}

\newcommand\Ker{\operatorname{Ker}}

\newcommand\Ran{\operatorname{Ran}}

\newcommand\Riem{\operatorname{Riem}}

\newcommand\supp{\operatorname{supp}}

\newcommand\vol{\operatorname{vol}}

\DeclareMathOperator{\YM}{\mathscr{Y}\!\!\mathscr{M}}

\newcommand\apriori{{\emph{a priori }}}

\newcommand\id{{\mathrm{id}}}

\newcommand\loc{{\mathrm{loc}}}

\newcommand\mutatis{{\emph{mutatis mutandis }}}

\numberwithin{equation}{section}
\numberwithin{figure}{section}

\usepackage{graphicx}
\usepackage{hyperref}
\usepackage{mathrsfs}
\usepackage[usenames]{color}
\usepackage{paralist}
\usepackage{slashed}
\usepackage{txfonts}
\usepackage{url}
\usepackage{verbatim}
\usepackage{xr}

\usepackage{amsmath}
\usepackage{amscd}
\usepackage{amssymb}
\usepackage{lmodern}

\hypersetup{pdftitle={Optimal Lojasiewicz Inequalities and Morse--Bott Yang--Mills Energy Functions}}
\hypersetup{pdfauthor={Paul M. N. Feehan}}

\begin{document}

\title[Optimal {\L}ojasiewicz--Simon Inequalities]{Optimal {\L}ojasiewicz--Simon Inequalities and Morse--Bott Yang--Mills Energy Functions}

\author[Paul M. N. Feehan]{Paul M. N. Feehan}
\address{Department of Mathematics, Rutgers, The State University of New Jersey, 110 Frelinghuysen Road, Piscataway, NJ 08854-8019, United States of America}
\email{feehan@math.rutgers.edu}

\date{November 16, 2020}

\begin{abstract}
For any compact Lie group $G$ and closed, smooth Riemannian manifold $(X,g)$ of dimension $d\geq 2$, we extend a result due to Uhlenbeck (1985) that gives existence of a flat connection on a principal $G$-bundle over $X$ supporting a connection with $L^p$-small curvature, when $p>d/2$, to the case of a connection with $L^{d/2}$-small curvature. We prove an optimal {\L}ojasiewicz--Simon gradient inequality for abstract Morse--Bott functions on Banach manifolds, generalizing an earlier result due to the author and Maridakis (2019), principally by removing the hypothesis that the Hessian operator be Fredholm with index zero. We apply this result to prove the optimal {\L}ojasiewicz--Simon gradient inequality for the self-dual Yang--Mills energy function near regular anti-self-dual connections over closed Riemannian four-manifolds and for the full Yang--Mills energy function over closed Riemannian manifolds of dimension $d \geq 2$, when known to be Morse--Bott at a given Yang--Mills connection. We also prove the optimal {\L}ojasiewicz--Simon gradient inequality by direct analysis near a given flat connection that is a regular point of the curvature map. 
\end{abstract}



\thanks{The author was partially supported by National Science Foundation grant DMS-1510064, the Simons Center for Geometry and Physics, Stony Brook, the Dublin Institute for Advanced Studies, and the Institut des Hautes {\'E}tudes Scientifiques, Bures-sur-Yvette, during the preparation of this article.}

\maketitle
\tableofcontents

\section{Introduction}
\label{sec:Introduction}
Since its discovery by {\L}ojasiewicz in the context of analytic functions on Euclidean spaces\footnote{The first page number refers to the version of {\L}ojasiewicz's original manuscript mimeographed by IHES while the page number in parentheses refers to the cited LaTeX version of his manuscript prepared by M. Coste and available on the Internet.} \cite[Proposition 1, p. 92 (67)]{Lojasiewicz_1965} and subsequent generalization by Simon to a class of analytic functions on certain H\"older spaces \cite[Theorem 3]{Simon_1983}, the \emph{{\L}ojasiewicz--Simon gradient inequality} has played a significant role in analyzing questions such as
\begin{inparaenum}[\itshape a\upshape)]
\item global existence, convergence, and analysis of singularities for solutions to nonlinear evolution equations that are realizable as gradient-like systems for an energy function,
\item uniqueness of tangent cones, and
\item energy gaps and discreteness of energies.
\end{inparaenum}
For a survey of applications of the {\L}ojasiewicz--Simon gradient inequality to gradient flows for real analytic functions on Banach spaces, including energy functions arising in applied mathematics, geometric analysis, or mathematical physics, we refer the reader to our article \cite{Feehan_Maridakis_Lojasiewicz-Simon_Banach} and monograph \cite{Feehan_yang_mills_gradient_flow_v4}.

In this article, which complements \cite{Feehan_yangmillsenergygap, Feehan_yangmillsenergygapflat_aim}, we establish optimal gradient inequalities of {\L}ojasiewicz--Simon type for the Yang--Mills and self-dual Yang--Mills energy functions and for $C^2$ functions on Banach spaces that are Morse--Bott near a critical point. These inequalities are proved by direct analysis and, in particular, none are proved by reduction to a {\L}ojasiewicz gradient inequality that is known to hold in finite dimensions. Optimal gradient inequalities (exponent $1/2$) are important because they imply that the gradient flow converges at an exponential (rather than power law) rate to the critical set \cite{Feehan_yang_mills_gradient_flow_v4}.

Suppose now that $G$ is a compact Lie group $G$ and $(X,g)$ is a closed, smooth Riemannian manifold of dimension $d\geq 2$. Our first main result, Theorem \ref{mainthm:Uhlenbeck_Chern_corollary_4-3_1_lessthan_p_lessthan_d}, extends a result due to Uhlenbeck \cite[Corollary 4.3]{UhlChern}, which gives existence of a flat connection on a principal $G$-bundle supporting a $W^{1,p}$ connection with $L^p$-small curvature, for $p>d/2$, to the case of a $W^{1,p}$ connection with $L^{d/2}$-small curvature. 

Next, we establish an optimal {\L}ojasiewicz gradient inequality (see Theorem \ref{mainthm:Lojasiewicz-Simon_gradient_inequality_Morse-Bott}) for abstract Morse--Bott functions on Banach manifolds, generalizing an earlier result due to the author and Maridakis \cite[Theorem 4]{Feehan_Maridakis_Lojasiewicz-Simon_Banach}. Our proof of Theorem \ref{mainthm:Lojasiewicz-Simon_gradient_inequality_Morse-Bott} is simpler than that of \cite[Theorem 4]{Feehan_Maridakis_Lojasiewicz-Simon_Banach} and, moreover, allows us to drop the requirement that the Hessian operator is Fredholm with index zero.

We apply Theorem \ref{mainthm:Lojasiewicz-Simon_gradient_inequality_Morse-Bott} to prove that the Yang--Mills energy function obeys the optimal {\L}ojasiewicz gradient inequality given by Theorem \ref{mainthm:Lojasiewicz-Simon_inequalities_Yang-Mills_energy_flat_Morse-Bott} when one restricts to a neighborhood of a flat connection $\Gamma$ that is a regular point of the curvature map, $A \mapsto F_A$, and hence that the Yang--Mills energy function is Morse--Bott near $\Gamma$. We also prove Theorem \ref{mainthm:Lojasiewicz-Simon_inequalities_Yang-Mills_energy_flat_Morse-Bott} by direct analysis without appealing to Theorem \ref{mainthm:Lojasiewicz-Simon_gradient_inequality_Morse-Bott}. 

We prove an optimal {\L}ojasiewicz gradient inequality for the self-dual Yang--Mills energy function near anti-self-dual connections, over closed Riemannian four-manifolds, that are regular points of the self-dual curvature map, $A \mapsto F_A^+$ (see Theorem \ref{mainthm:Optimal_Lojasiewicz-Simon_inequalities_self-dual_Yang-Mills_energy_function}). Finally, we prove an optimal {\L}ojasiewicz gradient inequality for the full Yang--Mills energy function over closed Riemannian manifolds of dimension $d \geq 2$, when known to be Morse--Bott at a given Yang--Mills connection (see Theorem \ref{mainthm:Lojasiewicz-Simon_inequalities_Yang-Mills_energy_Morse-Bott}).

Throughout this article, our conventions and notation are consistent with those of its two predecessors \cite{Feehan_yangmillsenergygap, Feehan_yangmillsenergygapflat_aim} and generally follow those of standard references such as Donaldson and Kronheimer \cite{DK}, Freed and Uhlenbeck \cite{FU}, and Friedman and Morgan \cite{FrM}. We shall not repeat those explanations here but we include a brief summary of our conventions and notation in Section \ref{subsec:Notation} for ease of reference.

\subsection{Existence of a flat connection in the case of critical Sobolev exponents}
Our first main result is a generalization, Theorem \ref{mainthm:Uhlenbeck_Chern_corollary_4-3_1_lessthan_p_lessthan_d} below, of part of Uhlenbeck's \cite[Corollary 4.3]{UhlChern} from the non-borderline case, $d/2<p<d$ and $L^p$-small curvature $F_A$, to $1<p<d$ and the \emph{borderline} case of $L^{d/2}$-small curvature. The relevant part of Uhlenbeck's \cite[Corollary 4.3]{UhlChern} is quoted in this article as Theorem \ref{thm:Uhlenbeck_Chern_corollary_4-3}. 

\begin{mainthm}[Existence of a flat connection on a principal bundle supporting a $W^{1,q}$ connection with $L^{d/2}$-small curvature, Coulomb gauge transformation, and Sobolev distance estimate]
\label{mainthm:Uhlenbeck_Chern_corollary_4-3_1_lessthan_p_lessthan_d}
Let $(X,g)$ be a closed, smooth Riemannian manifold of dimension $d\geq 2$, and $G$ be a compact Lie group, and $q \in (d/2, \infty]$ and $s_0 > 1$ be constants. Then there is a constant $\eps=\eps(g,G,s_0) \in (0,1]$ with the following significance. If $A$ is a $W^{1,q}$ connection on a smooth principal $G$-bundle $P$ over $X$ such that\footnote{We may choose $s_0>1$ arbitrarily close to $1$ when $d=2$ and, in particular, small enough that $|F_A| \in L^{s_0}(X;\RR)$.}
\begin{equation}
\label{eq:mainCurvature_Ldover2_small}
\|F_A\|_{L^{s_0}(X)} \leq \eps,
\end{equation}
where $s_0 = d/2$ when $d\geq 3$ or $s_0 > 1$ when $d=2$, then there is a $W^{1,q}$ flat connection $\Gamma$ on $P$. If $d\geq 3$ or $q \geq 4/3$ when $d=2$, then\footnote{By Wehrheim \cite[Theorem 9.4 (i)]{Wehrheim_2004}.} there is a $W^{2,q}$ gauge transformation $u$ of $P$ such that $u(\Gamma)$ is $C^\infty$.
\end{mainthm}

Theorem \ref{mainthm:Uhlenbeck_Chern_corollary_4-3_1_lessthan_p_lessthan_d} is most striking and useful in applications to Morse theory for the Yang--Mills energy function \eqref{eq:Yang-Mills_energy_function} in dimension $d=4$, in which case the condition \eqref{eq:mainCurvature_Ldover2_small} is equivalent to a requirement that the energy
\[
\YM(A) = \frac{1}{2}\int_X |F_A|^2\,d\vol_g
\]
be sufficiently small. We prove Theorem \ref{mainthm:Uhlenbeck_Chern_corollary_4-3_1_lessthan_p_lessthan_d} in Section \ref{sec:Existence_flat_connection_critical_exponents_Coulomb_distance}.

\begin{rmk}[Existence of flat connections in the case of borderline control over curvature]
\label{rmk:Existence_flat_connection_borderline_curvature_small}
The novel aspect of Theorem \ref{mainthm:Uhlenbeck_Chern_corollary_4-3_1_lessthan_p_lessthan_d} is the sufficiency (when $d \geq 3$) of the borderline hypothesis $\|F_A\|_{L^{d/2}(X)} \leq \eps$ in \eqref{eq:mainCurvature_Ldover2_small} to provide existence of a flat connection $\Gamma$ on the \emph{same} principal $G$-bundle $P$ as that supporting the connection $A$ with $L^{d/2}$-small curvature. The well-known argument due to Sedlacek \cite{Sedlacek} when $d=4$ would produce a flat connection $\Gamma$ on a possibly different principal $G$-bundle $Q$ but the classification of principal $G$-bundles, knowledge of the vector Pontrjagin classes, and the behavior of Sedlacek's obstruction class under weak limits ensures that $Q \cong P$ as continuous principal $G$-bundles. However, this is not how we prove Theorem \ref{mainthm:Uhlenbeck_Chern_corollary_4-3_1_lessthan_p_lessthan_d}. See the Introduction to Section \ref{subsec:Uhlenbeck_existence_flat_connection_and_apriori_estimate} for a discussion of this approach and further details. 

Instead, recall that Uhlenbeck's \cite[Theorem 1.3 or Theorem 2.1 and Corollary 2.2]{UhlLp} gives existence of local Coulomb gauges and \apriori estimates for local connection one-forms with $L^{d/2}$-small curvature. An application of her \cite[Theorem 1.3 or Theorem 2.1 and Corollary 2.2]{UhlLp} to a minimizing sequence of connections yields $W^{1,d/2}$ convergence of local connection one-forms and $W^{2,d/2}$ convergence of local gauge transformations. The Sobolev Embedding \cite[Theorem 4.12]{AdamsFournier} implies that $W^{2,p}(X;\RR) \subset C^0(X;\RR)$ is a continuous embedding when $p>d/2$ but not when $p=d/2$ and thus Uhlenbeck's patching arguments do not appear applicable at first glance. However, as we explain in Sections \ref{subsec:Continuous_principal_bundles} and \ref{subsec:Uhlenbeck_Chern_corollary_4-3_existence_flat_connection_critical_exponent}, the fact that the gauge-transformed local connection one-forms obey a Coulomb gauge condition is sufficient to give us $W^{2,p}$ and thus $C^0$ control over local gauge transformations with $p>d/2$ and this directly yields the isomorphism $Q \cong P$, without appeal to the classification of principal $G$-bundles --- see Theorems \ref{thm:Uhlenbeck_Lp_bound_3-2_Sobolev_Ld_small_connection_oneforms} and
\ref{thm:Uhlenbeck_Chern_corollary_4-3_critical_existence}. Partly related results were proved by Taubes \cite[Proposition 4.5 and Lemma A.1]{TauPath} when $d=4$, using a more difficult method, and by Riv{\`e}re \cite[Theorem IV.1]{Riviere_2002} when $d \geq 4$, using Lorentz spaces rather than the standard Sobolev spaces that we employ throughout this article. See Remark \ref{rmk:Related_results_Rivere_Taubes} for further discussion of the results due to Riv{\`e}re and Taubes and Remark \ref{rmk:Related_results_Isobe_Shevchishin} for a discussion of related results due to Isobe \cite{Isobe_2009} and Shevchishin \cite{Shevchishin_2002}.
\end{rmk}

\subsection{{\L}ojasiewicz--Simon gradient inequalities for Morse--Bott functions}
\label{subsec:Lojasiewicz-Simon_gradient_inequality_Morse-Bott}
In applications to geometry and topology, it is very useful to know when a given energy function is a Morse function (isolated critical points) or more generally a Morse--Bott function (non-isolated critical points). 

\begin{defn}[Morse--Bott function]
\label{defn:Morse-Bott_function}
(See Austin and Braam \cite[Section 3.1]{Austin_Braam_1995}.)
Let $\sB$ be a smooth Banach manifold, $\sE: \sB \to \RR$ be a $C^2$ function, and $\Crit\sE := \{x\in\sB:\sE'(x) = 0\}$. A smooth submanifold $\sC \hookrightarrow \sB$ is called a \emph{nondegenerate critical submanifold of $\sE$} if $\sC \subset \Crit\sE$ and
\begin{equation}
\label{eq:Nondegenerate_critical_submanifold}
(T\sC)_x = \Ker \sE''(x), \quad\forall\,x\in \sC,
\end{equation}
where $\sE''(x):(T\sB)_x \to (T\sB)_x^*$ is the Hessian of $\sE$ at the point $x \in \sC$. One calls $\sE$ a \emph{Morse--Bott function} if its critical set $\Crit\sE$ consists of nondegenerate critical submanifolds.

We say that a $C^2$ function $\sE:\sB\to\RR$ is \emph{Morse--Bott at a point $x_0 \in \sB$} if there is an open neighborhood $\sU\subset\sB$ of $x_0$ such that $\sU\cap\Crit\sE$ is a relatively open, smooth submanifold of $\sB$ and \eqref{eq:Nondegenerate_critical_submanifold} holds at $x_0$.
\end{defn}

In Definition \ref{defn:Morse-Bott_function}, if we had only assumed that $\sC \hookrightarrow \sB$ is a smooth submanifold with $\sC \subset \Crit\sE$, we would still have the inclusion,
\[
(T\sC)_x \subset \Ker \sE''(x),
\]
for each $x\in \sC$. Hence, the key assertion in \eqref{eq:Nondegenerate_critical_submanifold} is that \emph{equality} holds and thus each vector $v \in (T\sB)_x\cap\Ker \sE''(x)$ is \emph{integrable}, the tangent vector to a smooth path in $\sC$ through $x$.

Definition \ref{defn:Morse-Bott_function} is a restatement of definitions of a Morse--Bott function on a finite-dimensional manifold, but we omit the condition that $\sC$ be compact and connected as in Nicolaescu \cite[Definition 2.41]{Nicolaescu_morse_theory} or the condition that $\sC$ be compact in Bott \cite[Definition, p. 248]{Bott_1954}. Note that if $\sB$ is a Riemannian manifold and $\sN$ is the normal bundle of $\sC \hookrightarrow \sB$, so $\sN_x = (T\sC)_x^\perp$ for all $x \in \sC$, where $(T\sC)_x^\perp$ is the orthogonal complement of $(T\sC)_x$ in $(T\sB)_x$, then \eqref{eq:Nondegenerate_critical_submanifold} is equivalent to the assertion that the restriction of the Hessian to the fibers of the normal bundle of $\sC$,
\[
\sE''(x):\sN_x \to (T\sB)_x^*,
\]
is \emph{injective} for all $x \in \sC$; using the Riemannian metric on $\sB$ to identify $(T\sB)_x^* \cong (T\sB)_x$, we see that $\sE''(x):\sN_x \cong \sN_x$ is an isomorphism for all $x \in \sC$. In other words, the condition \eqref{eq:Nondegenerate_critical_submanifold} is equivalent to the assertion that the Hessian of $\sE$ is an isomorphism of the normal bundle $\sN$ when $\sB$ has a Riemannian metric.

For a development of Morse--Bott theory and a discussion of and references to its numerous applications, we refer to Austin and Braam \cite{Austin_Braam_1995}, Banyaga and Hurtubise \cite{Banyaga_Hurtubise_lectures_morse_homology, Banyaga_Hurtubise_2009, Banyaga_Hurtubise_2010, Banyaga_Hurtubise_2013}, Nicolaescu \cite{Nicolaescu_morse_theory}, and references cited therein.

\begin{defn}[Gradient map]
\label{defn:Huang_2-1-1}
(See Berger \cite[Section 2.5]{Berger_1977}, Huang \cite[Definition 2.1.1]{Huang_2006}.)
Let $\sU\subset \sX$ be an open subset of a Banach space $\sX$ and let $\sY$ be a Banach space with continuous embedding $\sY \subseteqq \sX^*$. A continuous map $\sM:\sU\to \sY$ is called a \emph{gradient map} if there exists a $C^1$ function $\sE:\sU\to\RR$ such that
\begin{equation}
\label{eq:Differential_and_gradient_maps}
\sE'(x)v = \langle v,\sM(x)\rangle_{\sX\times\sX^*}, \quad \forall\, x \in \sU, \quad v \in \sX,
\end{equation}
where $\langle \cdot , \cdot \rangle_{\sX\times\sX^*}$ is the canonical bilinear form on $\sX\times\sX^*$. The real-valued function $\sE$ is called a \emph{potential} for the gradient map $\sM$.
\end{defn}

When $\sY = \sX^*$ in Definition \ref{defn:Huang_2-1-1}, then the differential and gradient maps coincide.

\begin{mainthm}[{\L}ojasiewicz--Simon gradient inequality for $C^2$ Morse--Bott functions on Banach spaces]
\label{mainthm:Lojasiewicz-Simon_gradient_inequality_Morse-Bott}
(Compare Feehan and Maridakis \cite[Theorems 3 and 4]{Feehan_Maridakis_Lojasiewicz-Simon_Banach}.)
Let $\sX$, $\sY$, $\sG$, and $\sH$ be Banach spaces with continuous embeddings,
\[
\sX\subset \sG \quad\text{and}\quad \sY \subset \sH \subset \sG^* \subset \sX^*.
\]
Let $\sU \subset \sX$ be an open subset, $\sE:\sU\to\RR$ be a $C^2$ function, and $x_\infty\in\sU$ be a critical point of $\sE$, so $\sE'(x_\infty) = 0$. Let $\sM:\sU\to\sY$ be a $C^1$ gradient map for $\sE$ in the sense of Definition \ref{defn:Huang_2-1-1} and require that $\sE$ be Morse--Bott at $x_\infty$ in the sense of Definition \ref{defn:Morse-Bott_function}, so $\sU\cap\Crit\sE$ is a relatively open, smooth submanifold of $\sX$ and $K := \Ker\sE''(x_\infty) = T_{x_\infty}\Crit\sE$. Suppose that for each $x \in \sU$, the bounded linear operator
\[
\sM'(x): \sX \to \sY
\]
has an extension
\[
\sM_1(x): \sG \to \sH
\]
such that the following map is continuous,
\[
\sU \ni x \mapsto \sM_1(x) \in \sL(\sG, \sH).
\]
Assume that $K\subset\sX$ has a closed complement $\sX_0 \subset \sX$, that $\sK := \Ker\sM_1(x_\infty)\subset\sG$ has a closed complement $\sG_0 \subset \sG$ with $\sX_0 \subset \sG_0$, and that $\Ran\sM_1(x_\infty) \subset \sH$ is a closed subspace. Then there are constants $Z \in (0,\infty)$ and $\sigma \in (0,1]$ with the following significance. If $x \in \sU$ obeys
\begin{equation}
\label{eq:Lojasiewicz-Simon_gradient_inequality_neighborhood_Morse-Bott}
\|x-x_\infty\|_\sX < \sigma,
\end{equation}
then
\begin{equation}
\label{eq:Lojasiewicz-Simon_gradient_inequality_Morse-Bott}
\|\sM(x)\|_{\sH} \geq Z|\sE(x) - \sE(x_\infty)|^{1/2}.
\end{equation}
\end{mainthm}

We prove Theorem \ref{mainthm:Lojasiewicz-Simon_gradient_inequality_Morse-Bott} in Section \ref{sec:Lojasiewicz-Simon_gradient_inequality_abstract_functional_Morse-Bott}.

\begin{rmk}[Previous versions of the {\L}ojasiewicz--Simon gradient inequality for $C^2$ Morse--Bott functions on abstract Banach spaces]
\label{rmk:Previous_versions_Lojasiewicz-Simon_gradient_inequality_Morse-Bott_abstract_Banach_spaces}
Previous versions of Theorem \ref{mainthm:Lojasiewicz-Simon_gradient_inequality_Morse-Bott} were proved by Simon \cite[Lemma 3.13.1]{Simon_1996} (for a harmonic map energy function on a Banach space of $C^{2,\alpha}$ sections of a Riemannian vector bundle), Haraux and Jendoubi \cite[Theorem 2.1]{Haraux_Jendoubi_2007} (for functions on abstract Hilbert spaces) and in more generality by Chill in \cite[Corollary 3.12]{Chill_2003} (for functions on abstract Banach spaces); a more elementary version was proved by Huang as \cite[Proposition 2.7.1]{Huang_2006} (for functions on abstract Banach spaces). These authors do not use Morse--Bott terminology but their hypotheses imply this condition --- directly in the case of Haraux and Jendoubi and Chill and by a remark due to Simon in \cite[p. 80] {Simon_1996} that his integrability condition \cite[Equation (iii), p. 79]{Simon_1996} is equivalent to a restatement of the Morse--Bott condition. Their gradient inequalities are less general than our Theorem \ref{mainthm:Lojasiewicz-Simon_gradient_inequality_Morse-Bott}. See Feehan \cite[Remark 1.16 and Appendix C]{Feehan_lojasiewicz_inequality_all_dimensions_morse-bott} for further discussion of the relationship between definitions of integrability, such as those described by Adams and Simon \cite{Adams_Simon_1988}, and the Morse--Bott condition.
\end{rmk}

\begin{rmk}[On the proof of Theorem \ref{mainthm:Lojasiewicz-Simon_gradient_inequality_Morse-Bott}]
\label{rmk:Lojasiewicz-Simon_gradient_inequality_Morse-Bott_on_proof}
Special cases of Theorem \ref{mainthm:Lojasiewicz-Simon_gradient_inequality_Morse-Bott} can be obtained as consequences of suitable Morse--Bott lemmas (see Feehan \cite{Feehan_lojasiewicz_inequality_all_dimensions_morse-bott} for a discussion and references). However, proofs of Morse--Bott lemmas require care and it is unclear whether one would hold in the generality provided by Theorem \ref{mainthm:Lojasiewicz-Simon_gradient_inequality_Morse-Bott}. On the other hand, the proof of Theorem \ref{mainthm:Lojasiewicz-Simon_gradient_inequality_Morse-Bott} provided in Section \ref{sec:Lojasiewicz-Simon_gradient_inequality_abstract_functional_Morse-Bott} is quite direct.
\end{rmk}

\begin{rmk}[Comparison between Inequality \eqref{eq:Lojasiewicz-Simon_gradient_inequality_Morse-Bott} and other {\L}ojasiewicz-Simon gradient inequalities]
In \cite[Theorems 1, 2, and 3]{Feehan_Maridakis_Lojasiewicz-Simon_Banach}, Maridakis and the author establish versions of Theorem \ref{mainthm:Lojasiewicz-Simon_gradient_inequality_Morse-Bott} where the inequality \eqref{eq:Lojasiewicz-Simon_gradient_inequality_Morse-Bott} is replaced by
\begin{equation}
\label{eq:Lojasiewicz-Simon_gradient_inequality}
\|\sM(x)\|_{\sH} \geq Z|\sE(x) - \sE(x_\infty)|^\theta,
\end{equation}
for some $\theta \in [1/2,1)$, the operators $\sM'(x_\infty)$ and $\sM_1(x_\infty)$ are Fredholm with index zero, and $\sM:\sU\to\sY$ is real analytic. Those results are proved with the aid of a \emph{Lyapunov-Schmidt reduction} of $\sE$ (for example, \cite[Proposition 5.1]{Huang_2006}) to a real analytic function on an open neighborhood of the origin in Euclidean space and appealing to {\L}ojasiewicz's gradient inequality \cite{Lojasiewicz_1959, Lojasiewicz_1961, Lojasiewicz_1965}, with a simplified proof provided by Bierstone and Milman \cite[Theorem 6.4 and Remark 6.5]{BierstoneMilman}. However, the requirement that the operators $\sM'(x_\infty)$ and $\sM_1(x_\infty)$ be Fredholm can be restrictive. For example, in the context of Yang--Mills or coupled Yang--Mills energy functions, one must take a quotient of the affine space of all $W^{1,q}$ connections or pairs by the Banach Lie group $\Aut(P)$ of $W^{2,q}$ gauge transformations and that action can introduce singularities in the quotient space as we recall in Section \ref{subsec:Optimal_Lojasiewicz-Simon_inequalities_self-dual_Yang-Mills_energy_Morse-Bott}.
\end{rmk}

\begin{rmk}[Optimal {\L}ojasiewicz-Simon gradient inequalities and exponential convergence of gradient flow]
It is of considerable interest to know when the optimal exponent $\theta = 1/2$ is achieved, since in that case one can prove (for example, \cite[Theorem 24.21]{Feehan_yang_mills_gradient_flow_v4}) that a global solution $u:[0,\infty)\to\sX$ to a gradient system governed by the {\L}ojasiewicz--Simon gradient inequality,
\[
\frac{du}{dt} = -\sE'(u(t)), \quad u(0) = x_0,
\]
has \emph{exponential} rather than mere power-law rate of convergence
to the critical point $x_\infty$. See \cite[Section 2.1]{Feehan_yang_mills_gradient_flow_v4} for a detailed summary of results of this kind.
\end{rmk}

\begin{rmk}[Comparison between Theorem \ref{mainthm:Lojasiewicz-Simon_gradient_inequality_Morse-Bott} and a previous result due to the author and Maridakis]
Theorem \ref{mainthm:Lojasiewicz-Simon_gradient_inequality_Morse-Bott} is a generalization of our previous \cite[Theorems 3 and 4]{Feehan_Maridakis_Lojasiewicz-Simon_Banach}, but the advantage of Theorem \ref{mainthm:Lojasiewicz-Simon_gradient_inequality_Morse-Bott} here is that the operators $\sM'(x_\infty)$ and $\sM_1(x_\infty)$ are not required to be Fredholm with index zero. While that generalization can be established by modifying the proofs of \cite[Theorems 3 and 4]{Feehan_Maridakis_Lojasiewicz-Simon_Banach}, we instead give a more direct and much simpler proof in Section \ref{subsec:Lojasiewicz-Simon_gradient_inequality_Morse-Bott}. The latter proof also allows us to slightly relax other hypotheses on the Banach spaces and their embeddings. Of course, when $\sM'(x_\infty)$ or $\sM_1(x_\infty)$ are Fredholm operators, then their kernels are finite-dimensional and thus have closed complements by \cite[Lemma 4.21 (a)]{Rudin}, and their ranges are closed.
\end{rmk}

\begin{rmk}[Choices of the Banach spaces $\sG$ and $\sH$]
In typical applications of Theorem \ref{mainthm:Lojasiewicz-Simon_gradient_inequality_Morse-Bott} one chooses $\sG$ and $\sH$ to be Hilbert spaces and that simplifies the statement of the theorem since a closed subspace of a Hilbert space necessarily has a closed (orthogonal) complement \cite[Theorem 12.4]{Rudin}. However, the greater generality allows us to quickly infer several corollaries (see the forthcoming Corollaries \ref{maincor:Lojasiewicz-Simon_gradient_inequality_Morse-Bott_GisX} and \ref{maincor:Lojasiewicz-Simon_gradient_inequality_Morse-Bott_YisXdual}) analogous to \cite[Theorems 1, 2, and 4]{Feehan_Maridakis_Lojasiewicz-Simon_Banach} and whose statements are shorter and thus more easily understood, but Theorem \ref{mainthm:Lojasiewicz-Simon_gradient_inequality_Morse-Bott} is the most useful version in applications to proofs of global existence and convergence of gradient flows. For example, Theorem \ref{mainthm:Lojasiewicz-Simon_gradient_inequality_Morse-Bott} is the only version that yields Simon's \cite[Theorem 3]{Simon_1983} for all dimensions of the base manifold, $X$, with $\sX = C^{2,\alpha}(X;V)$ and $\sH = L^2(X;V)$ (where $V$ is a Riemannian vector bundle over $X$), and, moreover, for a wide variety of alternative choices of H{\"o}lder or Sobolev spaces for $\sX$; see \cite[Remark 1.14]{Feehan_Maridakis_Lojasiewicz-Simon_Banach}.
\end{rmk}

\begin{rmk}[Harmonic map energy function for maps from a Riemann surface into a closed Riemannian manifold]
For the harmonic map energy function, an optimal {\L}ojasiewicz--Simon gradient inequality,
\[
\|\sE'(f)\|_{L^p(S^2)} \geq Z|\sE(f) - \sE(f_\infty)|^{1/2},
\]
has been obtained by Kwon \cite[Theorem 4.2]{KwonThesis} for maps $f:S^2\to N$, where $N$ is a closed Riemannian manifold and $f$ is close to a harmonic map $f_\infty$ in the sense that
\[
\|f - f_\infty\|_{W^{2,p}(S^2)} < \sigma,
\]
where $p$ is restricted to the range $1 < p \leq 2$, and $f_\infty$ is assumed to be \emph{integrable} in the sense of \cite[Definitions 4.3 or 4.4 and Proposition 4.1]{KwonThesis}. Her \cite[Proposition 4.1]{KwonThesis} quotes results of Simon \cite[pp. 270--272]{Simon_1985} and Adams and Simon \cite{Adams_Simon_1988}. The \cite[Lemma 3.3]{Liu_Yang_2010} due to Liu and Yang is another example of an optimal {\L}ojasiewicz--Simon gradient inequality for the harmonic map energy function, but restricted to the setting of maps $f:S^2\to N$, where $N$ is a K{\"a}hler manifold of complex dimension $n \geq 1$ and nonnegative bisectional curvature, and the energy $\sE(f)$ is sufficiently small. The result of Liu and Yang generalizes that of Topping \cite[Lemma 1]{Topping_1997}, who assumes that $N = S^2$.
\end{rmk}

\begin{rmk}[Yamabe function for Riemannian metrics on a closed manifold]
For the Yamabe function, an optimal {\L}ojasiewicz--Simon gradient inequality, has been obtained by Carlotto, Chodosh, and Rubinstein \cite{Carlotto_Chodosh_Rubinstein_2015} under the hypothesis that the critical point is \emph{integrable} in the sense of their
\cite[Definition 8]{Carlotto_Chodosh_Rubinstein_2015}, a condition that they observe in \cite[Lemma 9]{Carlotto_Chodosh_Rubinstein_2015} (quoting \cite[Lemma 1]{Adams_Simon_1988} due to Adams and Simon) is equivalent to a function on Euclidean space given by the \emph{Lyapunov-Schmidt reduction} of $\sE$ being constant on an open neighborhood of the critical point.
\end{rmk}

\begin{rmk}[Yang--Mills energy function over a Riemann surface]
For the Yang--Mills energy function for connections on a principal $\U(n)$-bundle over a closed Riemann surface, an optimal {\L}ojasiewicz--Simon gradient inequality, has been obtained by R\r{a}de \cite[Proposition 7.2]{Rade_1992} when the Yang--Mills connection is \emph{irreducible}.
\end{rmk}

\begin{rmk}[$F$-function on the space of hypersurfaces in Euclidean space]
Colding and Minicozzi \cite{Colding_Minicozzi_2014sdg, Colding_Minicozzi_2015} have directly proved {\L}ojasiewicz--Simon gradient and distance inequalities
\cite[Equations (5.9) and (5.10)]{Colding_Minicozzi_Pedersen_2015bams} that do not involve Lyapunov-Schmidt reduction to a finite-dimensional gradient inequality. Their gradient inequality applies to the $F$ function \cite[Section 2.4]{Colding_Minicozzi_Pedersen_2015bams} on the space of hypersurfaces $\Sigma \subset \RR^{d+1}$ and is analogous to \eqref{eq:Lojasiewicz-Simon_gradient_inequality} with $\theta=2/3$. Their cited articles contain detailed technical statements of their inequalities while their article with Pedersen \cite{Colding_Minicozzi_Pedersen_2015bams} contains a less technical summary of some of their main results.
\end{rmk}

If $\sG = \sX$ and $\sH = \sY$, then the statement of Theorem \ref{mainthm:Lojasiewicz-Simon_gradient_inequality_Morse-Bott} simplifies to give the following generalization of \cite[Theorems 2 and 4]{Feehan_Maridakis_Lojasiewicz-Simon_Banach}.

\begin{maincor}[{\L}ojasiewicz--Simon gradient inequality for $C^2$ Morse--Bott functions on Banach spaces]
\label{maincor:Lojasiewicz-Simon_gradient_inequality_Morse-Bott_GisX}
(Compare Feehan and Maridakis \cite[Theorems 2 and 4]{Feehan_Maridakis_Lojasiewicz-Simon_Banach}.)
Let $\sX$ and $\sY$ be Banach spaces with a continuous embedding $\sY \subset \sX^*$. Let $\sU \subset \sX$ be an open subset, $\sE:\sU\to\RR$ be a $C^2$ function, and $x_\infty\in\sU$ be a critical point of $\sE$, so $\sE'(x_\infty) = 0$. Let $\sM:\sU\to\sY$ be a $C^1$ gradient map for $\sE$ in the sense of Definition \ref{defn:Huang_2-1-1} and require that $\sE$ be Morse--Bott at $x_\infty$ in the sense of Definition \ref{defn:Morse-Bott_function}, so $\sU\cap\Crit\sE$ is a relatively open, smooth submanifold of $\sX$ and $K := \Ker\sE''(x_\infty) = T_{x_\infty}\Crit\sE$. Assume that $K\subset\sX$ has a closed complement $\sX_0 \subset \sX$, and that $\Ran\sM'(x_\infty) \subset \sY$ is a closed subspace. Then there are constants $Z \in (0,\infty)$ and $\sigma \in (0,1]$ with the following significance. If $x \in \sU$ obeys
\begin{equation}
\label{eq:Lojasiewicz-Simon_gradient_inequality_neighborhood_Morse-Bott_GisX}
\|x-x_\infty\|_\sX < \sigma,
\end{equation}
then
\begin{equation}
\label{eq:Lojasiewicz-Simon_gradient_inequality_Morse-Bott_GisX}
\|\sM(x)\|_{\sY} \geq Z|\sE(x) - \sE(x_\infty)|^{1/2}.
\end{equation}
\end{maincor}

For example, Corollary \ref{maincor:Lojasiewicz-Simon_gradient_inequality_Morse-Bott_GisX} yields a version of Simon's \cite[Theorem 3]{Simon_1983} when $X$ has dimension $d=2$ or $3$ and choose $\sX = W^{1,p}(X;V)$ and $\sY = W^{-1,p}(X;V)$, where $p>d$ is small enough that $L^2(X;V) \subset W^{-1,p}(X;V)$; see \cite[Remark 1.15]{Feehan_Maridakis_Lojasiewicz-Simon_Banach}.

If in addition $\sY = \sX^*$, then the statement of Theorem \ref{mainthm:Lojasiewicz-Simon_gradient_inequality_Morse-Bott} simplifies further to give the following generalization of \cite[Theorems 1 and 4]{Feehan_Maridakis_Lojasiewicz-Simon_Banach}.

\begin{maincor}[{\L}ojasiewicz--Simon gradient inequality for $C^2$ Morse--Bott functions on Banach spaces]
\label{maincor:Lojasiewicz-Simon_gradient_inequality_Morse-Bott_YisXdual}
(Compare Feehan and Maridakis \cite[Theorems 1 and 4]{Feehan_Maridakis_Lojasiewicz-Simon_Banach}.)
Let $\sX$ be a Banach space, $\sU \subset \sX$ be an open subset, $\sE:\sU\to\RR$ be a $C^2$ function, and $x_\infty\in\sU$ be a critical point of $\sE$, so $\sE'(x_\infty) = 0$. Require that $\sE$ be Morse--Bott at $x_\infty$ in the sense of Definition \ref{defn:Morse-Bott_function}, so $\sU\cap\Crit\sE$ is a relatively open, smooth submanifold of $\sX$ and $K := \Ker\sE''(x_\infty) = T_{x_\infty}\Crit\sE$. Assume that $K\subset\sX$ has a closed complement, $\sX_0 \subset \sX$, and that $\Ran\sE''(x_\infty) \subset \sX^*$ is a closed subspace. Then there are constants, $Z \in (0,\infty)$ and $\sigma \in (0,1]$, with the following significance. If $x \in \sU$ obeys
\begin{equation}
\label{eq:Lojasiewicz-Simon_gradient_inequality_neighborhood_Morse-Bott_YisXdual}
\|x-x_\infty\|_\sX < \sigma,
\end{equation}
then
\begin{equation}
\label{eq:Lojasiewicz-Simon_gradient_inequality_Morse-Bott_YisXdual}
\|\sE'(x)\|_{\sX^*} \geq Z|\sE(x) - \sE(x_\infty)|^{1/2}.
\end{equation}
\end{maincor}

For example, Corollary \ref{maincor:Lojasiewicz-Simon_gradient_inequality_Morse-Bott_GisX} yields R\r{a}de's {\L}ojasiewicz-Simon gradient inequality for the Yang--Mills energy function when the base manifold has dimension $d=2$ or $3$ and our version of the same inequality \cite[Theorem 23.17]{Feehan_yang_mills_gradient_flow_v4} when $d=4$, for $\sX = W^{1,2}(X;\ad P)$, but not $d\geq 5$, nor does it yield any version of Simon's \cite[Theorem 3]{Simon_1983}.

\subsection{Optimal {\L}ojasiewicz--Simon inequalities and Morse--Bott properties for the self-dual Yang--Mills energy function near anti-self-dual connections}
\label{subsec:Optimal_Lojasiewicz-Simon_inequalities_self-dual_Yang-Mills_energy_Morse-Bott}
We define the \emph{Yang--Mills-energy function} by \cite[p. 548]{Atiyah_Bott_1983}
\begin{equation}
\label{eq:Yang-Mills_energy_function}
\YM(A)  := \frac{1}{2}\int_X |F_A|^2\,d\vol_g,
\end{equation}
where $A$ is a $W^{1,q}$ connection on $P$ and curvature \cite[Equation (2.1.13)]{DK},
\[
F_A = d_A\circ d_A \in L^2(X; \wedge^2(T^*X)\otimes\ad P),
\]
where $q \geq \max\{2, 4d/(d+4)\}$. Writing $A = A_1+a$, for any $C^\infty$ connection $A_1$ on $P$, we have \cite[Equation (2.1.14)]{DK}
\begin{equation}
\label{eq:Curvature_FA_as FA1+a}
F_A = F_{A_1} + d_{A_1}a + a\wedge a.
\end{equation}
The constraint $q\geq 2$ ensures that $d_{A_1}a \in L^2(X; \wedge^2(T^*X)\otimes\ad P)$ and the constraint $q \geq 4d/(d+4)$ is equivalent to $q^* := dq/(d-q) \geq 4$ and thus $W^{1,q}(X;\RR) \subset L^4(X;\RR)$ when $q<d$ by \cite[Theorem 4.12, Part I (C)]{AdamsFournier}. Hence, $a \in L^4(X; T^*X\otimes\ad P)$ and $a\wedge a \in L^2(X; \wedge^2(T^*X)\otimes\ad P)$, which gives $F_A \in L^2(X; \wedge^2(T^*X)\otimes\ad P)$, as desired. Note that $d/2 \geq 4d/(d+4) \iff d \geq 4$ and $4d/(d+4) < 2$ only when $d=2,3$.

In order to ensure that the energy $\YM(A)$ in \eqref{eq:Yang-Mills_energy_function} is well-defined for a $W^{1,q}$ connection $A$ and that the action of gauge transformations on $P$ is also well-defined, we shall assume for consistency and simplicity throughout this article that $q \in [2,\infty)$ and obeys $q>d/2$ in this context, even though that condition may be stronger than necessary in some instances.

In writing \eqref{eq:Curvature_FA_as FA1+a}, we are slightly abusing notation since in the setting of \cite[Section 2.1]{DK}, for example, a representation, $\rho:G \hookrightarrow \End_\CC(\CC^n)$, is assumed and $a\wedge a$ denotes a combination of wedge product of one-forms $a \in \Omega^1(X;\End_\CC(E))$ and multiplication in $\End_\CC(E)$, where $E$ is the complex vector bundle, $P\times_\rho \CC^n$. Since we view $a \in \Omega^1(X;\ad P)$ and $F_A \in \Omega^2(X;\ad P)$ (as in \cite{AHS}) rather than $F_A \in \Omega^2(X;\End_\CC(E))$ (as in \cite{DK}), we should more precisely write (see the parenthetical remark just below \cite[Equation (2.1.14)]{DK} or \cite[Lemma 4.5]{AHS})
\begin{equation}
\label{eq:Curvature_FA_as FA1+a_bracket}
F_A = F_{A_1} + d_{A_1}a + \frac{1}{2}[a,a],
\end{equation}
where $[a,a](\eta,\zeta) := [a(\eta),a(\zeta)]$ for vector fields $\eta, \zeta \in C^\infty(TX)$ and $[\,,\,]$ denotes the Lie bracket on the Lie algebra $\fg$ of $G$. Compare \cite[Theorem II.5.2]{Kobayashi_Nomizu_v1} or \cite[p. 430]{AHS}. On the other hand, for $a,b \in \Omega^1(X;\ad P)$, the exterior covariant derivative $d_Ab$ is expressed in terms of $d_{A_1}b$ when $A=A_1+a$ by (see \cite[Sections 2.1.1 and 2.1.2]{DK} or \cite[Equations (3.3) and (4.1)]{AHS})
\begin{equation}
\label{eq:Exterior_covariant_derivative_dA_as dA1+a_bracket}
d_Ab = d_{A_1+a}b = d_{A_1}b + [a,b] = d_{A_1}b + 2 a\wedge b.
\end{equation}
Normally, these factors of $\frac{1}{2}$ or $2$ are immaterial and in such cases we abuse notation and omit them.

For $q \in [2,\infty)$ obeying $q>d/2$, let $\sB(P) := \sA(P)/\Aut(P)$ denote the quotient of the affine space $\sA(P)$ of $W^{1,q}$ connections on $P$, modulo the action of the group $\Aut(P)$ of $W^{2,q}$ automorphisms (or gauge transformations) of the principal $G$-bundle, $P$. We refer the reader to Donaldson and Kronheimer \cite[Section 4.2]{DK} or Freed and Uhlenbeck \cite[Chapter 3]{FU} for constructions of a smooth Banach manifold structure on the quotient, $\sB^*(P) := \sA^*(P)/\Aut(P)$, where $\sA^*(P)\subset \sA(P)$ is by definition the open subset consisting of $W^{1,q}$ connections on $P$ whose isotropy group is minimal, namely the center of $G$ \cite[p. 132]{DK}. Let
\begin{equation}
\label{eq:Moduli_space_flat_connections}
M_0(P) := \{A \in \sA(P): F_A = 0\}/\Aut(P)
\end{equation}
denote the moduli space of $W^{1,q}$ \emph{flat connections} on $P$. We write
\begin{equation}
\label{eq:W12_distance_to_moduli_space_flat_connections}
\dist_{W^{1,2}(X)} \left([A], M_0(P)\right)
:=
\inf_{\begin{subarray}{c} u \in \Aut(P), \\ [\Gamma] \in M_0(P) \end{subarray}}
\|u(A) - \Gamma\|_{W_\Gamma^{1,2}(X)}.
\end{equation}
Recall that the Yang--Mills energy function, $\YM:\sA(P)\to\RR$, in \eqref{eq:Yang-Mills_energy_function} has differential map, $\YM':\sA(P)\to T^*\sA(P)$, given by
\begin{equation}
\label{eq:Gradient_Yang-Mills_energy_function}
\YM'(A)(a) = (F_A,d_Aa)_{L^2(X)}= (d_A^*F_A,a)_{L^2(X)},
\end{equation}
for all $A \in \sA(P)$ and $a \in T_A\sA(P) = W^{1,q}(X;T^*X\otimes\ad P)$, where $T_A^*\sA(P) \cong W^{-1,q'}(X;T^*X\otimes\ad P)$ and $q' \in (1,2]$ is the dual H\"older exponent defined by $1/q+1/q'=1$.

We temporarily now restrict our attention to the case of $X$ of dimension $d=4$. For a $C^\infty$ connection, $A$, on $P$ we recall the splitting \cite[Equation (2.1.25)]{DK},
\begin{equation}
\label{eq:Donaldson_Kronheimer_2-1-25}
F_A = F_A^+ + F_A^- \in \Omega^2(X;\ad P) = \Omega^+(X;\ad P) \oplus \Omega^-(X;\ad P),
\end{equation}
corresponding to the splitting, $\wedge^2(T^*X) = \wedge^+(T^*X)\oplus\wedge^-(T^*X)$, into positive and negative eigenspaces, $\wedge^\pm(T^*X)$, of the Hodge star operator $*$ on $\wedge^2(T^*X)$, defined by the metric $g$, so $\Omega^\pm(X; \ad P) = C^\infty(X; \wedge^\pm(T^*X)\otimes\ad P)$ and \cite[Equation (1.3)]{TauSelfDual}
\begin{equation}
\label{eq:Taubes_1982_1-3}
F_A^\pm = \frac{1}{2}(1 \pm *)F_A \in \Omega^\pm(X; \ad P).
\end{equation}
Rather than consider the full Yang--Mills energy function \eqref{eq:Yang-Mills_energy_function}, for which it appears difficult to show has the Morse--Bott property at critical points that are not flat connections, we shall consider the \emph{self-dual Yang--Mills energy function}, $\YM_+:\sA(P) \to \RR$, on the affine space of $W^{1,q}$ connections $A$ on $P$ (with $q\geq 2$),
\begin{equation}
\label{eq:Self-dual_Yang-Mills_energy_function}
\YM_+(A)  := \frac{1}{2}\int_X |F_A^+|^2\,d\vol_g.
\end{equation}
Our definition \eqref{eq:Self-dual_Yang-Mills_energy_function} is partly motivated by the fact that when, for example, $G=\SU(n)$ and the second Chern number of $P$ is non-negative, $c_2(P)[X] \geq 0$, the energy function $\YM:\sA(P) \to \RR$, achieves its absolute minimum value at a connection $A$ if and only if $A$ is \emph{anti-self-dual}, so $F_A^+ = 0$, and $\YM(A)=16\pi^2c_2(P)[X]$, a constant that depends only on the topology of the principal $G$-bundle $P$; see \cite[Equation (2.1.33)]{DK} for $G=\SU(n)$ and \cite[Section 10]{Feehan_yang_mills_gradient_flow_v4} for more general formulae for the energies of anti-self-dual connections in the case of compact Lie groups. Our definition \eqref{eq:Self-dual_Yang-Mills_energy_function} of $\YM_+$ effectively subtracts this topological constant from $\YM$ in \eqref{eq:Yang-Mills_energy_function}.

Proceeding as in the case of the full Yang--Mills energy function, we see that $\YM_+:\sA(P) \to \RR$ has differential map, $\YM_+':\sA(P) \to T^*\sA(P)$, given by
\begin{equation}
\label{eq:Gradient_self-dual_Yang-Mills_energy_function}
\YM_+'(A)(a) = (F_A^+,d_A^+a)_{L^2(X)}= (d_A^{+,*}F_A^+,a)_{L^2(X)},
\end{equation}
for all $a \in T_A\sA(P) = W^{1,q}(X;T^*X\otimes\ad P)$.

We denote the finite-dimensional subvariety of gauge-equivalence classes of solutions to the \emph{anti-self-dual} equation with respect to $g$ by
\begin{equation}
\label{eq:Moduli_space_anti-self-dual_connections}
M_+(P,g) := \left\{[A] \in \sB(P): F_A^+ = 0 \quad\text{a.e. on } X\right\}.
\end{equation}
As usual \cite[Section 2.3.1]{DK}, one denotes $d_A^+ = \frac{1}{2}(1+*)d_A:\Omega^1(X; \ad P)\to \Omega^+(X; \ad P)$ and $H_A^2 = \Coker d_A^+$ \cite[Equation (4.2.27)]{DK}. We recall from \cite[Section 4.2.5]{DK} that if $H_A^2 = 0$ then
\[
\widetilde M_+(P,g) := \{B \in \sA(P): F_B^+ = 0 \quad\text{a.e. on } X\}
\]
is a smooth manifold near $A$ and
\[
M_+^*(P,g) := M_+(P,g)\cap \sB^*(P)
\]
is a smooth manifold near $[A]$. The Generic Metrics Theorem \cite[Corollary 4.3.18]{DK} due to Freed and Uhlenbeck implies that $H_A^2 = 0$ for all $[A] \in M_+^*(P,g)$ if $G=\SU(2)$ or $\SO(3)$ and $g$ is suitably generic.

If $F_A^+=0$, then $\YM_+'(A) \equiv 0$ by \eqref{eq:Gradient_self-dual_Yang-Mills_energy_function} and $A$ is a critical point of $\YM_+:\sA(P) \to \RR$, so that
\[
\widetilde M_+(P,g) \subset \widetilde{\Crit}\YM_+ \cap \sA(P),
\]
where $\widetilde{\Crit}\YM_+$ denotes the critical set of $\YM_+:\sA(P) \to \RR$. Conversely, if $A \in \widetilde{\Crit}\YM_+$ and $\Coker d_A^+ = 0$ then \eqref{eq:Gradient_self-dual_Yang-Mills_energy_function} implies that $F_A^+=0$ and $A \in \widetilde M_+(P,g)$.

By gauge invariance, the self-dual Yang--Mills energy function is well-defined on the quotient, $\YM_+:\sB^*(P) \to \RR$ (with $q>2$), and we have the inclusion,
\[
M_+^*(P,g) \subset \Crit\YM_+ \cap \sB^*(P),
\]
where $\Crit\YM_+$ denotes the critical set of $\YM_+:\sB^*(P) \to \RR$. Conversely, if $[A] \in \Crit\YM_+$ and $\Coker d_A^+ = 0$, then $[A] \in M_+^*(P,g)$.

We have the following analogue of \cite[Theorem 4]{Feehan_Maridakis_Lojasiewicz-Simon_coupled_Yang-Mills_v6} for the coupled boson Yang--Mills energy function, but with the improvement that $\theta=1/2$, the optimal {\L}ojasiewicz--Simon exponent.

\begin{mainthm}[Optimal {\L}ojasiewicz--Simon inequalities for the self-dual Yang--Mills energy function]
\label{mainthm:Optimal_Lojasiewicz-Simon_inequalities_self-dual_Yang-Mills_energy_function}
Let $(X,g)$ be a closed, four-dimensional, smooth Riemannian manifold, $G$ be a compact Lie group, $P$ be a smooth principal $G$-bundle over $X$, and $q>2$ be a constant. If $A_\infty$ is a $W^{1,q}$ anti-self-dual Yang--Mills connection on $P$ that is \emph{regular},
\begin{equation}
\label{eq:Regular_ASD_connection}
H_{A_\infty}^2
:=
\Coker \left( d_{A_\infty}^+: W^{1,q}(X;T^*X\otimes\ad P) \to L^q(X;\wedge^+(T^*X)\otimes\ad P)\right) = 0,
\end{equation}
then $\YM_+:\sA(P) \to \RR$ is a Morse--Bott function at $A_\infty$ and there are constants $C, Z \in (0, \infty)$ and $\sigma \in (0,1]$, depending on $A_\infty$, $g$, and $G$, with the following significance. If $A$ is a $W^{1,q}$ connection on $P$ obeying the \emph{{\L}ojasiewicz--Simon neighborhood} condition,
\begin{equation}
\label{eq:Lojasiewicz-Simon_gradient_inequality_self-dual_Yang-Mills_energy_L4_neighborhood}
\|A - A_\infty\|_{L^4(X)} < \sigma,
\end{equation}
then the self-dual Yang--Mills energy function \eqref{eq:Self-dual_Yang-Mills_energy_function} obeys the \emph{optimal {\L}ojasiewicz--Simon distance and gradient inequalities},
\begin{align}
\label{eq:Lojasiewicz-Simon_distance_inequality_self-dual_Yang-Mills_energy}
\YM_+(A)^{1/2} &\geq C\|A - A_\infty\|_{W_{A_\infty}^{1,2}(X)},
\\
\label{eq:Lojasiewicz-Simon_gradient_inequality_self-dual_Yang-Mills_energy}
\|\YM_+'(A)\|_{W_{A_\infty}^{-1,2}(X)}
&\geq
Z|\YM_+(A)|^{1/2}.
\end{align}
Moreover, if the isotropy group of $A_\infty$ in $\Aut(P)$ is minimal (the center of $G$), then $\YM_+:\sB^*(P) \to \RR$ is a Morse--Bott function at $[A_\infty]$.
\end{mainthm}

\begin{rmk}[Extension of the Morse--Bott property for Yang--Mills energy functions]
\label{rmk:Morse-Bott_quotient_space}
There are ways of defining the Morse--Bott property of the self-dual Yang--Mills energy function \eqref{eq:Self-dual_Yang-Mills_energy_function} when passing to the quotient by $\Aut(P)$ that are less restrictive than that implied by Theorem \ref{mainthm:Optimal_Lojasiewicz-Simon_inequalities_self-dual_Yang-Mills_energy_function}. For example, one could fix a base point $x_0\in X$ and consider the quotient space $\sB(P,x_0) = (\sA(P)\times P|_{x_0})/\Aut(P)$ of based connections, as in Atiyah and Bott \cite{Atiyah_Bott_1983} or Taubes \cite[p. 328]{TauFrame}, and use the fact that $\sB(P,x_0)$ is necessarily a Banach manifold, since $\Aut(P)$ acts freely in this quotient. Another approach (see Feehan \cite[Definition 7.6]{Feehan_nonlinear_uhlenbeck_estimate}) is to consider the restriction of the self-dual Yang--Mills energy function \eqref{eq:Self-dual_Yang-Mills_energy_function} to a Coulomb-gauge slice $A_\infty+\Ker d_{A_\infty}^*\cap W^{1,q}(X;T^*X\otimes\ad P)$ through $A_\infty$. Similar remarks apply to the full Yang--Mills energy function \eqref{eq:Yang-Mills_energy_function} in Theorem \ref{mainthm:Lojasiewicz-Simon_inequalities_Yang-Mills_energy_flat_Morse-Bott} (near flat connections) and Theorem \ref{mainthm:Lojasiewicz-Simon_inequalities_Yang-Mills_energy_Morse-Bott} (near Yang--Mills connections).
\end{rmk}

We prove Theorem \ref{mainthm:Optimal_Lojasiewicz-Simon_inequalities_self-dual_Yang-Mills_energy_function} in Section \ref{sec:Morse-Bott_property_Yang-Mills_energy_functions}.

\begin{rmk}[Replacement of $W^{-1,2}$ by $L^2$ norm in the gradient inequality]
\label{rmk:Replacement_W_minus_12_by_L2_norm_gradient_inequality}  
Because $W_\Gamma^{1,2}(X;T^*X\otimes \ad P) \subset L^2(X;T^*X\otimes \ad P)$ is a continuous, dense embedding of Sobolev spaces (for any $d \geq 2$), we obtain a continuous embedding of Sobolev spaces,
\[
L^2(X;T^*X\otimes \ad P) \subset W_\Gamma^{-1,2}(X;T^*X\otimes \ad P),
\]
by duality and so the $W^{-1,2}$ norm in \eqref{eq:Lojasiewicz-Simon_gradient_inequality_self-dual_Yang-Mills_energy} can be replaced by the stronger $L^2$ norm, as convenient in the analysis of the gradient flow equation for $\YM$ \cite{Feehan_yang_mills_gradient_flow_v4}.
\end{rmk}

\begin{rmk}[Small self-dual Yang--Mills energy hypothesis]
\label{rmk:Small_self-dual_Yang--Mills_energy_hypothesis}
A more sophisticated (and more difficult) analysis would allow us to replace the small $L^4$ distance hypothesis \eqref{eq:Lojasiewicz-Simon_gradient_inequality_self-dual_Yang-Mills_energy_L4_neighborhood} by a small self-dual Yang--Mills energy hypothesis,
\[
\YM_+(A) < \eps,
\]
for a fixed constant $\eps = \eps(g,G) \in (0,1]$,
and replace the inequality \eqref{eq:Lojasiewicz-Simon_distance_inequality_self-dual_Yang-Mills_energy} by
\[
\YM_+(A)^{1/2} \geq C\dist_{W^{1,2}}\left([A],M_+(P,g)\right),
\]
both of which are more appropriate for Morse--Bott theory. We hope to describe this refinement elsewhere.
\end{rmk}

\begin{rmk}[Two approaches to the proof of the optimal {\L}ojasiewicz--Simon distance and gradient inequality for the self-dual Yang--Mills energy function near a regular anti-self-dual connection]
\label{rmk:Two_proofs_optimal_LS gradient_inequality_self-dual_Yang-Mills_energy}
Theorem \ref{mainthm:Optimal_Lojasiewicz-Simon_inequalities_self-dual_Yang-Mills_energy_function} is proved in Section \ref{subsec:Self-dual_Yang-Mills_energy_function_near_anti-self-dual_connections}.
As we explain there, the gradient inequality \eqref{eq:Lojasiewicz-Simon_gradient_inequality_self-dual_Yang-Mills_energy} may be proved in two different ways:
\begin{inparaenum}[\itshape a\upshape)]
\item by direct geometric analysis using methods of Yang--Mills gauge theory, and
\item by first establishing that $\YM_+$ is Morse--Bott at an anti-self-dual connection that is regular in the sense of \eqref{eq:Regular_ASD_connection} (see Lemma \ref{lem:Morse-Bott_property_self-dual_Yang-Mills_energy}) and then appealing to our Theorem \ref{mainthm:Lojasiewicz-Simon_gradient_inequality_Morse-Bott}, giving the optimal gradient inequality for an abstract Morse--Bott function on a Banach space.
\end{inparaenum}
\end{rmk}

\subsection{Optimal {\L}ojasiewicz--Simon inequalities and Morse--Bott properties for the Yang--Mills energy function near flat connections}
\label{subsec:Optimal_Lojasiewicz-Simon_inequalities_Yang-Mills_energy_Morse-Bott_flat}
We return to the case where $X$ is a manifold of arbitrary dimension $d \geq 2$.

\begin{mainthm}[Optimal {\L}ojasiewicz--Simon inequalities and Morse--Bott properties for the Yang--Mills energy function near regular flat connections]
\label{mainthm:Lojasiewicz-Simon_inequalities_Yang-Mills_energy_flat_Morse-Bott}
Let $(X,g)$ be a closed, smooth Riemannian manifold of dimension $d \geq 2$, and $G$ be a compact Lie group, $P$ be a smooth principal $G$-bundle over $X$, and $q \in [2,\infty)$ obeying $q>d/2$ and $r_0>2$ be constants. If $\Gamma$ is a $W^{1,q}$ flat connection on $P$ that is \emph{regular} in the sense that
\begin{equation}
\label{eq:Regular_flat_connection}
H_\Gamma^2(X;\ad P)
:=
\frac{\Ker \left( d_\Gamma: L^q(X;\wedge^2(T^*X)\otimes\ad P) \to W^{-1,q}(X;\wedge^3(T^*X)\otimes\ad P)\right)}
{\Ran \left( d_\Gamma: W^{1,q}(X;T^*X\otimes\ad P) \to L^q(X;\wedge^2(T^*X)\otimes\ad P)\right)}
= 0,
\end{equation}
then $\YM:\sA(P) \to \RR$ is a Morse--Bott function at $\Gamma$ and there are constants $C, Z \in (0, \infty)$ and $\sigma \in (0,1]$, depending on $\Gamma$, $g$, $G$, and $r_0$ with the following significance. If $A$ is a $W^{1,q}$ connection on $P$ obeying the \emph{{\L}ojasiewicz--Simon neighborhood} condition,
\begin{equation}
\label{eq:Lojasiewicz-Simon_gradient_inequality_Yang-Mills_energy_flat_nbhd}
\|A - \Gamma\|_{L^{r_0}(X)} < \sigma,
\end{equation}
where $r_0=d$ when $d\geq 3$ and $r_0>2$ when $d=2$, then the Yang--Mills energy function \eqref{eq:Yang-Mills_energy_function} obeys the \emph{optimal {\L}ojasiewicz--Simon distance and gradient inequalities},
\begin{align}
\label{eq:Lojasiewicz-Simon_distance_inequality_Yang-Mills_energy_flat_local}
\YM(A)^{1/2} &\geq C\|A-\Gamma\|_{W^{1,2}(X)},
\\
\label{eq:Lojasiewicz-Simon_gradient_inequality_Yang-Mills_energy_flat_local}
\|\YM'(A)\|_{W^{-1,2}(X)}
&\geq
Z|\YM(A)|^{1/2}.
\end{align}
Moreover, if the isotropy group of $\Gamma$ in $\Aut(P)$ is minimal (the center of $G$), then $\YM:\sB^*(P) \to \RR$ is a Morse--Bott function at $[\Gamma]$.
\end{mainthm}

We prove Theorem \ref{mainthm:Lojasiewicz-Simon_inequalities_Yang-Mills_energy_flat_Morse-Bott} in Section \ref{sec:Morse-Bott_property_Yang-Mills_energy_functions}.

When $X$ has dimension two, then Poincar\'e duality (for example, see \cite[Lemma 2.1]{Ho_Wilkin_Wu_2019}) implies that
$H_\Gamma^2(X;\ad P) \cong H_\Gamma^0(X;\ad P)$ (this observation is used in \cite[p. 189]{MMR}), where
\[
H_\Gamma^0(X;\ad P)
:=
\Ker \left( d_\Gamma: W^{2,q}(X;\ad P) \to W^{1,q}(X;T^*X\otimes\ad P)\right).
\]
Recall \cite[p. 132]{DK} that $H_\Gamma^0(X;\ad P)$ is isomorphic to the tangent space to the isotropy group of $\Gamma$ in $\Aut(P)$. In the special case that $G = \SU(2)$ or $\SO(3)$, then a connection $A$ on $P$ is \emph{irreducible} if and only if the isotropy of $A$ in $\Aut(P)$ is the center of $G$ \cite[p. 133]{DK}. Thus, if $G = \SU(2)$ (with center $\ZZ/2\ZZ$) and $\dim X = 2$ (with $\genus(X) \geq 1$) and $\Gamma$ is an irreducible flat connection, then $H_\Gamma^0(X;\ad P) = 0$ and consequently $H_\Gamma^2(X;\ad P) = 0$, so $[\Gamma]$ is a smooth point of $M_0(P)$. In particular, the moduli space of gauge equivalence classes of irreducible flat connections on $P$, namely $M_0^*(P) := M_0(P)\cap\sB^*(P)$, is a smooth manifold (of dimension $6\genus(X)-6$ \cite{Sengupta_1998jgp}) and therefore $\YM:\sA^*(P)\to\RR$ is a Morse--Bott function near the critical set,
\[
\widetilde M_0^*(P) := \{A \in \sA^*(P): F_A = 0\},
\]
and $\YM:\sB^*(P)\to\RR$ is a Morse--Bott function near the critical set $M_0^*(P)$, in the sense of Definition \ref{defn:Morse-Bott_function} for both cases. When $X$ is a Riemann surface, there is a vast literature devoted to the study of $M_0(P)$ from many different perspectives, often in the context of its interpretation as the character variety, $\Hom(\pi_1(X),G)/G$, of representations of the fundamental group $\pi_1(X)$ in $G$ \cite[Proposition 2.2.3]{DK} or in the context of the symplectic structure on $\sA(P)$ and interpretation of (a multiple of) the map $A\mapsto F_A$ as a moment map. We refer to Section \ref{subsec:Yang-Mills_energy_function_near_yang-mills_connections_over_Riemann_surfaces} and Atiyah and Bott \cite{Atiyah_Bott_1983} and the many articles that cite \cite{Atiyah_Bott_1983} for further details and discussions of the Morse--Bott properties of $\YM$ over Riemann surfaces from a variety of perspectives.

When $X$ has dimension three and is a circle bundle over a closed Riemann surface, the geometry of $\Hom(\pi_1(X),\SU(2))/\SU(2)$ is described by Morgan, Mrowka, and Ruberman in \cite[Chapter 13]{MMR}.

\subsection{Optimal {\L}ojasiewicz--Simon inequalities and Morse--Bott properties for the Yang--Mills energy function near arbitrary Yang--Mills connections}
\label{subsec:Optimal_Lojasiewicz-Simon_inequalities_Yang-Mills_energy_Morse-Bott}
One calls $A$ a \emph{Yang-Mills connection} if it is a critical point of the Yang--Mills energy function \eqref{eq:Yang-Mills_energy_function} on the affine space of $W^{1,q}$ connections, $\YM:\sA(P)\to\RR$ on $P$, that is, $\YM'(A) = 0$ and so by \eqref{eq:Gradient_Yang-Mills_energy_function}, obeys
\begin{equation}
\label{eq:Yang_Mills}
d_A^*F_A = 0
\end{equation}
in a sense that depends on the regularity of $A$, weakly if $A$ is $W^{1,q}$ or strongly if $A$ is $W^{2,q}$, with $q \in [2,\infty)$ obeying $q > d/2$. We define
\begin{equation}
\label{:Yang-Mills_energy_function_critical_set}
\Crit\YM := \{A \in \sA(P): \YM'(A) = 0\}.
\end{equation}
The Yang--Mills energy function, $\YM:\sA(P)\to\RR$, in \eqref{eq:Yang-Mills_energy_function} has the Hessian operator, $\YM(A):T_A\sA(P) \to T_A^*\sA(P)$, at $A \in \sA(P)$ given by
\begin{equation}
\label{eq:Hessian_Yang-Mills_energy_function}
\YM''(A)(a)b = (d_Aa,d_Ab)_{L^2(X)} + (F_A,a\wedge b)_{L^2(X)},
\end{equation}
for all $a, b \in W^{1,q}(X;T^*X\otimes\ad P)$. We now state a partial generalization of Theorem \ref{mainthm:Lojasiewicz-Simon_inequalities_Yang-Mills_energy_flat_Morse-Bott}.

\begin{mainthm}[Optimal {\L}ojasiewicz--Simon inequalities when the Yang--Mills energy function is Morse--Bott near a Yang--Mills connection]
\label{mainthm:Lojasiewicz-Simon_inequalities_Yang-Mills_energy_Morse-Bott}
Let $(X,g)$ be a closed, smooth Riemannian manifold of dimension $d\geq 2$ and $G$ be a compact Lie group, $P$ be a smooth principal $G$-bundle over $X$, and $q \in [2,\infty)$ obeying $q>d/2$ be a constant. Let $A_\infty$ be a $W^{1,q}$ Yang--Mills connection on $P$ and assume that $\YM:\sA(P) \to \RR$ is a Morse--Bott function at $A_\infty$ in the sense of Definition \ref{defn:Morse-Bott_function}, so $\Crit\YM\cap\, \sU_{A_\infty}$ is a $C^\infty$ relatively open submanifold of the space $\sA(P)$ of $W^{1,q}$ connections on $P$ for some open neighborhood $\sU_{A_\infty}$ of $A_\infty$ and
\[
\Ker\YM''(A_\infty) \cap W_{A_\infty}^{1,q}(X;T^*X\otimes \ad P) = T_{A_\infty}\Crit\YM.
\]
Then there are constants $Z \in (0, \infty)$ and $\sigma \in (0,1]$, depending on $A_\infty$, $g$, and $G$ with the following significance. If $A$ is a $W^{1,q}$ connection on $P$ that obeys the \emph{{\L}ojasiewicz--Simon neighborhood} condition,
\begin{equation}
\label{eq:Lojasiewicz-Simon_gradient_inequality_Yang-Mills_energy_nbhd}
\|A - A_\infty\|_{W^{1,2}_{A_\infty}(X)} < \sigma,
\end{equation}
then the Yang--Mills energy \eqref{eq:Yang-Mills_energy_function} obeys the \emph{optimal {\L}ojasiewicz--Simon gradient inequality},
\begin{equation}
\label{eq:Lojasiewicz-Simon_gradient_inequality_Yang-Mills_energy_local}
\|\YM'(A)\|_{W_{A_\infty}^{-1,2}(X)}
\geq
Z|\YM(A)|^{1/2}.
\end{equation}
\end{mainthm}

We prove Theorem \ref{mainthm:Lojasiewicz-Simon_inequalities_Yang-Mills_energy_Morse-Bott} in Section \ref{sec:Morse-Bott_property_Yang-Mills_energy_functions}.

\begin{rmk}[Optimal {\L}ojasiewicz--Simon gradient inequality for the Yang--Mills energy function over Riemann surfaces]
R\r{a}de has shown (see \cite[Proposition 7.2]{Rade_1992} ) that if $d=2$ and $G=\U(n)$ and $A_\infty$ is an irreducible Yang--Mills connection, then inequality \eqref{eq:Lojasiewicz-Simon_gradient_inequality_Yang-Mills_energy_local} in Theorem \ref{mainthm:Lojasiewicz-Simon_inequalities_Yang-Mills_energy_Morse-Bott} holds. His proof (see \cite[Section 10]{Rade_1992}) is very different from our proof of \eqref{eq:Lojasiewicz-Simon_gradient_inequality_Yang-Mills_energy_local} and does not proceed by showing that $\YM$ is Morse--Bott at $A_\infty$. 
\end{rmk}

I am grateful to the referee for pointing out that the following theorem holds, thus answering a question that I had posed in an earlier version of this article. 

\begin{mainthm}[Optimal {\L}ojasiewicz--Simon inequalities and Morse--Bott properties for Yang--Mills energy functions over Riemann surfaces]
\label{mainthm:Lojasiewicz-Simon_inequalities_Yang-Mills_energy_irreducible_Riemann_surface}  
Continue the hypotheses of Theorem \ref{mainthm:Lojasiewicz-Simon_inequalities_Yang-Mills_energy_Morse-Bott}, but assume that $G=\U(n)$ with $n\geq 2$ and $d=2$ and replace the assumption that the Yang--Mills connection $A_\infty$ is Morse--Bott by the assumption that it has trivial isotropy subgroup in $\Aut(P)$. Then $A_\infty$ is Morse--Bott and the remaining conclusions of Theorem \ref{mainthm:Lojasiewicz-Simon_inequalities_Yang-Mills_energy_Morse-Bott} continue to hold.
\end{mainthm}

We shall give two proofs of Theorem \ref{mainthm:Lojasiewicz-Simon_inequalities_Yang-Mills_energy_irreducible_Riemann_surface} in Section \ref{subsec:Yang-Mills_energy_function_near_yang-mills_connections_over_Riemann_surfaces} --- an indirect proof based on R\r{a}de \cite[Proposition 7.2]{Rade_1992} and Feehan \cite[Theorem 2]{Feehan_lojasiewicz_inequality_all_dimensions_morse-bott} and a direct proof due to the referee.

In Theorem \ref{mainthm:Optimal_Lojasiewicz-Simon_inequalities_self-dual_Yang-Mills_energy_function}, we noted that the self-dual Yang--Mills energy function $\YM$ is Morse--Bott at an anti-self-dual connection $A_\infty$ that is a regular point of the map,
\[
W_{A_\infty}^{1,q}(X;T^*X\otimes \ad P) \ni A \mapsto F_A^+ \in L^q(X;\wedge^+(T^*X)\otimes \ad P),
\]
or, equivalently, that the map
\[
W_{A_\infty}^{1,q}(X;T^*X\otimes \ad P) \ni A \mapsto F_A \in L^q(X;\wedge^2(T^*X)\otimes \ad P),
\]
is transverse to the subspace $L^q(X;\wedge^-(T^*X)\otimes \ad P)$. Similarly, the final conclusion of Theorem \ref{mainthm:Lojasiewicz-Simon_inequalities_Yang-Mills_energy_flat_Morse-Bott} may be rephrased as the assertion that the Yang--Mills energy function $\YM$ is Morse--Bott at a flat connection $\Gamma$ that is a regular point of the map,
\[
W_\Gamma^{1,q}(X;T^*X\otimes \ad P) \ni A \mapsto \Pi_\Gamma F_A \in \Ker d_\Gamma \cap L^q(X;\wedge^2(T^*X)\otimes \ad P),
\]
where $\Pi_\Gamma: L^2(X;\wedge^2(T^*X)\otimes \ad P) \to \Ker d_\Gamma \cap L^2(X;\wedge^2(T^*X)\otimes \ad P)$ is $L^2$-orthogonal projection or, equivalently, that the map
\[
W_\Gamma^{1,q}(X;T^*X\otimes \ad P) \ni A \mapsto \Pi_\Gamma F_A \in L^q(X;\wedge^2(T^*X)\otimes \ad P),
\]
is transverse to the subspace $\Ran d_\Gamma \cap L^q(X;\wedge^2(T^*X)\otimes \ad P)$. The gauge-theoretic concept of $\Gamma$ as a \emph{regular point} of the map $A \mapsto F_A$ is based on the elliptic complex containing $d_\Gamma:\Omega^1(X;\ad P) \to \Omega^2(X;\ad P)$ while the gauge-theoretic concept of $A_\infty$ as a \emph{regular point} of the map $A \mapsto F_A^+$ is based on the elliptic complex containing $d_{A_\infty}^+:\Omega^1(X;\ad P) \to \Omega^+(X;\ad P)$.

At an arbitrary critical point $A_\infty$ for the Yang--Mills energy function $\YM$ over a manifold $X$ of dimension $d \geq 2$, there is no deformation theory that is exactly analogous to those just described for flat or anti-self-dual connections. Koiso \cite[Lemma 1.5]{Koiso_1987} proposes employing the observation that the following ``Dual Bianchi Identity'' (see \cite[p. 235]{DK} or \cite[p. 577]{ParkerGauge}),
\[
d_A^*d_A^*F_A = 0,
\]
which holds for \emph{any} connection $A$, be used to define an elliptic complex for the Yang--Mills equation, perhaps by analogy with viewing the Bianchi Identity, $d_AF_A=0$, as motivation for the concept of a regular point in the zero locus of the map $A\mapsto F_A$. However, as Koiso himself seems to suggest \cite[p. 156]{Koiso_1987}, this does not appear to yield a useful deformation theory for arbitrary solutions $A_\infty$ to the Yang--Mills equation, except possibly when the formal dimension of the critical set is zero at the gauge-equivalence class $[A_\infty]$ \cite[Corollary 2.11]{Koiso_1987}. In the case of flat connections on principal $G$-bundles over Riemannian manifolds, Ho, Wilkin and Wu \cite{Ho_Wilkin_Wu_2019} compare concepts of regular points from the perspectives of gauge theory and character varieties.

\subsection{Morse--Bott functions and moment maps}
We briefly note the well-known relationship between Morse--Bott functions and moment maps and recall the following result due to Atiyah.

\begin{thm}[Moment maps and Morse--Bott functions]
\label{thm:Atiyah}
(See Atiyah \cite{Atiyah_1982} or Nicolaescu \cite[Theorem 3.52]{Nicolaescu_morse_theory}.)
Let $(M,\omega)$ be a compact symplectic manifold equipped with a Hamiltonian action of the torus $T = S^1\times\cdots\times S^1$ ($\nu$ times for $\nu \geq 1$). Let $\mu:M\to\ft^*$ be the moment map of this action, where $\ft$ denotes the Lie algebra of $T$. Then, for every $\xi \in \ft$, the function
\begin{equation}
\phi_\xi:M \ni x \mapsto \langle \xi, \mu(x)\rangle_{\ft\times\ft^*} \in \RR
\end{equation}
is Morse--Bott. The critical submanifolds are $T$-invariant symplectic submanifolds of $M$ and all the Morse indices and co-indices are even.
\end{thm}

If we define $\sE:M\to\RR$ by setting $\sE(x) = \frac{1}{2}\|\mu(x)\|^2$, then $\sE'(x)\xi = \langle \xi, \mu(x)\rangle$ and $\Crit\sE = \mu^{-1}(0)$ and $\sE''(x)(\eta)\xi = \langle \xi, \mu'(x)\eta\rangle$. If $x_0 \in M$ is a regular point in the zero-locus of the moment map $\mu$, then $\mu^{-1}(0)\cap U$ is a relatively open, smooth submanifold of $M$ and $x_0$ is a critical point of $\sE:M\to\RR$ and $T_{x_0}(\mu^{-1}(0)\cap U) = \Ker\sE''(x_0)\cap T_{x_0}M$. In other words, $\sE$ is Morse--Bott at regular points $x_0 \in \mu^{-1}(0)$. Some aspects of Atiyah's Theorem have been extended, at least formally, to more general finite and infinite-dimensional settings and we refer to \cite[Section 6.5.1--3]{DK} for a discussion of moment maps and a survey of examples. The instances most relevant to this article include the
\begin{inparaenum}[\itshape a\upshape)]
\item affine space $\sA(P)$ of $W^{1,q}$ connections on a principal $G$-bundle $P$ over a Riemann surface $X$ and moment map $A\mapsto F_A$ for the Banach Lie group $\Aut(P)$ of gauge transformations \cite{Atiyah_Bott_1983}; and more generally, the
\item affine space $\sA(P)$ of $W^{1,q}$ connections on a principal $G$-bundle $P$ over a \emph{symplectic} manifold $(X,\omega)$ of dimension $2n$ and moment map $A\mapsto F_A\wedge\omega^{n-1}$ for $\Aut(P)$ \cite[Proposition 6.5.8]{DK}.
\end{inparaenum}
Donaldson and Kronheimer also point out that the Atiyah--Hitchin--Drinfel$'$d--Manin (ADHM) description of instantons over $\RR^4$ \cite{ADHM}, \cite[Section 3.3.2]{DK} may be viewed as the zero-locus of a suitably defined moment map \cite[p. 250]{DK}.

For further discussion of Morse--Bott functions, moment maps, and gradient flows in symplectic geometry, we refer to Donaldson and Kronheimer \cite[Section 6.5]{DK}, Kirwan \cite{Kirwan_cohomology_quotients_symplectic_algebraic_geometry, Kirwan_1987}, Lerman \cite{Lerman_2005}, Swoboda \cite{Swoboda_2012}, the references cited therein, and to Atiyah and Bott \cite{Atiyah_Bott_1983} and wealth of articles citing \cite{Atiyah_Bott_1983}.

\subsection{Notation}
\label{subsec:Notation}
For the notation of function spaces, we follow Adams and Fournier \cite{AdamsFournier}. If $V$ is a Riemannian vector bundle with orthogonal, smooth connection $A$ over a smooth Riemannian manifold $X$, we let $W_A^{k,p}(X;V)$ denote its Sobolev space of sections with up to $k$ covariant derivatives in $L^p$. We write $W^{k,p}(X;V)$ if the connection is unimportant for the context or if the Sobolev space is defined using standard definitions for functions on Euclidean space from \cite[Chapter 3]{AdamsFournier} and choices of local coordinate charts for $X$ and local trivializations for $V$. To define Sobolev norms of maps from a manifold into a compact Lie group, $G$, we choose a faithful unitary representation, $G \hookrightarrow \End(\CC^N)$.

If $G$ is compact Lie group and $P$ is a principal $G$-bundle over a manifold $X$, we let $\ad P := P\times_{\ad}\fg$ denote the real vector bundle associated to $P$ by the adjoint representation of $G$ on its Lie algebra, $\Ad:G \ni u \to \Ad_u \in \Aut\fg$. We fix a $G$-invariant inner product on the Lie algebra $\fg$ and thus define a fiber metric on $\ad P$. (When $G$ is semi-simple, one may use the Killing form to define a $G$-invariant inner product $\fg$.) When $X$ is equipped with a smooth Riemannian metric $g$, we let $\Inj(X,g)$ denote the injectivity radius of $(X,g)$ and, when $X$ also has an orientation, denote the corresponding volume form by $d\vol_g$. Unless stated otherwise, all manifolds are assumed to be compact and without boundary (closed), connected, orientable, and smooth.

We let $\NN:=\left\{1,2,3,\ldots\right\}$ denote the set of positive integers. We use $C=C(*,\ldots,*)$ to denote a constant which depends at most on the quantities appearing on the parentheses. In a given context, a constant denoted by $C$ may have different values depending on the same set of arguments and may increase from one inequality to the next. We emphasize that a constant $\eps$ (respectively, $C$) may need to be chosen sufficiently small (respectively, large) by writing $\eps \in (0,1]$ (respectively, $C \in [1,\infty)$).

For notation in functional analysis, we follow Brezis \cite{Brezis} and Rudin \cite{Rudin}. If $\sX, \sY$ is a pair of Banach spaces, then $\sL(\sX,\sY)$ denotes the Banach space of all continuous linear operators from $\sX$ to $\sY$. We denote the continuous dual space of $\sX$ by $\sX^* = \sL(\sX,\RR)$. We write $\alpha(x) = \langle x, \alpha \rangle_{\sX\times\sX^*}$ for the canonical pairing between $\sX$ and its dual space, where $x \in \sX$ and $\alpha \in \sX^*$. If $T \in \sL(\sX, \sY)$, then its range and kernel are denoted by $\Ran T$ and $\Ker T$, respectively.

\subsection{Acknowledgments}
I am very grateful to the National Science Foundation for their support and to the Simons Center for Geometry and Physics, Stony Brook, the Dublin Institute for Advanced Studies, and the Institut des Hautes {\'E}tudes Scientifiques, Bures-sur-Yvette, for their hospitality and support during the preparation of this article. I thank Manousos Maridakis for many helpful conversations regarding {\L}ojasiewicz--Simon gradient inequalities, Yasha Berchenko-Kogan for useful communications regarding Yang--Mills gauge theory, Changyou Wang for useful conversations regarding geometric analysis, and George Daskalopoulos, Richard Wentworth, and Graeme Wilkin for helpful correspondence regarding the Yang--Mills equations over Riemann surfaces. Lastly, I am most grateful to the anonymous referee for a thoughtful review of our article and for alerting me to the fact that Theorem  \ref{mainthm:Lojasiewicz-Simon_inequalities_Yang-Mills_energy_irreducible_Riemann_surface} should hold and suggesting a proof.

\section{Existence of a flat connection for critical exponents}
\label{sec:Existence_flat_connection_critical_exponents_Coulomb_distance}
Our goal in this section is to prove Theorem \ref{mainthm:Uhlenbeck_Chern_corollary_4-3_1_lessthan_p_lessthan_d}, thus extending part of a result \cite[Corollary 4.3]{UhlChern} due to Uhlenbeck, quoted as the forthcoming Theorem \ref{thm:Uhlenbeck_Chern_corollary_4-3}. Specifically, we relax the forthcoming curvature hypothesis \eqref{eq:Curvature_Lq_small}, that is,
\[
\|F_A\|_{L^q(X)} \leq \eps,
\]
when $q>d/2$ to the weaker condition \eqref{eq:mainCurvature_Ldover2_small}, namely,
\[
\|F_A\|_{L^{s_0}(X)} \leq \eps,
\]
where $s_0 = d/2$ when $d\geq 3$ or $s_0 > 1$ when $d=2$.

In Section \ref{subsec:Uhlenbeck_existence_flat_connection_and_apriori_estimate}, we recall the statement of Theorem \ref{thm:Uhlenbeck_Chern_corollary_4-3}, together with remarks on its hypotheses. In Section \ref{subsec:Flat_bundles}, we review the equivalent characterizations of flat bundles. In Section \ref{subsec:Extension_Uhlenbeck_local_Coulomb_gauge_theorem}, we establish an extension of Uhlenbeck's \cite[Theorem 1.3 or Theorem 2.1 and Corollary 2.2]{UhlLp} on existence of a local Coulomb gauge and \apriori $W^{1,p}$ estimate for connections with $L^{d/2}$-small curvature over a ball to include the range $1 < p < d/2$ (when $d\geq 3$). Our principal goal in Section \ref{subsec:Continuous_principal_bundles} is to prove Theorem \ref{thm:Uhlenbeck_Lp_bound_3-2_Sobolev_Ld_small_connection_oneforms}, which yields a continuous isomorphism between principal $G$-bundles that support Sobolev connections whose local connection one-forms in Coulomb gauge are $L^d$-close and whose corresponding transition functions are $L^p$-close for some $p \in (d/2,d)$. In Section \ref{subsec:Uhlenbeck_Chern_corollary_4-3_existence_flat_connection_critical_exponent}
we establish Theorem \ref{thm:Uhlenbeck_Chern_corollary_4-3_critical_existence}, verifying the assertion in Theorem \ref{mainthm:Uhlenbeck_Chern_corollary_4-3_1_lessthan_p_lessthan_d} of existence of a $C^\infty$ flat connection $\Gamma$ on a principal bundle supporting a $W^{1,q}$ connection $A$ with $L^{d/2}$-small curvature $F_A$ for $d \geq 3$ or $L^{s_0}$-small curvature for $d = 2$ and $s_0>1$.

\subsection{Existence of a flat connection for supercritical exponents}
\label{subsec:Uhlenbeck_existence_flat_connection_and_apriori_estimate}
In \cite{UhlChern}, Uhlenbeck establishes the

\begin{thm}[Existence of a flat connection on a principal bundle supporting a $W^{1,q}$ connection with $L^q$-small curvature]
\label{thm:Uhlenbeck_Chern_corollary_4-3}
(See Uhlenbeck \cite[Corollary 4.3]{UhlChern}.)
Let $(X,g)$ be a closed, smooth Riemannian manifold of dimension $d\geq 2$, and $G$ be a compact Lie group, and $q \in (d/2, \infty]$ be a constant. Then there is a constant $\eps=\eps(g,G,q) \in (0,1]$ with the following significance. If $A$ is a $W^{1,q}$ connection on a smooth principal $G$-bundle $P$ over $X$ such that
\begin{equation}
\label{eq:Curvature_Lq_small}
\|F_A\|_{L^q(X)} \leq \eps,
\end{equation}
then there is a $W^{1,q}$ flat connection $\Gamma$ on $P$.
\end{thm}

We give a detailed proof of Theorem \ref{thm:Uhlenbeck_Chern_corollary_4-3}, together with additional estimates, in our \cite[Theorem 1]{Feehan_nonlinear_uhlenbeck_estimate}. Uhlenbeck \cite[Corollary 4.3]{UhlChern} had also asserted that there is a $W^{2,q}$ gauge transformation $u$ of $P$ such that
\begin{equation}
\label{eq:Intro_Uhlenbeck_corollary_4-3_W1q_estimate}
  \|u(A)-\Gamma\|_{W^{1,q}(X)} \leq C\|F_A\|_{L^q(X)},
\end{equation}
where $C=C(d,G,q) \in [1,\infty)$ is a constant and\footnote{This Coulomb gauge condition appears in Uhlenbeck's proof of \cite[Corollary 4.3]{UhlChern}, though not in the statement of her result.} $d_\Gamma^*(u(A)-\Gamma)=0$. In \cite[Theorem 5.1]{Feehan_yangmillsenergygapflat_aim}, we had attempted to supply a detailed proof of \cite[Corollary 4.3]{UhlChern}, which was omitted in \cite{UhlChern}, but our argument was incorrect as we explain in \cite{Feehan_yangmillsenergygapflat_corrigendum}.

While we proved \cite[Theorem 9]{Feehan_nonlinear_uhlenbeck_estimate} that the estimate \eqref{eq:Intro_Uhlenbeck_corollary_4-3_W1q_estimate} holds when the Yang--Mills energy function on the space of Sobolev connections is Morse--Bott along the moduli subspace $M(P)$ of flat connections, it does not hold when the Yang--Mills energy function fails to be Morse--Bott, such as at the product connection in the moduli space of flat $\mathrm{SU}(2)$ connections over a real two-dimensional torus. In \cite[Appendix A]{Feehan_nonlinear_uhlenbeck_estimate}, we describe an example due to Mrowka \cite{Mrowka_7-30-2018} which shows that \eqref{eq:Intro_Uhlenbeck_corollary_4-3_W1q_estimate} cannot hold in general. However, in \cite[Theorem 9]{Feehan_nonlinear_uhlenbeck_estimate}, we prove that a useful modification of Uhlenbeck's estimate,
\begin{equation}
\label{eq:Intro_Uhlenbeck_corollary_4-3_corrected_W1q_estimate}
\|u(A)-\Gamma\|_{W_\Gamma^{1,p}(X)} \leq C\|F_A\|_{L^p(X)}^\lambda, 
\end{equation}
where $p\in(1,q]$ and $C=C(d,G,p) \in [1,\infty)$ and the positive exponent $\lambda=\lambda(g,G,\Gamma) \in (0,1]$ reflects the possibly singular structure of the moduli space $M(P)$ near $[\Gamma]$.

\subsection{Flat bundles}
\label{subsec:Flat_bundles}
We recall the equivalent characterizations of \emph{flat bundles} \cite[Section 1.2]{Kobayashi}, that is, bundles admitting a flat connection. Let $G$ be a Lie group and $P$ be a smooth principal $G$-bundle over a smooth manifold $X$. Let $\{U_\alpha\}_{\alpha\in\sI}$ be an open cover of $X$ with local sections, $\sigma_\alpha: U_\alpha \to P$ and $g_{\alpha\beta}: U_\alpha\cap U_\beta \to G$ be the family of transition functions defined by $\{U_\alpha, \sigma_\alpha\}$. A \emph{flat structure} in $P$ is given by $\{U_\alpha, \sigma_\alpha\}_{\alpha\in\sI}$ such that the $g_{\alpha\beta}$ are all constant maps. A connection in $P$ is said to be \emph{flat} if its curvature vanishes identically.

\begin{prop}[Characterizations of flat principal bundles]
\label{prop:Kobayashi_1-2-6}
(See \cite[Proposition 1.2.6]{Kobayashi}.)
For a smooth principal $G$-bundle $P$ over a smooth manifold $X$, the following conditions are equivalent:
\begin{enumerate}
  \item $P$ admits a flat structure,
  \item $P$ admits a flat connection,
  \item $P$ is defined by a representation\footnote{In the sense of \cite[Section 1.2]{Kobayashi}.} $\pi_1(X) \to G$.
\end{enumerate}
\end{prop}

Given a flat structure on $P$, we may construct a flat connection $\Gamma$ on $P$ using the zero local connection one-forms $\gamma_\alpha \equiv 0$ on $U_\alpha$, for each $\alpha$ as in \cite[Equation $(1.2.1')$]{Kobayashi}, and observing that the compatibility conditions \cite[Equation (1.1.16)]{Kobayashi},
\[
0 = \gamma_\beta
= g_{\alpha\beta}^{-1}\gamma_\alpha g_{\alpha\beta} + g_{\alpha\beta}^{-1} dg_{\alpha\beta}
= 0 \quad\text{on } U_\alpha\cap U_\beta,
\]
are automatically obeyed.

\subsection{An extension of Uhlenbeck's Theorem on existence of a local Coulomb gauge}
\label{subsec:Extension_Uhlenbeck_local_Coulomb_gauge_theorem}
We shall need to extend Uhlenbeck's Theorem on existence of a local Coulomb gauge to include the range $1 < p < d/2$ when $d \geq 3$ as well as $d/2 \leq p <d$. The required extension is given by Corollary \ref{cor:Uhlenbeck_theorem_1-3_p_lessthan_dover2}. We first recall the original statement of Uhlenbeck's Theorem (with a clarification due to Wehrheim).

\begin{thm}[Existence of a local Coulomb gauge and \apriori estimate for a Sobolev connection with $L^{d/2}$-small curvature]
\label{thm:Uhlenbeck_Lp_1-3}
(See Uhlenbeck \cite[Theorem 1.3 or Theorem 2.1 and Corollary 2.2]{UhlLp} or Wehrheim \cite[Theorem 6.1]{Wehrheim_2004}.)
Let $d\geq 2$, and $G$ be a compact Lie group, and $p \in (1,\infty)$ obeying $d/2 \leq p < d$ and $s_0>1$ be constants. Then there are constants, $\eps=\eps(d,G,p,s_0) \in (0,1]$ and $C=C(d,G,p,s_0) \in [1,\infty)$, with the following significance. For $q \in [p,\infty)$, let $A$ be a $W^{1,q}$ connection on $B\times G$ such that
\begin{equation}
\label{eq:Ldover2_ball_curvature_small}
\|F_A\|_{L^{s_0}(B)} \leq \eps,
\end{equation}
where $B \subset \RR^d$ is the unit ball with center at the origin and $s_0=d/2$ when $d\geq 3$ and $s_0 > 1$ when $d=2$. Then there is a $W^{2,q}$ gauge transformation, $u:B\to G$, such that the following holds. If $A = \Theta + a$, where $\Theta$ is the product connection on $B\times G$, and $u(A) = \Theta + u^{-1}au + u^{-1}du$, then
\begin{align*}
d^*(u(A) - \Theta) &= 0 \quad \text{a.e. on } B,
\\
(u(A) - \Theta)(\vec n) &= 0 \quad \text{on } \partial B,
\end{align*}
where $\vec n$ is the outward-pointing unit normal vector field on $\partial B$, and
\begin{equation}
\label{eq:Uhlenbeck_1-3_W1p_norm_connection_one-form_leq_constant_Lp_norm_curvature}
\|u(A) - \Theta\|_{W^{1,p}(B)} \leq C\|F_A\|_{L^p(B)}.
\end{equation}
\end{thm}

\begin{rmk}[Restriction of $p$ to the range $1<p<\infty$]
The restriction $p\in(1,\infty)$ should be included in the statements of \cite[Theorem 1.3 or Theorem 2.1 and Corollary 2.2]{UhlLp} since the bound \eqref{eq:Uhlenbeck_1-3_W1p_norm_connection_one-form_leq_constant_Lp_norm_curvature} ultimately follows from an \apriori $L^p$ estimate for an elliptic system that is apparently only valid when $1<p<\infty$. Wehrheim makes a similar observation in her \cite[Remark 6.2 (d)]{Wehrheim_2004}. This is also the reason that when $d=2$, we require $s_0>1$ in \eqref{eq:Ldover2_ball_curvature_small}.
\end{rmk}

\begin{rmk}[Dependencies of the constants in Theorem \ref{thm:Uhlenbeck_Lp_1-3}]
\label{rmk:Uhlenbeck_Lp_1-3_constant_dependencies}
(See \cite[Remark 4.2]{Feehan_yangmillsenergygapflat_aim}.)
The statements of \cite[Theorem 1.3 or Theorem 2.1 and Corollary 2.2]{UhlLp} imply that the constants, $\eps$ in \eqref{eq:Ldover2_ball_curvature_small} and $C$ in \eqref{eq:Uhlenbeck_1-3_W1p_norm_connection_one-form_leq_constant_Lp_norm_curvature}, only depend the dimension, $d$. However, their proofs suggest that these constants may also depend on $G$ and $p$ through the appeal to an elliptic estimate for $d+d^*$ in the verification of \cite[Lemma 2.4]{UhlLp} and arguments immediately following.
\end{rmk}

\begin{rmk}[Construction of a $W^{k+1,q}$ transformation to Coulomb gauge]
\label{rmk:Uhlenbeck_theorem_1-3_Wkp}
(See \cite[Remark 4.3]{Feehan_yangmillsenergygapflat_aim}.)
We note that if $A$ is of class $W^{k,q}$, for an integer $k \geq 1$ and $q \geq 2$, then the gauge transformation, $u$, in Theorem \ref{thm:Uhlenbeck_Lp_1-3} is of class $W^{k+1,q}$; see \cite[page 32]{UhlLp}, the proof of \cite[Lemma 2.7]{UhlLp} via the Implicit Function Theorem for smooth functions on Banach spaces, and our proof of \cite[Theorem 1.1]{FeehanSlice} --- a global version of Theorem \ref{thm:Uhlenbeck_Lp_1-3}.
\end{rmk}

\begin{rmk}[Non-flat Riemannian metrics]
\label{rmk:Non-flat_Riemannian_metrics}
Theorem \ref{thm:Uhlenbeck_Lp_1-3} continues to hold for geodesic unit balls in a manifold $X$ endowed a non-flat Riemannian metric, $g$. The only difference in this more general situation is that the constants $C$ and $\eps$ will depend on bounds on the Riemann curvature tensor $\Riem$. See Wehrheim \cite[Theorem 6.1]{Wehrheim_2004}.
\end{rmk}

We now provide an extension of Theorem \ref{thm:Uhlenbeck_Lp_1-3} to include the range $1 < p < d/2$ (and in particular, $p=2$, when $d \geq 5$).

\begin{cor}[Existence of a local Coulomb gauge and \apriori $W^{1,p}$ estimate for a Sobolev connection with $L^{d/2}$-small curvature when $p < d/2$]
\label{cor:Uhlenbeck_theorem_1-3_p_lessthan_dover2}
Assume the hypotheses of Theorem \ref{thm:Uhlenbeck_Lp_1-3}, but allow any $p \in (1,\infty)$ obeying $p < d/2$ when $d \geq 3$. Then the estimate \eqref{eq:Uhlenbeck_1-3_W1p_norm_connection_one-form_leq_constant_Lp_norm_curvature} holds for $1 < p < d/2$.
\end{cor}

\begin{proof}
The proof of Theorem \ref{thm:Uhlenbeck_Lp_1-3} by Uhlenbeck in \cite[Section 2]{UhlLp} makes use of the hypothesis $d/2\leq p < d$ through her appeal to a H\"older inequality and a Sobolev embedding. However, an alternative H\"older inequality and Sobolev embedding apply for the case $1 < p < d/2$, as we now explain. Write $a := u(A)-\Theta \in W^{1,q}(B;T^*X\otimes\fg)$ for brevity and observe that by ellipticity of the first-order operator $d+d^*:\Omega^1(B;\fg) \to \Omega^2(B;\fg)\oplus \Omega^0(B;\fg)$ with its Neumann boundary condition, we have the \apriori global estimate (see \cite[Theorem 5.1 and p. 102, last paragraph]{Wehrheim_2004}),
\begin{equation}
\label{eq:Wehrheim_theorem_5-1-ii}
\|a\|_{W^{1,p}(B)} \leq C\|(d+d^*)a\|_{L^p(B)},
\end{equation}
for $C = C(d,G,p) \in [1,\infty)$. Using $d^*a = 0$ and $F_{u(A)} = F_{\Theta + a} = F_\Theta + da + a\wedge a = da + a\wedge a$ and $|F_{u(A)}| = |F_A|$ a.e. on $B$, the preceding bound yields
\[
\|a\|_{W^{1,p}(B)} \leq C\left( \|F_A\|_{L^p(B)} + \|a\wedge a\|_{L^p(B)} \right).
\]
We can estimate $\|a\wedge a\|_{L^p(B)}$ by writing, for a constant $c = c(d,G) \in [1,\infty)$,
\[
\|a\wedge a\|_{L^p(B)} \leq c\|a\|_{L^s(B)}\|a\|_{L^d(B)},
\]
where $s>p$ is defined by $1/p = 1/s + 1/d$, that is, $1/s = (d-p)/dp$ or $s = dp/(d-p)$. Recall from \cite[Theorem 4.12, Part I (C)]{AdamsFournier} that there is a continuous embedding of Sobolev spaces, $W^{1,p}(B;\RR) \subset L^{p^*}(B;\RR)$, when $1\leq p < d$ (by hypothesis, we have $1 < p < d/2$) and $p^* = dp/(d-p) = s$. Hence, noting that $(d/2)^* = d(d/2)/(d-(d/2)) = d$, we obtain\footnote{Throughout this article, we apply the pointwise Kato Inequality \cite[Equation (6.20)]{FU} to pass from a Sobolev inequality for scalar functions to a Sobolev inequality with the same constant for sections of a vector bundle.}
\begin{align*}
\|a\|_{L^d(B)} &\leq C\|a\|_{W^{1,d/2}(B)},
\\
\|a\|_{L^{p^*}(B)} &\leq C\|a\|_{W^{1,p}(B)},
\end{align*}
for $C=C(d)$ or $C=C(d,p) \in [1,\infty)$, respectively. Therefore,
\[
\|a\wedge a\|_{L^p(B)}
\leq
c\|a\|_{L^{p^*}(B)}\|a\|_{L^d(B)}
\leq
C\|a\|_{W^{1,p}(B)}\|a\|_{W^{1,d/2}(B)},
\]
for $C=C(d,G,p) \in [1,\infty)$. The estimate \eqref{eq:Uhlenbeck_1-3_W1p_norm_connection_one-form_leq_constant_Lp_norm_curvature} (with $p=d/2$) from Theorem \ref{thm:Uhlenbeck_Lp_1-3} yields
\[
\|a\|_{W^{1,d/2}(B)} \leq C\|F_A\|_{L^{d/2}(B)},
\]
for $C = C(d,G) \in [1, \infty)$. But $\|F_A\|_{L^{d/2}(B)} \leq \eps$ by hypothesis \eqref{eq:Ldover2_ball_curvature_small} of Theorem \ref{thm:Uhlenbeck_Lp_1-3}, so we may combine the preceding inequalities to give, for $C=C(d,G,p) \in [1,\infty)$,
\[
\|a\wedge a\|_{L^p(B)} \leq C\eps\|a\|_{W^{1,p}(B)}.
\]
Consequently,
\begin{align*}
\|a\|_{W^{1,p}(B)} &\leq C\left(\|F_A\|_{L^p(B)} + \|a\wedge a\|_{L^p(B)}\right)
\\
&\leq C\left(\|F_A\|_{L^p(B)} + \eps\|a\|_{W^{1,p}(B)}\right).
\end{align*}
Hence, for small enough $\eps = \eps(d,G,p) \in (0,1]$, we may use rearrangement to find
\[
\|a\|_{W^{1,p}(B)} \leq C\|F_A\|_{L^p(B)},
\]
and this yields \eqref{eq:Uhlenbeck_1-3_W1p_norm_connection_one-form_leq_constant_Lp_norm_curvature} when $p \in (1,d/2)$.
\end{proof}

For completeness, we shall also include the following extension of Theorem \ref{thm:Uhlenbeck_Lp_1-3} (and slight improvement of our \cite[Corollary 4.4]{Feehan_yangmillsenergygapflat_aim}) to include the range $d \leq p < \infty$, although this extension will not be needed in this article.

\begin{cor}[Existence of a local Coulomb gauge and \apriori $W^{1,p}$ estimate for a Sobolev connection one-form with $L^{\bar p}$-small curvature when $p \geq d$]
\label{cor:Uhlenbeck_theorem_1-3_p_geq_d}
Assume the hypotheses of Theorem \ref{thm:Uhlenbeck_Lp_1-3}, but consider $d \leq p < \infty$ and strengthen \eqref{eq:Ldover2_ball_curvature_small} to\footnote{In \cite[Corollary 4.4]{Feehan_yangmillsenergygapflat_aim}, we assumed the still stronger condition, $\|F_A\|_{L^p(B)} \leq \eps$. }
\begin{equation}
\label{eq:Lbarp_ball_curvature_small}
\|F_A\|_{L^{\bar p}(B)} \leq \eps,
\end{equation}
where $\bar p = dp/(d+p)$ when $p>d$ and $\bar p > d/2$ when $p=d$. Then the estimate \eqref{eq:Uhlenbeck_1-3_W1p_norm_connection_one-form_leq_constant_Lp_norm_curvature} holds for $d \leq p < \infty$ and constant $C = C(d,p,\bar p,G) \in [1,\infty)$.
\end{cor}

\begin{proof}
We modify the proof of Corollary \ref{cor:Uhlenbeck_theorem_1-3_p_lessthan_dover2} and separately consider the cases $d<p<\infty$ and $p=d$. When $p > d$, then \cite[Theorem 4.12, Part I (A)]{AdamsFournier} provides a continuous embedding of Sobolev spaces, $W^{1,p}(B;\RR) \subset L^\infty(B;\RR)$. Also,
\cite[Theorem 4.12, Part I (C)]{AdamsFournier} provides a continuous embedding of Sobolev spaces, $W^{1,\bar p}(B;\RR) \subset L^p(B;\RR)$ when $p={\bar p}^*:=d\bar p/(d-\bar p) \in (d,\infty)$, that is, $\bar p = dp/(d+p) \in (d/2,d)$. Thus,
\[
\|a\wedge a\|_{L^p(B)} \leq c\|a\|_{L^p(B)}\|a\|_{L^\infty(B)} \leq C\|a\|_{W^{1,\bar p}(B)}\|a\|_{W^{1,p}(B)},
\]
for $c = c(d,G) \in [1,\infty)$ and $C=C(d,G,p) \in [1,\infty)$. Because $\bar p \in (d/2,d)$, Theorem \ref{thm:Uhlenbeck_Lp_1-3} applies to give
\[
\|a\|_{W^{1,\bar p}(B)} \leq C\|F_A\|_{L^{\bar p}(B)}.
\]
Thus, for $F_A$ obeying \eqref{eq:Lbarp_ball_curvature_small} with $\bar p = dp/(d+p)$, the proof of Corollary \ref{cor:Uhlenbeck_theorem_1-3_p_lessthan_dover2} yields estimate \eqref{eq:Uhlenbeck_1-3_W1p_norm_connection_one-form_leq_constant_Lp_norm_curvature}.

When $p=d$, choose $s \in (d,\infty)$ and define $t \in (d,\infty)$ by $1/d = 1/s+1/t$, so that
\[
\|a\wedge a\|_{L^p(B)} \leq c\|a\|_{L^s(B)}\|a\|_{L^t(B)}.
\]
For $\bar s = ds/(d+s) \in (d/2,d)$, we have a continuous embedding of Sobolev spaces, $W^{1,\bar s}(B;\RR) \subset L^s(B;\RR)$. Also, \cite[Theorem 4.12, Part I (B)]{AdamsFournier} provides a continuous embedding of Sobolev spaces, $W^{1,d}(B;\RR) \subset L^t(B;\RR)$. Therefore, applying these embeddings to the preceding inequality yields
\[
\|a\wedge a\|_{L^p(B)} \leq C\|a\|_{W^{1,\bar s}(B)}\|a\|_{W^{1,d}(B)},
\]
for $C=C(d,G,p,s) \in [1,\infty)$. Because $\bar s \in (d/2,d)$, Theorem \ref{thm:Uhlenbeck_Lp_1-3} again applies to give
\[
\|a\|_{W^{1,\bar s}(B)} \leq C\|F_A\|_{L^{\bar s}(B)}.
\]
Thus, for $F_A$ obeying \eqref{eq:Lbarp_ball_curvature_small} with $\bar p = \bar s$, the proof of Corollary \ref{cor:Uhlenbeck_theorem_1-3_p_lessthan_dover2} again yields estimate \eqref{eq:Uhlenbeck_1-3_W1p_norm_connection_one-form_leq_constant_Lp_norm_curvature}.
\end{proof}

\subsection{Continuous principal bundles}
\label{subsec:Continuous_principal_bundles}
Our principal goal in this subsection is to prove the forthcoming Theorem \ref{thm:Uhlenbeck_Lp_bound_3-2_Sobolev_Ld_small_connection_oneforms}, which yields a continuous isomorphism between principal $G$-bundles that support Sobolev connections whose local connection one-forms in Coulomb gauge are $L^d$-small and whose corresponding transition functions are $L^p$-close for some $p \in (d/2,d)$.

Recall \cite[Theorem 5.3.2]{Husemoller} that a continuous principal $G$-bundle $P_g$ over $X$ is uniquely defined up to isomorphism by a collection of maps, $g_{\alpha\beta}:U_\alpha \cap U_\beta \to G$, corresponding to a covering $\sU = \{U_\alpha\}_{\alpha\in\sI})$ of $X$ by open subsets, that obeys the cocycle condition,
\begin{equation}
\label{eq:Cocycle_condition_transition_functions}
g_{\alpha\beta}g_{\beta\gamma}g_{\gamma\alpha} = \id_G \quad\text{on } U_\alpha \cap U_\beta \cap U_\gamma,
\end{equation}
for all $\alpha, \beta, \gamma \in \sI$ such that $U_\alpha \cap U_\beta \cap U_\gamma \neq \emptyset$. The condition \eqref{eq:Cocycle_condition_transition_functions} implies that $g_{\alpha\alpha}=\id_G$ on $U_\alpha$ and $g_{\alpha\beta}^{-1} = g_{\beta\alpha}$ on $U_\alpha\cap U_\beta$.
Moreover, according to \cite[Proposition 5.2.5]{Husemoller}, a bundle $P_g$ is isomorphic to $P_h = (\{h_{\alpha\beta}\}_{\alpha,\beta\in\sI}, \{U_\alpha\}_{\alpha\in\sI})$ if and only if there exist continuous maps, $\rho_\alpha:U_\alpha \to G$ for all $\alpha\in\sI$, such that
\begin{equation}
\label{eq:Equivalent_transition_functions}
h_{\alpha\beta} = \rho_\alpha^{-1}g_{\alpha\beta}\rho_\beta \quad\text{on } U_\alpha\cap U_\beta,
\end{equation}
for all $\alpha, \beta \in \sI$ such that $U_\alpha \cap U_\beta \neq \emptyset$. The result below due to Uhlenbeck provides a useful criterion for existence of the collection of maps $\{\rho_\alpha\}_{\alpha\in\sI}$.

\begin{prop}[Isomorphisms of principal bundles with sufficiently close transition functions]
\label{prop:Uhlenbeck_1982Lp_3-2}
(See Uhlenbeck \cite[Proposition 3.2]{UhlLp}, Wehrheim \cite[Lemma 7.2 (i)]{Wehrheim_2004}.)
Let $G$ be a compact Lie group and $X$ be a compact manifold of dimension $d\geq 2$ endowed with a Riemannian metric $g$. Let $\{g_{\alpha\beta}\}$ and $\{h_{\alpha\beta}\}$ be two sets of continuous transition functions with respect to a finite open cover $\sU = \{U_\alpha\}_{\alpha\in\sI}$ of $X$. Then there exist constants, $\eps = \eps(g, G, \sU) \in (0,1]$ and $C = C(g, G, \sU) \in [1,\infty)$, with the following significance. If
\begin{equation}
\label{eq:Uhlenbeck_1982Lp_proposition_3-2_C0_small_g_alphabeta-h_alphabeta}
\delta := \sup_{\begin{subarray}{c} x\in U_\alpha\cap U_\beta, \\ \alpha, \beta \in \sI \end{subarray}}
\left|g_{\alpha\beta}(x) - h_{\alpha\beta}(x)\right| \leq \eps,
\end{equation}
then there exists a finite open cover $\sV = \{V_\alpha\}_{\alpha\in\sI}$ of $X$, with $V_\alpha \subset U_\alpha$, and a set of continuous maps $\rho_\alpha : V_\alpha \to G$ such that
\[
\rho_\alpha g_{\alpha\beta} \rho_\beta^{-1} = h_{\alpha\beta} \quad\text{on }V_\alpha \cap V_\beta
\]
and
\begin{equation}
\label{eq:Uhlenbeck_1982Lp_proposition_3-2_C0_bound_rho_alpha-id}
\sup_{\begin{subarray}{c} x \in V_\alpha, \\ \alpha \in \sI \end{subarray}}
\left|\rho_\alpha(x) - \id_G\right| \leq C\delta.
\end{equation}
In particular, the principal $G$-bundle defined by $\{g_{\alpha\beta}\}$ is isomorphic to the principal $G$-bundle defined by $\{h_{\alpha\beta}\}$.
\end{prop}

\begin{rmk}[Dependencies of the constants $\eps$ and $C$ in Proposition \ref{prop:Uhlenbeck_1982Lp_3-2}]
\label{rmk:Uhlenbeck_1982Lp_proposition_3-2_constant_dependencies}
The dependencies of the constants $\eps$ and $C$ in \cite[Proposition 3.2]{UhlLp} are not explicitly labeled, but those in Proposition \ref{prop:Uhlenbeck_1982Lp_3-2} are inferred from its proof in \cite{UhlLp}.
\end{rmk}

Next, we have the

\begin{thm}[$W^{2,p}$ bounds on transition functions for continuous principal bundles with $L^d$-small local connection one-forms in Coulomb gauge]
\label{thm:W2p_bound_transition_functions_Ld_small_connection_oneforms}
Let $(X,g)$ be a closed, smooth Riemannian manifold of dimension $d\geq 2$, and $G$ be a compact Lie group, $q>d/2$ be a constant, $P$ be a $W^{2,q}$ principal $G$-bundle over $X$, and $\sU = \{U_\alpha\}_{\alpha\in\sI}$ be a finite cover of $X$ by open subsets. Let $A$ be a $W^{1,d/2}$ connection on $P$ and $\sigma_\alpha:U_\alpha \to P$ be $W^{2,q}$ local sections such that the local connection one-forms
\[
a_\alpha := \sigma_\alpha^*A \in W^{1,d/2}(U_\alpha;T^*X\otimes\fg)
\]
obey, for each $\alpha \in \sI$,
\[
d^{*_g}a_\alpha = 0 \quad\text{a.e. on } U_\alpha.
\]
Let $\{g_{\alpha\beta}\}_{\alpha,\beta\in\sI}$ be the corresponding set of transition functions in $W^{2,q}(U_\alpha\cap U_\beta;G)$ for each $\alpha,\beta \in \sI$ such that $U_\alpha\cap U_\beta \neq \emptyset$. If $p \leq q$ obeys\footnote{By analogy with Corollary \ref{cor:Uhlenbeck_theorem_1-3_p_geq_d} or \cite[Corollary 4.4]{Feehan_yangmillsenergygapflat_aim}, the condition $p < d$ could be relaxed to $p \leq q$ at the expense in the hypothesis \eqref{eq:Coulomb_gauge_connection_one-forms_Ld_small} of replacing the $L^d$ norm by an $L^p$ norm when $p > d$.}
$1 < p < d$ and $\sV = \{V_\alpha\}_{\alpha\in\sI}$ is a finite cover of $X$ by open subsets such that $V_\alpha \Subset U_\alpha$, then there are constants $C = C(g,G,p,\sU,\sV) \in [1,\infty)$ and $\eps = \eps(g,G,p,\sU,\sV) \in (0,1]$ with the following significance. If
\begin{equation}
\label{eq:Coulomb_gauge_connection_one-forms_Ld_small}
\max_{\alpha \in \sI}\|a_\alpha\|_{L^d(U_\alpha \cap U_\beta)} \leq \eps,
\end{equation}
then
\begin{equation}
\label{eq:Coulomb_gauge_transition_functions_W2p_bound}
\|g_{\alpha\beta}\|_{W^{2,p}(V_\alpha\cap V_\beta)} \leq C.
\end{equation}
\end{thm}

\begin{rmk}[Uniform H\"older norm bounds on transition functions for continuous principal bundles with $L^d$-small local connection one-forms in Coulomb gauge]
\label{rmk:Holder_bound_transition_functions_Ld_small_connection_oneforms}
Recall from \cite[Theorem 4.12, Part II]{AdamsFournier} that there is a continuous embedding,
\[
W^{2,p}(U;\RR) \subset C^\delta(\bar U;\RR),
\]
where for an open subset $U \subset \RR^d$ (obeying an interior cone condition) and
\begin{inparaenum}[\itshape a\upshape)]
\item $0 < \delta \leq 2-(d/p)$ if $p < d < 2p$, or
\item $0 < \delta < 1$ if $p = d$,
\end{inparaenum}
and so the transition functions $g_{\alpha\beta}$ in Theorem \ref{thm:W2p_bound_transition_functions_Ld_small_connection_oneforms} obey a uniform $C^\delta(\bar V_\alpha\cap \bar V_\beta;G)$ bound.
\end{rmk}

The proof of Theorem \ref{thm:W2p_bound_transition_functions_Ld_small_connection_oneforms} is very similar to the proof of the forthcoming Theorem \ref{thm:Uhlenbeck_Lp_bound_3-2_Sobolev_Ld_small_connection_oneforms} and so is omitted: one simply uses $b_\alpha = a_\alpha$ and $h_{\alpha\beta} = g_{\alpha\beta}$ in the proof of Theorem \ref{thm:Uhlenbeck_Lp_bound_3-2_Sobolev_Ld_small_connection_oneforms} and notes that because $G$ is compact, $\|g_{\alpha\beta}\|_{L^\infty(U_\alpha\cap U_\beta)} \leq C$. Our proof of the forthcoming Theorem \ref{thm:Uhlenbeck_Chern_corollary_4-3_critical_existence} relies on the following generalization of Proposition \ref{prop:Uhlenbeck_1982Lp_3-2}.

\begin{thm}[Continuous principal bundles with $L^d$-small local connection one-forms in Coulomb gauge]
\label{thm:Uhlenbeck_Lp_bound_3-2_Sobolev_Ld_small_connection_oneforms}
Let $(X,g)$ be a closed, smooth Riemannian manifold of dimension $d\geq 2$, and $G$ be a compact Lie group, $q>d/2$ be a constant, $P$ and $Q$ be $W^{2,q}$ principal $G$-bundles over $X$, and $\sU=\{U_\alpha\}_{\alpha\in\sI}$ be a finite cover of $X$ by open subsets. Let $A$ and $B$ be $W^{1,d/2}$ connections on $P$ and $Q$, respectively, and $\sigma_\alpha,\, \varsigma_\alpha:U_\alpha \to P$ be $W^{2,q}$ local sections such that the local connection one-forms,
\[
a_\alpha := \sigma_\alpha^*A \in W^{1,d/2}(U_\alpha;T^*X\otimes\fg)
\quad\text{and}\quad
b_\alpha := \varsigma_\alpha^*B \in W^{1,d/2}(U_\alpha;T^*X\otimes\fg),
\]
obey, for each $\alpha \in \sI$,
\[
d^{*_g}a_\alpha = 0 = d^{*_g}b_\alpha \quad\text{a.e. on } U_\alpha.
\]
Let $\{g_{\alpha\beta}\}_{\alpha,\beta\in\sI}$ and $\{h_{\alpha\beta}\}_{\alpha,\beta\in\sI}$ be the corresponding sets of transition functions in $W^{2,q}(U_\alpha\cap U_\beta;G)$ for each $\alpha,\beta \in \sI$ such that $U_\alpha\cap U_\beta \neq \emptyset$. If $p \leq q$ obeys\footnote{By analogy with Corollary \ref{cor:Uhlenbeck_theorem_1-3_p_geq_d} or \cite[Corollary 4.4]{Feehan_yangmillsenergygapflat_aim}, the condition $p < d$ could be relaxed to $p \leq q$ at the expense in the hypothesis \eqref{eq:Coulomb_gauge_connection_one-forms_Ld_small_pair} of replacing the $L^d$ norm by an $L^p$ norm when $p > d$.}
$1 < p < d$ and $\sV=\{V_\alpha\}_{\alpha\in\sI}$ is a finite cover of $X$ by open subsets such that $V_\alpha \Subset U_\alpha$, then there are constants $C = C(g,G,p,\sU,\sV) \in [1,\infty)$ and $\eps = \eps(g,G,p) \in (0,1]$ with the following significance. If
\begin{equation}
\label{eq:Coulomb_gauge_connection_one-forms_Ld_small_pair}
\max_{\alpha \in \sI}\|a_\alpha\|_{L^d(U_\alpha \cap U_\beta)} \leq \eps
\quad\text{and}\quad
\max_{\alpha \in \sI}\|b_\alpha\|_{L^d(U_\alpha \cap U_\beta)} \leq \eps,
\end{equation}
then
\begin{multline}
\label{eq:Coulomb_gauge_transition_functions_W2p_bound_connection_one-forms}
\|g_{\alpha\beta}-h_{\alpha\beta}\|_{W^{2,p}(V_\alpha\cap V_\beta)}
\leq C\|g_{\alpha\beta}-h_{\alpha\beta}\|_{L^p(U_\alpha\cap U_\beta)}
\\
+ C\left(\|a_\alpha-b_\alpha\|_{L^d(U_\alpha\cap U_\beta)} + \|a_\beta-b_\beta\|_{L^d(U_\alpha\cap U_\beta)} \right).
\end{multline}
Moreover, if
\begin{equation}
\label{eq:Coulomb_gauge_transition_functions_Lp_close}
\max_{\alpha,\beta \in \sI}\|g_{\alpha\beta}-h_{\alpha\beta}\|_{L^p(U_\alpha\cap U_\beta)} \leq \eps,
\end{equation}
then
\begin{equation}
\label{eq:Coulomb_gauge_transition_functions_W2p_small}
\max_{\alpha,\beta \in \sI}
\|g_{\alpha\beta} - h_{\alpha\beta}\|_{W^{2,p}(V_\alpha\cap V_\beta)} \leq C\eps.
\end{equation}
Finally, if $p > d/2$ and $\eps= \eps(g,G,p,\sU,\sV) \in (0,1]$ is sufficiently small, then $P$ is isomorphic to $Q$ as a continuous principal $G$-bundle.
\end{thm}

\begin{proof}
We proceed by simplifying Taubes' proof of his \cite[Lemma A.1]{TauPath} (where $d=4$) and Rivi\`ere's proof of his \cite[Theorem IV.1]{Riviere_2002} (where $d\geq 4$). Let us observe that $d^{*_g} = -*_gd*_g:\Omega^1(X;\End(\fg))\to\Omega^0(X;\End(\fg))$ by \cite[Section 6.1]{Warner}, where $*=*_g:\Omega^l(X;\RR)\to\Omega^{d-l}(X;\RR)$ (for integers $0\leq l \leq d$) is the Hodge $*$-operator for the Riemannian metric $g$ on $X$ and we write $d^*=d^{*_g}$ for brevity in the remainder of the proof. Because $d^*a_\alpha = 0$ on $U_\alpha$ for all $\alpha \in \sI$, then the identity
\begin{equation}
\label{eq:dg_alphabeta}
dg_{\alpha\beta} = g_{\alpha\beta}a_\beta + a_\alpha g_{\alpha\beta} \quad\text{on } U_\alpha\cap U_\beta,
\end{equation}
yields
\begin{align*}
d^*dg_{\alpha\beta} &= -*(dg_{\alpha\beta}\wedge *a_\beta) + g_{\alpha\beta} d^*a_\beta + (d^*a_\alpha) g_{\alpha\beta} + *((*a_\alpha)\wedge dg_{\alpha\beta})
\\
&= -*(dg_{\alpha\beta}\wedge *a_\beta) + *((*a_\alpha)\wedge dg_{\alpha\beta}) \quad\text{on } U_\alpha\cap U_\beta.
\end{align*}
Similarly, we have
\begin{align*}
dh_{\alpha\beta} &= h_{\alpha\beta}b_\beta + b_\alpha h_{\alpha\beta},
\\
d^*dh_{\alpha\beta} &= -*(dh_{\alpha\beta}\wedge *b_\beta) + *((*b_\alpha)\wedge dh_{\alpha\beta}) \quad\text{on } U_\alpha\cap U_\beta.
\end{align*}
For brevity, define
\[
f_{\alpha\beta} := g_{\alpha\beta}-h_{\alpha\beta} \quad\text{on } U_\alpha\cap U_\beta, \quad\forall\,\alpha,\beta \in \sI,
\]
and observe that, by subtracting the corresponding the equations for $dh_{\alpha\beta}$ and $d^*dh_{\alpha\beta}$ from those for $dg_{\alpha\beta}$ and $d^*dg_{\alpha\beta}$, we obtain
\begin{align}
\label{eq:df_alpha_beta}
df_{\alpha\beta} &= f_{\alpha\beta}a_\beta + a_\alpha f_{\alpha\beta} + h_{\alpha\beta}(a_\beta-b_\beta) + (a_\alpha-b_\alpha) h_{\alpha\beta},
  \\
  \label{eq:Deltaf_alpha_beta} 
d^*df_{\alpha\beta}
&=
                      -*(df_{\alpha\beta}\wedge *a_\beta) + *((*a_\alpha)\wedge df_{\alpha\beta})
  \\
  \notag
  &\quad -*(dh_{\alpha\beta}\wedge *(a_\beta-b_\beta)) + *((*(a_\alpha-b_\alpha))\wedge dh_{\alpha\beta}).
\end{align}
If $\varphi_{\alpha_\beta} \in C_0^\infty(U_\alpha\cap U_\beta;\RR)$, then
\begin{align*}
d(\varphi_{\alpha_\beta}f_{\alpha\beta})
&= (d\varphi_{\alpha_\beta})f_{\alpha\beta} + \varphi_{\alpha_\beta}df_{\alpha\beta},
\\
d^*d(\varphi_{\alpha_\beta}f_{\alpha\beta})
&=
(d^*d\varphi_{\alpha_\beta})f_{\alpha\beta} + 2\langle \grad\varphi_{\alpha_\beta}, \grad f_{\alpha\beta}\rangle
+ \varphi_{\alpha_\beta}d^*df_{\alpha\beta}.
\end{align*}
Therefore, writing $c_\alpha : = a_\alpha-b_\alpha$ for $\alpha \in \sI$ for brevity, we have
\begin{align*}
d^*d(\varphi_{\alpha_\beta}f_{\alpha\beta})
&=
-*(df_{\alpha\beta}\wedge *a_\beta)\varphi_{\alpha_\beta} + *((*a_\alpha)\wedge df_{\alpha\beta})\varphi_{\alpha_\beta}
\\
&\quad - *(dh_{\alpha\beta}\wedge *c_\beta)\varphi_{\alpha_\beta} + *((*c_\alpha)\wedge dh_{\alpha\beta})\varphi_{\alpha_\beta}
\\
&\quad + (d^*d\varphi_{\alpha_\beta})f_{\alpha\beta} + 2\langle \grad\varphi_{\alpha_\beta}, \grad f_{\alpha\beta}\rangle,
\end{align*}
which gives
\begin{align*}
d^*d(\varphi_{\alpha_\beta}f_{\alpha\beta})
&=
-*(d(\varphi_{\alpha_\beta}f_{\alpha\beta})\wedge *a_\beta) + *((*a_\alpha)\wedge d(\varphi_{\alpha_\beta}f_{\alpha\beta}))
\\
&\quad + *((d\varphi_{\alpha_\beta})f_{\alpha\beta}\wedge *a_\beta) - *((*a_\alpha)\wedge (d\varphi_{\alpha_\beta})f_{\alpha\beta})
\\
&\quad - *(dh_{\alpha\beta}\wedge *c_\beta)\varphi_{\alpha_\beta} + *((*c_\alpha)\wedge dh_{\alpha\beta})\varphi_{\alpha_\beta}
\\
&\quad + (d^*d\varphi_{\alpha_\beta})f_{\alpha\beta} + 2\langle \grad\varphi_{\alpha_\beta}, \grad f_{\alpha\beta}\rangle.
\end{align*}
Assume that $\supp\varphi_\alpha \subset U_\alpha'$, where $U_\alpha' \Subset U_\alpha$ is an open subset (obeying an interior cone condition) for each $\alpha\in\sI$. Thus, for $d\geq 2$ and $p \in [1,d)$ and $p^* = dp/(d-p)\in [d/(d-1), \infty)$, so $1/p = 1/p^* + 1/d$, we have (provided $q \geq p$), for $c=c(g,G)\in [1,\infty)$,
\begin{align*}
\|d^*d(\varphi_{\alpha_\beta}f_{\alpha\beta})\|_{L^p(X)}
&\leq
c\|d(\varphi_{\alpha_\beta}f_{\alpha\beta})\|_{L^{p^*}(X)}
\left( \|a_\alpha\|_{L^d(U_\alpha\cap U_\beta)} + \|a_\beta\|_{L^d(U_\alpha\cap U_\beta)} \right)
\\
&\quad + c\|f_{\alpha\beta}\|_{L^{p^*}(U_\alpha'\cap U_\beta')}
\left( \|a_\alpha\|_{L^d(U_\alpha\cap U_\beta)} + \|a_\beta\|_{L^d(U_\alpha\cap U_\beta)} \right)
\\
&\quad + c\|dh_{\alpha\beta}\|_{L^{p^*}(U_\alpha\cap U_\beta)}\left(\|a_\alpha-b_\alpha\|_{L^d(U_\alpha\cap U_\beta)} + \|a_\beta-b_\beta\|_{L^d(U_\alpha\cap U_\beta)} \right),
\\
&\quad + c\left(\|d^*d\varphi_{\alpha_\beta}\|_{L^\infty(X)}\|f_{\alpha\beta}\|_{L^p(U_\alpha'\cap U_\beta')} + \|d\varphi_{\alpha_\beta}\|_{L^\infty(X)}\|df_{\alpha\beta}\|_{L^p(U_\alpha'\cap U_\beta')}\right).
\end{align*}
Using $W^{1,p}(U;\RR) \subset L^{p^*}(U;\RR)$, the continuous embedding of Sobolev spaces given by \cite[Theorem 4.12, Part I, Case C]{AdamsFournier}, for an open subset $U \subset \RR^d$ (obeying an interior cone condition), we obtain
\begin{align}
\label{eq:Coulomb_gauge_transition_functions_W2p_bound_connection_one-forms_raw}   
{}&\|d^*d(\varphi_{\alpha_\beta}f_{\alpha\beta})\|_{L^p(X)}
  \\
  \notag
&\quad \leq
C\|\varphi_{\alpha_\beta}f_{\alpha\beta}\|_{W^{2,p}(X)}
\left( \|a_\alpha\|_{L^d(U_\alpha\cap U_\beta)} + \|a_\beta\|_{L^d(U_\alpha\cap U_\beta)} \right)
  \\
  \notag
&\qquad + C\|f_{\alpha\beta}\|_{W^{1,p}(U_\alpha'\cap U_\beta')}
\left(\|a_\alpha\|_{L^d(U_\alpha\cap U_\beta)} + \|a_\beta\|_{L^d(U_\alpha\cap U_\beta)} \right)
  \\
  \notag
&\qquad + C\|h_{\alpha\beta}\|_{W^{2,p}(U_\alpha\cap U_\beta)}\left(\|a_\alpha-b_\alpha\|_{L^d(U_\alpha\cap U_\beta)} + \|a_\beta-b_\beta\|_{L^d(U_\alpha\cap U_\beta)} \right)
  \\
  \notag
&\qquad + C\left(\|d^*d\varphi_{\alpha_\beta}\|_{L^\infty(X)}\|f_{\alpha\beta}\|_{L^p(U_\alpha'\cap U_\beta')} + \|d\varphi_{\alpha_\beta}\|_{L^\infty(X)}\|df_{\alpha\beta}\|_{L^p(U_\alpha'\cap U_\beta')}\right),
\end{align}
for a constant $C=C(g,G,p,\sU)\in[1,\infty)$. For any $p \in (1,\infty)$, the following bound follows from the \apriori $L^p$ interior estimate for a linear second-order elliptic operator with scalar principal symbol \cite[Theorem 9.11]{GilbargTrudinger},
\begin{equation}
\label{eq:Gilbarg_Trudinger_9-36}
\|\varphi_{\alpha_\beta}f_{\alpha\beta}\|_{W^{2,p}(X)} \leq C\left(\|d^*d(\varphi_{\alpha_\beta}f_{\alpha\beta})\|_{L^p(X)} + \|\varphi_{\alpha_\beta}f_{\alpha\beta}\|_{L^p(X)}\right),
\end{equation}
for $C = C(g,G,p) \in [1,\infty)$. Hence, for $a_\alpha$ obeying \eqref{eq:Coulomb_gauge_connection_one-forms_Ld_small_pair} and choosing $\eps=\eps(g,G,p)\in[1,\infty)$ sufficiently small, rearrangement in \eqref{eq:Coulomb_gauge_transition_functions_W2p_bound_connection_one-forms_raw} with the aid of \eqref{eq:Gilbarg_Trudinger_9-36} gives
\begin{align*}
\|d^*d(\varphi_{\alpha_\beta}f_{\alpha\beta})\|_{L^p(X)}
&\leq C\|f_{\alpha\beta}\|_{W^{1,p}(U_\alpha'\cap U_\beta')}
\left(\|a_\alpha\|_{L^d(U_\alpha\cap U_\beta)} + \|a_\beta\|_{L^d(U_\alpha\cap U_\beta)} \right)
\\
&\quad + C\|h_{\alpha\beta}\|_{W^{2,p}(U_\alpha\cap U_\beta)}\left(\|a_\alpha-b_\alpha\|_{L^d(U_\alpha\cap U_\beta)} + \|a_\beta-b_\beta\|_{L^d(U_\alpha\cap U_\beta)} \right)
\\
&\quad + C\left(\|d^*d\varphi_{\alpha_\beta}\|_{L^\infty(X)}\|f_{\alpha\beta}\|_{L^p(U_\alpha'\cap U_\beta')} + \|d\varphi_{\alpha_\beta}\|_{L^\infty(X)}\|df_{\alpha\beta}\|_{L^p(U_\alpha'\cap U_\beta')}\right),
\end{align*}
and thus
\begin{align*}
\|\varphi_{\alpha_\beta}f_{\alpha\beta}\|_{W^{2,p}(X)}
&\leq C\|h_{\alpha\beta}\|_{W^{2,p}(U_\alpha\cap U_\beta)}\left(\|a_\alpha-b_\alpha\|_{L^d(U_\alpha\cap U_\beta)} + \|a_\beta-b_\beta\|_{L^d(U_\alpha\cap U_\beta)} \right)
\\
&\quad + C\left(1 + \|d^*d\varphi_{\alpha_\beta}\|_{L^\infty(X)}\right)\|f_{\alpha\beta}\|_{L^p(U_\alpha'\cap U_\beta')}
\\
&\quad + C\left(1 + \|d\varphi_{\alpha_\beta}\|_{L^\infty(X)}\right)\|df_{\alpha\beta}\|_{L^p(U_\alpha'\cap U_\beta')}.
\end{align*}
We now choose $\varphi_{\alpha\beta}$ obeying $\varphi_{\alpha\beta}=1$ on $V_\alpha\cap V_\beta$, so the preceding inequality and the $W^{2,p}$ bounds \eqref{eq:Coulomb_gauge_transition_functions_W2p_bound} for $h_{\alpha\beta}$ (with $g_{\alpha\beta}$ and $V_\alpha$ replaced by $h_{\alpha\beta}$ and $U_\alpha$, respectively) yield
\begin{multline}
\label{eq:Coulomb_gauge_transition_functions_W2p_bound_connection_one-forms_refined}  
\|g_{\alpha\beta}-h_{\alpha\beta}\|_{W^{2,p}(V_\alpha\cap V_\beta)}
\leq C\left(\|g_{\alpha\beta}-h_{\alpha\beta}\|_{L^p(U_\alpha\cap U_\beta)} + \|df_{\alpha\beta}\|_{L^p(U_\alpha\cap U_\beta)} \right)
\\
+ C\left(\|a_\alpha-b_\alpha\|_{L^d(U_\alpha\cap U_\beta)} + \|a_\beta-b_\beta\|_{L^d(U_\alpha\cap U_\beta)} \right),
\end{multline}
recalling that $f_{\alpha\beta} = g_{\alpha\beta}-h_{\alpha\beta}$. We can eliminate the term $\|df_{\alpha\beta}\|_{L^p(U_\alpha\cap U_\beta)}$ from the right-hand side of the inequality \eqref{eq:Coulomb_gauge_transition_functions_W2p_bound_connection_one-forms_refined} by using the identity \eqref{eq:df_alpha_beta} to give the estimate
\begin{align*}
\|df_{\alpha\beta}\|_{L^p(U_\alpha\cap U_\beta)} &\leq c\|f_{\alpha\beta}\|_{L^{p^*}(U_\alpha\cap U_\beta)}
\left( \|a_\alpha\|_{L^d(U_\alpha\cap U_\beta)} + \|a_\beta\|_{L^d(U_\alpha\cap U_\beta)} \right)
  \\
                                                 &\quad + c\|h_{\alpha\beta}\|_{L^{p^*}(U_\alpha\cap U_\beta)}\left(\|a_\alpha-b_\alpha\|_{L^d(U_\alpha\cap U_\beta)} + \|a_\beta-b_\beta\|_{L^d(U_\alpha\cap U_\beta)} \right)
  \\
   &\leq C\|f_{\alpha\beta}\|_{W^{1,p}(U_\alpha\cap U_\beta)}
\left( \|a_\alpha\|_{L^d(U_\alpha\cap U_\beta)} + \|a_\beta\|_{L^d(U_\alpha\cap U_\beta)} \right)
  \\
&\quad + C\|h_{\alpha\beta}\|_{W^{1,p}(U_\alpha\cap U_\beta)}\left(\|a_\alpha-b_\alpha\|_{L^d(U_\alpha\cap U_\beta)} + \|a_\beta-b_\beta\|_{L^d(U_\alpha\cap U_\beta)} \right),
\end{align*}
for constants $c=c(g,G)\in [1,\infty)$ and $C=C(g,G,p,\sU)\in[1,\infty)$ and applying the continuous Sobolev embedding $W^{1,p}(U_\alpha\cap U_\beta) \subset L^{p^*}(U_\alpha\cap U_\beta) $ to obtain the last inequality. Hence, for $a_\alpha$ obeying \eqref{eq:Coulomb_gauge_connection_one-forms_Ld_small_pair} and choosing $\eps=\eps(g,G,p)\in[1,\infty)$ sufficiently small, rearrangement gives 
\begin{multline}
\label{eq:Lp_bound_df_alpha_beta} 
\|df_{\alpha\beta}\|_{L^p(U_\alpha\cap U_\beta)} \leq c\|f_{\alpha\beta}\|_{L^p(U_\alpha\cap U_\beta)}
\left( \|a_\alpha\|_{L^d(U_\alpha\cap U_\beta)} + \|a_\beta\|_{L^d(U_\alpha\cap U_\beta)} \right)
  \\
 + c\|h_{\alpha\beta}\|_{W^{1,p}(U_\alpha\cap U_\beta)}\left(\|a_\alpha-b_\alpha\|_{L^d(U_\alpha\cap U_\beta)} + \|a_\beta-b_\beta\|_{L^d(U_\alpha\cap U_\beta)} \right).
\end{multline}
We now substitute the bound \eqref{eq:Lp_bound_df_alpha_beta} into the right-hand side of \eqref{eq:Coulomb_gauge_transition_functions_W2p_bound_connection_one-forms_refined} and use the $W^{1,p}$ bounds \eqref{eq:Coulomb_gauge_transition_functions_W2p_bound} for $h_{\alpha\beta}$ (with $g_{\alpha\beta}$ and $V_\alpha$ replaced by $h_{\alpha\beta}$ and $U_\alpha$, respectively) to obtain the desired estimate \eqref{eq:Coulomb_gauge_transition_functions_W2p_bound_connection_one-forms}. Moreover, by combining the estimate \eqref{eq:Coulomb_gauge_transition_functions_W2p_bound_connection_one-forms} and the bounds \eqref{eq:Coulomb_gauge_connection_one-forms_Ld_small_pair} and \eqref{eq:Coulomb_gauge_transition_functions_Lp_close} (with $p\leq d$) we obtain the estimate \eqref{eq:Coulomb_gauge_transition_functions_W2p_small}.

If $p>d/2$, then $W^{2,p}(U;\RR) \subset C^0(U;\RR)$ is a continuous embedding of Sobolev spaces by \cite[Theorem 4.12, Part I, Case A]{AdamsFournier}, for an open subset $U \subset \RR^d$ (obeying an interior cone condition), so \eqref{eq:Coulomb_gauge_transition_functions_W2p_small} yields
\[
\max_{\alpha,\beta \in \sI}
\|g_{\alpha\beta} - h_{\alpha\beta}\|_{C^0(\bar V_\alpha\cap \bar V_\beta)} \leq C\eps,
\]
for $C = C(g,G,p,\sU,\sV) \in [1,\infty)$. We now appeal to Proposition \ref{prop:Uhlenbeck_1982Lp_3-2} to give the desired isomorphism between $P$ and $Q$ to complete the proof of Theorem \ref{thm:Uhlenbeck_Lp_bound_3-2_Sobolev_Ld_small_connection_oneforms}.
\end{proof}

\begin{rmk}[Related results due to Riv\`ere and Taubes]
\label{rmk:Related_results_Rivere_Taubes}
Theorem \ref{thm:Uhlenbeck_Lp_bound_3-2_Sobolev_Ld_small_connection_oneforms} is essentially equivalent to
Riv\`ere's \cite[Theorem IV.1]{Riviere_2002} and that in turn may be viewed as a generalization of part of Taubes' \cite[Proposition 4.5 and Lemma A.1]{TauPath} from the case of $d=4$ to arbitrary $d\geq 4$. (It is likely that \cite[Theorem IV.1]{Riviere_2002} also holds for $d=3$ and possibly even $d=2$, but that would require carefully checking that all of the results used in the proof involving Lorentz spaces when $d \geq 4$ (as implicit throughout \cite{Riviere_2002}) also hold for $d=2$ or $3$.) In \cite[Theorem IV.1]{Riviere_2002}, Riv\`ere does not explicitly state that the local sections are continuous, but this appears to be implied by the proof.

In \cite{Riviere_2002}, Rivi\`ere uses \emph{Lorentz spaces} to obtain the necessary $L^\infty$ control over transition functions in this case of borderline $W^{1,d/2}$ strong convergence of local connection one-forms in Coulomb gauge. References for key results on Lorentz spaces employed by Rivi\`ere include Brezis and Wainger \cite{Brezis_Wainger_1980}, Lorentz \cite{Lorentz_1950, Lorentz_1951}, Peetre \cite{Peetre_1966, Peetre_1969}, Stein and Weiss \cite{Stein_Weiss_introduction_fourier_analysis}, Tartar \cite{Tartar_1998, Tartar_2007}, and Grafakos \cite{Grafakos_classical_fourier_analysis} for a more recent exposition. Our proof of Theorem \ref{thm:Uhlenbeck_Lp_bound_3-2_Sobolev_Ld_small_connection_oneforms} appears simpler than those of \cite[Theorem IV.1]{Riviere_2002} or \cite[Lemma A.1]{TauPath} since it only requires standard results on Sobolev spaces \cite{AdamsFournier} and an \apriori $L^p$ estimate for the Laplace operator \cite{GilbargTrudinger}.
\end{rmk}

\begin{rmk}[Related results due to Isobe and Shevchishin]
\label{rmk:Related_results_Isobe_Shevchishin}
Isobe has shown that any two $C^0$ principal $G$-bundles that are sufficiently close to each other in the $W^{1,d}(X)$-norm are necessarily isomorphic (see \cite[Theorem 1.1 and Proposition 3.1]{Isobe_2009}); compare Shevchishin \cite[Theorem 2.6]{Shevchishin_2002} for a related result. However, the proofs of \cite[Theorem 1.1 and Proposition 3.1]{Isobe_2009} are quite involved whereas the proof of Theorem \ref{thm:Uhlenbeck_Lp_bound_3-2_Sobolev_Ld_small_connection_oneforms} is direct and the result more than adequate for our application.
\end{rmk}

\subsection{Existence of a flat connection for the critical exponent}
\label{subsec:Uhlenbeck_Chern_corollary_4-3_existence_flat_connection_critical_exponent}
In this subsection, we establish the forthcoming Theorem \ref{thm:Uhlenbeck_Chern_corollary_4-3_critical_existence} --- an extension of Theorem \ref{thm:Uhlenbeck_Chern_corollary_4-3} --- giving existence of a $C^\infty$ flat connection on a principal $G$-bundle supporting a $W^{1,q}$ connection with $L^{d/2}$-small curvature for $d \geq 3$ or $L^{s_0}$-small curvature for $d = 2$ and $s_0>1$. This will also verify Theorem \ref{mainthm:Uhlenbeck_Chern_corollary_4-3_1_lessthan_p_lessthan_d}.

Suppose temporarily that $X$ is a closed, \emph{four}-dimensional, oriented, \emph{topological} manifold and that $G$ is a compact \emph{simple} Lie group. We recall from \cite[Appendix]{Sedlacek}, \cite[Propositions A.1 and A.2]{TauSelfDual} that a topological principal $G$-bundle $P$ over $X$ is classified up to isomorphism by a cohomology class $\eta(P) \in H^2(X;\pi_1(G))$ and its \emph{first Pontrjagin class}, $p_1(P) \in H^4(X;\ZZ)$, or equivalently, \emph{first Pontrjagin degree}, $\langle p_1(P), [X] \rangle \in \ZZ$, where $[X] \in H_4(X;\ZZ)$ denotes the fundamental class of $X$. The topological invariant $\eta \in H^2(X; \pi_1(G))$ is the \emph{obstruction} to the existence of a principal $G$-bundle $P$ over $X$ with a specified Pontrjagin degrees.

In his Ph.D. thesis \cite{SedlacekThesis} and its published version \cite{Sedlacek}, Sedlacek applied the direct minimization method to the Yang--Mills energy function \eqref{eq:Yang-Mills_energy_function} on the affine space of $W^{1,q}$ connections on a smooth principal $G$-bundle $P$ over a closed, four-dimensional, smooth Riemannian manifold $(X,g)$ to prove existence of a $C^\infty$ Yang--Mills connection $A_\infty$ on a smooth principal $G$-bundle $P_\infty$ over $X$, where $\eta(P_\infty) = \eta(P)$ and $p_1(P_\infty)[X] \geq p_1(P)[X]$ (see \cite[Theorems 4.3, 5.5, and 7.1 and Corollary 5.6]{Sedlacek}). Here, $\eta(P)$ is the obstruction class (see \cite[Section 2]{Sedlacek}), $p_1(P) \in H^4(X;\ZZ)$ is the Pontrjagin class of $P$, and $p_1(P)[X] \in \ZZ$ is the Pontrjagin degree for $P$. The case $p_1(P_\infty)[X] > p_1(P)[X]$ arises due to the phenomenon of energy bubbling, as explained in \cite[Sections 5 and 7]{Sedlacek}. In his proof of \cite[Theorems 4.1 and 4.3 and Proposition 4.2]{Sedlacek}, Sedlacek considers a sequence of $C^\infty$ connections $\{A^i\}_{i=1}^\infty$ on $P$ such that
\[
\YM(A^i) \searrow m(\eta), \quad\text{as } i \to \infty,
\]
where $m(\eta) := \inf\{\YM(A):\, A$ is a $C^\infty$ connection on a smooth principal $G$-bundle $P'$ such that $\eta(P') = \eta$\} and finds a $C^\infty$ Yang--Mills connection $A_\infty$ on a smooth principal $G$-bundle $P_\infty$ with $\eta(P_\infty) = \eta(P)$ by \cite[Theorem 5.6]{Sedlacek}. If $P$ supports a $C^\infty$ connection $A$ obeying the condition \eqref{eq:mainCurvature_Ldover2_small} with $d=4$, namely
\[
\|F_A\|_{L^2(X)} \leq \eps,
\]
then the Chern--Weil representation of characteristic classes \cite{MilnorStasheff} implies that $p_1(\ad P)[X] = 0$ for small enough $\eps = \eps(g,G,k) \in (0,1]$ (where $k = p_1(\ad P)[X] \in \ZZ)$. (Arguing along these lines, Sedlacek obtains his \cite[Theorem 7.1]{Sedlacek}.) But $\YM(A_\infty) \leq \YM(A)$ and thus also $p_1(\ad P_\infty)[X] = 0$. Hence, $P$ is isomorphic to $P_\infty$ as a continuous principal $G$-bundle, at least when $G$ is simple, by the preceding remarks on their classification.

While Sedlacek confines his attention to manifolds $X$ of dimension $d = 4$, his argument employs Uhlenbeck's Theorem \ref{thm:Uhlenbeck_Lp_1-3}, which is valid for the unit ball $B \subset \RR^d$ of any dimension $d \geq 2$. As we discuss here, it is therefore not difficult to modify his proof to yield a version of his \cite[Theorem 4.3]{Sedlacek} which is also valid for $X$ of any dimension $d \geq 2$. Furthermore, that generalization to $d\geq 2$ of \cite[Theorem 4.3]{Sedlacek} from $d=4$ will yield the desired enhancement (from $p > d/2$ to $p = d/2$) of Theorem \ref{thm:Uhlenbeck_Chern_corollary_4-3} (existence of a $C^\infty$ flat connection $\Gamma$ on a principal $G$-bundle $P$ supporting a $W^{1,q}$ connection $A$ with $L^p(X)$-small curvature $F_A$).

\begin{thm}[Existence of a $C^\infty$ flat connection on a principal bundle supporting a $W^{1,q}$ connection with $L^{d/2}$-small curvature]
\label{thm:Uhlenbeck_Chern_corollary_4-3_critical_existence}
Let $(X,g)$ be a closed, smooth Riemannian manifold of dimension $d\geq 2$, and $G$ be a compact Lie group, and $s_0>1$ be a constant. Then there is a constant, $\eps=\eps(g,G,s_0) \in (0,1]$, with the following significance. If $q \in (d/2, \infty]$ and $A$ is a $W^{1,q}$ connection on a smooth principal $G$-bundle $P$ over $X$ whose curvature obeys \eqref{eq:mainCurvature_Ldover2_small}, namely
\[
\|F_A\|_{L^{s_0}(X)} \leq \eps,
\]
where $s_0 = d/2$ when $d\geq 3$ or $s_0 > 1$ when $d=2$, then there is a $C^\infty$ flat connection $\Gamma$ on $P$.
\end{thm}

\begin{proof}
Suppose the conclusion is false, so we may select a sequence $\{A^i\}_{i=1}^\infty$ of $W^{1,q}$ connections on $P$ such that $\|F_{A^i}\|_{L^{d/2}(X)} \to 0$ as $i \to \infty$ but $P$ does not admit a $C^\infty$ flat connection. Choose a finite cover of $X$ by geodesic balls, $B^\alpha = B_\varrho(x_\alpha) \subset X$ with centers $x_\alpha \in X$ and radius $\varrho \in (0,\Inj(X,g))$, for all $\alpha\in\sI$. With the aid of geodesic normal coordinates, one sees that the Riemannian metric, $g$, is $C^1$-close to a flat metric in a small enough open neighborhood of $x_\alpha$ (see Aubin \cite[Definition 1.24, Proposition 1.25, and Corollary 1.32]{Aubin}). Choose $\eps \in (0,1]$ small enough that we can apply Theorem \ref{thm:Uhlenbeck_Lp_1-3}. Hence, there are a sequence of $W^{2,q}$ local sections $\sigma_\alpha^i:B_\alpha \to P$, and $W^{2,q}$ transition functions $g_{\alpha\beta}^i:B_\alpha\cap B_\beta \to G$, and local connection one-forms $a_\alpha^i = (\sigma_\alpha^i)^*A^i \in W^{1,q}(B_\alpha;T^*X\otimes\fg)$, such that for all $i \in \NN$ and $\alpha,\beta,\gamma \in \sI$,
\begin{align*}
d^{*_g}a_\alpha^i &= 0 \quad\text{on } B_\alpha,
\\
\|a_\alpha^i\|_{W^{1,d/2}(B_\alpha)} &\leq c\|F_{A^i}\|_{L^{d/2}(B_\alpha)},
\\
g_{\alpha\beta}^ig_{\beta\gamma}^ig_{\gamma\alpha}^i &= \id_G \quad\text{on } B_\alpha \cap B_\beta \cap B_\gamma,
\\
\nabla g_{\alpha\beta}^i = dg_{\alpha\beta}^i &= g_{\alpha\beta}^ia_\beta^i + a_\alpha^i g_{\alpha\beta}^i
\quad\text{on } B_\alpha\cap B_\beta,
\end{align*}
where $c=c(g,G)\in [1,\infty)$. Hence, for all $\alpha, \beta \in \sI$ and $i \to \infty$, we have
\begin{align*}
a_\alpha^i &\to 0 \quad\text{in } W^{1,d/2}(B_\alpha;T^*X\otimes\fg),
\\
\nabla g_{\alpha\beta}^i &\to 0 \quad\text{in } L^d(B_\alpha\cap B_\beta;G),
\end{align*}
since, using the continuous embedding of Sobolev spaces, $W^{1,d/2}(B;\RR) \subset L^d(B;\RR)$ by \cite[Theorem 4.12, Part I (C)]{AdamsFournier} for any ball $B \Subset \RR^d$,
\[
\|\nabla g_{\alpha\beta}^i\|_{L^d(B_\alpha\cap B_\beta)} \leq c\left(\|a_\alpha^i\|_{W^{1,d/2}(B_\alpha)} + \|a_\beta^i\|_{W^{1,d/2}(B_\beta)}\right),
\]
and the fact that $G$ is compact, so $\|g_{\alpha\beta}^i\|_{L^\infty(B_\alpha\cap B_\beta)} \leq c_0$, where $c_0=c_0(G)$ and $c=c(g,G,\varrho)\in[1,\infty)$. Moreover, because
\[
\nabla^2 g_{\alpha\beta}^i = \nabla g_{\alpha\beta}^i\otimes a_\beta^i + g_{\alpha\beta}^i\nabla a_\beta^i + (\nabla a_\alpha^i) g_{\alpha\beta}^i + a_\alpha^i \otimes \nabla g_{\alpha\beta}^i,
\]
and $L^d(B_\alpha\cap B_\beta) \times L^d(B_\alpha\cap B_\beta) \to L^{d/2}(B_\alpha\cap B_\beta)$ is a continuous Sobolev multiplication map and $W^{1,d/2}(B;\RR) \subset L^d(B;\RR)$ is a continuous Sobolev embedding,  we see that
\[
\nabla^2 g_{\alpha\beta}^i \to 0 \quad\text{in } L^{d/2}(B_\alpha\cap B_\beta;G),
\]
for all $\alpha, \beta \in \sI$, as $i \to \infty$. In particular, the sequence $\{g_{\alpha\beta}^i\}_{i=1}^\infty$ is uniformly bounded in $W^{2,d/2}(B_\alpha\cap B_\beta;G)$ and because the Sobolev embedding, $W^{2,d/2}(B;\RR) \Subset W^{1,r}(B;\RR)$ for $r \in [1,d)$, is compact by the Rellich-Kondrachov Theorem (see \cite[Theorem 6.3]{AdamsFournier}), then, after passing to a subsequence, there is a collection of maps, $h_{\alpha\beta}: B_\alpha \cap B_\beta \to G$ such that $\nabla h_{\alpha\beta} = 0$ on $B_\alpha\cap B_\beta$ and
\[
g_{\alpha\beta}^i \to h_{\alpha\beta} \quad\text{in } W^{1,r}(B_\alpha\cap B_\beta;G), \quad i\to\infty,
\]
for all $\alpha,\beta\in\sI$. (Note that $W^{1,r}(B;\RR) \subset L^d(B;\RR)$ is a continuous Sobolev embedding when $r \in [1,d)$ obeys $r^*=dr/(d-r) \geq d$, that is, $r\geq d/2$, and so we also have that $g_{\alpha\beta}^i \to h_{\alpha\beta}$ in $L^d(B_\alpha\cap B_\beta;G)$ as $i\to\infty$.) Hence, the sequence $\{g_{\alpha\beta}^i\}_{i=1}^\infty$ of $W^{2,q}$ transition functions, defining a sequence of $W^{2,q}$ principal $G$-bundles $P_i$ isomorphic to $P$ (as continuous principal $G$-bundles), converges in $W^{1,r}(B_\alpha\cap B_\beta;G)$ to a collection of constant maps $\{h_{\alpha\beta}\}_{\alpha,\beta\in\sI}$ obeying the cocycle condition,
\[
h_{\alpha\beta} h_{\beta\gamma} h_{\gamma\alpha} = \id_G \quad\text{on } B_\alpha \cap B_\beta \cap B_\gamma, \quad\forall\, \alpha,\beta,\gamma \in \sI.
\]
Therefore, by Proposition \ref{prop:Kobayashi_1-2-6} the collection $\{h_{\alpha\beta}\}_{\alpha,\beta\in\sI}$ defines a $C^\infty$ flat connection $\Gamma$ on a $C^\infty$ principal $G$-bundle $Q$ over $X$ with local connection one-forms $b_\alpha = 0$ on $B_\alpha$, for all $\alpha \in \sI$. But Theorem \ref{thm:Uhlenbeck_Lp_bound_3-2_Sobolev_Ld_small_connection_oneforms} implies that the sequence $\{g_{\alpha\beta}^i\}_{i=1}^\infty$ actually converges to $h_{\alpha\beta}$ in $W_\loc^{2,p}(B_\alpha\cap B_\beta;G)$, for any $p \leq q$ obeying $1 < p < d$ and all $\alpha, \beta \in \sI$, and that $Q$ is isomorphic to $P$ as a continuous principal bundle. This contradicts our initial assumption and thus proves Theorem \ref{thm:Uhlenbeck_Chern_corollary_4-3_critical_existence}.
\end{proof}

\begin{rmk}[Alternative proof of convergence of transition functions]
Rather than apply the Rellich-Kondrachov Theorem in the proof of Theorem \ref{thm:Uhlenbeck_Chern_corollary_4-3_critical_existence}, we may instead observe that the difference between the average $h_{\alpha\beta}^i := (g_{\alpha\beta}^i)_{B_\alpha \cap B_\beta} \in G$ of $g_{\alpha\beta}^i$ on $B_\alpha \cap B_\beta$,
\[
(g_{\alpha\beta})_{B_\alpha \cap B_\beta} := \frac{1}{\vol(B_\alpha \cap B_\beta)}\int_{B_\alpha \cap B_\beta} g_{\alpha\beta}\,d\vol, \quad\forall\, \alpha, \beta \in \sI,
\]
and $g_{\alpha\beta}^i$ may be estimated via the Poincar\'e Inequality \cite[Theorem 5.8.1]{Evans2},
\begin{equation}
\label{eq:Poincare_inequality_g-h}
\|g_{\alpha\beta}^i-h_{\alpha\beta}^i\|_{L^p(B_\alpha \cap B_\beta)} \leq C\|dg_{\alpha\beta}^i\|_{L^p(B_\alpha \cap B_\beta)}, \quad\forall\, \alpha, \beta \in \sI, \quad i \in \NN.
\end{equation}
But $G$ is compact and thus, after passing to a subsequence and relabelling, we may suppose that the sequence $\{h_{\alpha\beta}^i\}_{i=}^\infty$ converges to a limit $h_{\alpha\beta} \in G$ and consequently the sequence $\{g_{\alpha\beta}^i\}_{i=1}^\infty$ converges in $W^{2,p}(B_\alpha \cap B_\beta;G)$ to a limit $h_{\alpha\beta}$.
\end{rmk}

\section{{\L}ojasiewicz--Simon gradient inequalities for Morse--Bott functions}
\label{sec:Lojasiewicz-Simon_gradient_inequality_abstract_functional_Morse-Bott}
Our goal in this section is to give the

\begin{proof}[Proof of Theorem \ref{mainthm:Lojasiewicz-Simon_gradient_inequality_Morse-Bott}]
We begin with several reductions that simplify the proof. First, observe that if $\sE_0:\sU\to\RR$ is defined by $\sE_0(x) := \sE(x+x_\infty)$, then $\sE_0'(0)=0$, so we may assume without loss of generality that $x_\infty=0$ and relabel $\sE_0$ as $\sE$. Second, recall that by hypothesis, $\sX = \sX_0\oplus K$ (a direct sum of Banach spaces), where $\sX_0 \subset \sX$ is a closed subspace (a Banach space) complementing $K = \Ker\sE''(0) = \Ker\sM'(0)$. Hence, by applying a $C^2$ diffeomorphism to a neighborhood of the origin in $\sX$ and possibly shrinking $\sU$, we may assume without loss of generality that $\sU \cap \Crit\sE = \sU \cap K$, recalling that $K = T_{x_\infty}\Crit\sE$ by hypothesis that $\sE$ is Morse--Bott at $x_\infty$. Third, observe that if $\sE_0:\sU\to\RR$ is defined by $\sE_0(x) := \sE(x) - \sE(0)$, then $\sE_0(0)=0$, so we may once again relabel $\sE_0$ as $\sE$ and assume that $\sE(0)=0$.

By hypothesis, $\sG = \sG_0\oplus \sK$ (a direct sum of Banach spaces), where $\sK = \Ker\sM_1(0)$ has closed complement  $\sG_0$ (a Banach space), and $\sH_0 = \Ran\sM_1(0) \subset \sH$ is a closed subspace (a Banach space). Hence, the bounded operator $\sM_1(0):\sG_0\to\sH_0$ is bijective and thus invertible by the Open Mapping Theorem. Note that $K \subset \sK$ by definition of $\sM_1(0)$ and $\sX_0 \subset \sG_0$ by hypothesis.

By shrinking $\sU$ if necessary, we may assume without loss of generality that $\sU$ is convex. By the Mean Value Theorem and the hypothesis that $\sM:\sU\to\sY$ is $C^1$ and writing $x = \xi+k \in \sU$, for $\xi \in \sU\cap\sX_0$ and $k\in \sU\cap K$ and noting that $\sM(k)=0$ for all $k \in \sU\cap K$, we have
\begin{align*}
\sM(\xi+k) &= \int_{0}^{1}\sM'(t\xi)\xi\,dt
\\
&= \sM'(0)\xi + \int_{0}^{1}(\sM'(k+t\xi)-\sM'(0))\xi\,dt
\\
&= \sM_1(0)\xi + \int_{0}^{1}(\sM_1(k+t\xi)-\sM_1(0))\xi\,dt.
\end{align*}
Thus,
\[
\|\sM(\xi+k)\|_\sH \geq \|\sM_1(0)\xi\|_\sH - \max_{t\in[0,1]}\|(\sM_1(k+t\xi)-\sM_1(0))\xi\|_\sH.
\]
Because $\xi \in \sX_0 \subset \sG_0$ and $\sM_1(0):\sG_0\to\sH_0$ is invertible, we have
\[
\|\xi\|_\sG = \|\xi\|_{\sG_0} = \|\sM_1(0)^{-1}\sM_1(0)\xi\|_{\sG_0} \leq \|\sM_1(0)^{-1}\|_{\sL(\sH_0,\sG_0)}
\|\sM_1(0)\xi\|_{\sH_0}.
\]
Therefore,
\[
\|\sM_1(0)\xi\|_\sH = \|\sM_1(0)\xi\|_{\sH_0}
\geq \frac{\|\xi\|_\sG}{\|\sM_1(0)^{-1}\|_{\sL(\sH_0,\sG_0)}} =: 2C_0\|\xi\|_\sG.
\]
On the other hand, given $\eps \in (0,1]$,
\begin{align*}
\max_{t\in[0,1]}\|(\sM_1(k+t\xi)-\sM_1(0))\xi\|_\sH
&\leq
\max_{t\in[0,1]}\|\sM_1(k+t\xi)-\sM_1(0)\|_{\sL(\sG,\sH)}\|\xi\|_\sG
\\
&\leq \eps\|\xi\|_\sG,
\end{align*}
for $\|\xi\|_\sX , \|k\|_\sX \leq \delta=\delta(\eps) \in (0,1]$, where the final inequality follows by the hypothesis of continuity of $\sM_1(x) \in \sL(\sG,\sH)$ with respect to $x \in \sU$. Consequently, choosing $\eps\leq C_0$ yields
\begin{equation}
\label{eq:Gradient_lower_bound}
\|\sM(\xi+k)\|_\sH \geq C_0\|\xi\|_\sG, \quad\forall\, \xi+k \in\sX \text{ such that } \|\xi\|_\sX, \|k\|_\sX \leq \delta.
\end{equation}
In the other direction, since $\sE(k)=0$ and $\sE'(k)=0$ for all $k \in \sU\cap K$,
\[
\sE(\xi+k) = \int_{0}^{1}\sE''(k+t\xi)\xi^2\,dt = \sE''(0)\xi^2 + \int_{0}^{1}(\sE''(k+t\xi)-\sE''(0))\xi^2\,dt.
\]
Now, $\sE''(0)\xi^2 = \langle \xi,\sM'(0)\xi\rangle_{\sX\times\sX^*} = \langle \xi,\sM_1(0)\xi\rangle_{\sG\times\sG^*}$ (using the continuous embeddings, $\sX \subset \sG$ and $\sH \subset \sG^*$, the latter with norm $\kappa \in [1,\infty)$). Therefore,
\begin{align*}
|\sE''(0)\xi^2| &= \left|\langle \xi,\sM_1(0)\xi\rangle_{\sG\times\sG^*}\right|
\\
                &\leq \|\xi\|_\sG\|\sM_1(0)\xi\|_{\sG^*}
  \\
  &\leq \kappa\|\xi\|_\sG\|\sM_1(0)\xi\|_\sH
\\
&\leq \kappa\|\sM_1(0)\|_{\sL(\sG,\sH)}\|\xi\|_\sG^2 =: \frac{1}{2}C_1\|\xi\|_\sG^2.
\end{align*}
Similarly, $\sE''(k+t\xi)\xi^2 = \langle \xi,\sM'(k+t\xi)\xi\rangle_{\sX\times\sX^*} = \langle \xi, \sM_1(k+t\xi)\xi\rangle_{\sG\times\sG^*}$ and
\begin{align*}
\left|\int_{0}^{1}(\sE''(k+t\xi)-\sE''(0))\xi^2\,dt\right|
&=
\left|\int_{0}^{1}\langle \xi,(\sM_1(k+t\xi)-\sM_1(0))\xi\rangle_{\sG\times\sG^*}\,dt\right|
\\
&\leq \|\xi\|_\sG \max_{t\in[0,1]}\|(\sM_1(k+t\xi)-\sM_1(0))\xi\|_{\sG^*}
\\
&\leq \kappa\|\xi\|_\sG \max_{t\in[0,1]}\|(\sM_1(k+t\xi)-\sM_1(0))\xi\|_\sH
\\
&\leq \kappa\|\xi\|_\sG^2 \max_{t\in[0,1]}\|\sM_1(k+t\xi)-\sM_1(0)\|_{\sL(\sG,\sH)}
\\
&\leq \kappa\eps\|\xi\|_\sG^2, \quad\text{for } \|\xi\|_\sX, \|k\|_\sX \leq \delta.
\end{align*}
Consequently, choosing $\eps \in (0,1]$ so that $\kappa\eps \leq \frac{1}{2}C_1$, we obtain
\begin{equation}
\label{eq:Energy_upper_bound}
|\sE(\xi+k)| \leq C_1\|\xi\|_\sG^2, \quad\forall\, \xi+k\in\sX \text{ such that } \|\xi\|_\sX, \|k\|_\sX \leq \delta.
\end{equation}
Combining \eqref{eq:Gradient_lower_bound} and \eqref{eq:Energy_upper_bound} yields
\[
\|\sM(x)\|_\sH \geq Z|\sE(x)|^{1/2}, \quad\forall\, x=\xi+k\in\sX \text{ such that } \|\xi\|_\sX, \|k\|_\sX \leq \delta,
\]
for $Z := C_0/\sqrt{C_1}$. This completes the proof of Theorem \ref{mainthm:Lojasiewicz-Simon_gradient_inequality_Morse-Bott}.
\end{proof}

\section{Morse--Bott property of Yang--Mills energy functions}
\label{sec:Morse-Bott_property_Yang-Mills_energy_functions}
In our articles \cite{Feehan_Maridakis_Lojasiewicz-Simon_Banach, Feehan_Maridakis_Lojasiewicz-Simon_coupled_Yang-Mills_v6} with Maridakis we only gave a few examples where the energy functions $\sE$ were known to be Morse--Bott. In this section, we provide two criteria for when Yang--Mills energy functions are Morse--Bott. Those criteria are simplest in the case of the self-dual Yang--Mills energy function near anti-self-dual connections over four-dimensional manifolds, which we discuss in Section \ref{subsec:Self-dual_Yang-Mills_energy_function_near_anti-self-dual_connections} (and where we prove Theorem \ref{mainthm:Optimal_Lojasiewicz-Simon_inequalities_self-dual_Yang-Mills_energy_function}), and in the case of the Yang--Mills energy function near flat connections over manifolds of dimension $d \geq 2$, which we discuss in Section \ref{subsec:Yang-Mills_energy_function_near_flat_connections} (and where we prove Theorem \ref{mainthm:Lojasiewicz-Simon_inequalities_Yang-Mills_energy_flat_Morse-Bott}). Finally, in Section \ref{subsec:Yang-Mills_energy_function_near_yang-mills_connections} we give the short proof of Theorem \ref{mainthm:Lojasiewicz-Simon_inequalities_Yang-Mills_energy_Morse-Bott}.

\subsection{Self-dual Yang--Mills energy function near anti-self-dual connections}
\label{subsec:Self-dual_Yang-Mills_energy_function_near_anti-self-dual_connections}
In this subsection, we assume that $(X,g)$ is a closed, \emph{four}-dimensional, smooth Riemannian manifold and that, as usual, $G$ is a compact Lie group and $P$ is a smooth principal $G$-bundle over $X$. The self-dual Yang--Mills energy function, $\YM_+:\sA(P)\to\RR$ in \eqref{eq:Self-dual_Yang-Mills_energy_function}, has Hessian map, $\YM_+'':\sA(P)\to T^*\sA(P)\times T^*\sA(P)$, given by
\begin{equation}
\label{eq:Hessian_self-dual_Yang-Mills_energy_function}
\YM_+''(A)(a,b) = (d_A^+a,d_A^+b)_{L^2(X)} + (F_A^+,a\wedge b)_{L^2(X)},
\end{equation}
for all $a, b \in T_A\sA(P) = W^{1,q}(X;T^*X\otimes\ad P)$.

\begin{lem}[Morse--Bott property of the self-dual Yang--Mills energy function at regular anti-self-dual connections]
\label{lem:Morse-Bott_property_self-dual_Yang-Mills_energy}
Let $(X,g)$ be a closed, four-dimensional, smooth Riemannian manifold, $G$ be a compact Lie group, $P$ be a smooth principal $G$-bundle over $X$, and $q>2$ be a constant. If $A$ is a $W^{1,q}$ anti-self-dual Yang--Mills connection on $P$ such that $\Coker d_A^+ = 0$, then $\YM_+:\sA(P) \to \RR$ is a Morse--Bott function at $A$ in the sense of Definition \ref{defn:Morse-Bott_function}. Moreover, if in addition the isotropy group of $A$ in $\Aut(P)$ is the center of $G$, then $\YM:\sB^*(P) \to \RR$ is a Morse--Bott function at $[A]$. 
\end{lem}

\begin{proof}
We first consider $\YM_+:\sA(P) \to \RR$. From Donaldson and Kronheimer \cite[Section 4.2.5]{DK}, the intersection of the subvariety, $\widetilde{M}_+(P,g) = \{B \in \sA(P): F_B^+ = 0\}$, with an open ball $\widetilde{U}_A(\eps) \subset \sA(P)$ with center $A$ and small enough radius $\eps=\eps(A,g) \in (0,1]$, is a smooth manifold if $\Coker d_A^+ = 0$, since the latter property means that $0 \in L^q(X;\wedge^+(T^*X)\otimes\ad P)$ is a regular value of the map $\sA(P) \ni A \mapsto F_A^+ \in L^q(X;\wedge^+(T^*X)\otimes\ad P)$. Moreover, $\Coker d_B^+ = 0$ for small enough $\eps$ and all $B \in \widetilde{U}_A(\eps)$ since the property of $d_A^+$ being surjective is open. Hence, from the discussion in Section \ref{subsec:Optimal_Lojasiewicz-Simon_inequalities_self-dual_Yang-Mills_energy_Morse-Bott},
\[
\widetilde M_+(P,g) \cap \widetilde{U}_A(\eps) = \widetilde{\Crit}\YM_+ \cap\, \widetilde{U}_A(\eps),
\]
and $\widetilde{\Crit}\YM_+ \cap\, \widetilde{U}_A(\eps)$ is a smooth manifold. Because $F_A^+=0$, we have by \eqref{eq:Hessian_self-dual_Yang-Mills_energy_function} that
\[
\YM_+''(A)(a,b) = (d_A^+a,d_A^+b)_{L^2(X)} = (d_A^{+,*}d_A^+a,b)_{L^2(X)}.
\]
On the other hand, the tangent space to $\widetilde{\Crit}\YM_+ \subset \sA(P)$ at $A$ is given by
\[
T_A\widetilde{\Crit}\YM_+ = \Ker \left(d_A^+: W^{1,q}(X;T^*X\otimes\ad P) \to L^q(X;\wedge^+(T^*X)\otimes\ad P)\right).
\]
But then
\begin{align*}
\Ker\YM_+''(A) &= \Ker \left(d_A^{+,*}d_A^+: W^{1,q}(X;T^*X\otimes\ad P) \to W^{-1,q}(X;T^*X\otimes\ad P)\right)
\\
               &= \Ker \left(d_A^+: W^{1,q}(X;T^*X\otimes\ad P) \to L^q(X;\wedge^+(T^*X)\otimes\ad P)\right)
  \\
  &= T_A\widetilde{\Crit}\YM_+,
\end{align*}
and thus $\YM_+:\sA(P) \to \RR$ is a Morse--Bott function at $A$ by Definition \ref{defn:Morse-Bott_function}.

We now consider $\YM_+:\sB^*(P) \to \RR$. The argument here is very similar and again relies on \cite[Section 4.2]{DK} for a description of the manifold structures of $M_+(P,g)$ and $\sB(P)$. We let $U_{[A]}(\eps) \subset \sB^*(P)$ denote the open ball with center $[A]$ and radius $\eps$ and now find that
\[
M_+^*(P,g) \cap U_{[A]}(\eps) = \Crit\YM_+ \cap U_{[A]}(\eps),
\]
and $\Crit\YM_+ \cap U_{[A]}(\eps)$ is a smooth manifold. The tangent space to $\Crit\YM_+ \subset \sB^*(P)$ at $[A]$ is thus given by
\[
T_A\Crit\YM_+ = \Ker \left(d_A^+: \Ker d_A^*\cap W^{1,q}(X;T^*X\otimes\ad P) \to L^q(X;\wedge^+(T^*X)\otimes\ad P)\right).
\]
But then
\begin{align*}
  {}&\Ker\YM_+''(A)
  \\
  &= \Ker \left(d_A^{+,*}d_A^+: \Ker d_A^*\cap W^{1,q}(X;T^*X\otimes\ad P) \to \Ker d_A^*\cap W^{-1,q}(X;T^*X\otimes\ad P)\right)
\\
&=
     \Ker \left(d_A^+: \Ker d_A^*\cap W^{1,q}(X;T^*X\otimes\ad P) \to L^q(X;\wedge^+(T^*X)\otimes\ad P)\right)
  \\
  &= T_{[A]}\Crit\YM_+,
\end{align*}
and thus $\YM_+:\sB^*(P) \to \RR$ is a Morse--Bott function at $[A]$ by Definition \ref{defn:Morse-Bott_function}.
\end{proof}

The {\L}ojasiewicz--Simon gradient inequality \eqref{eq:Lojasiewicz-Simon_gradient_inequality_self-dual_Yang-Mills_energy} in
Theorem \ref{mainthm:Optimal_Lojasiewicz-Simon_inequalities_self-dual_Yang-Mills_energy_function} may be proved as a consequence of the Morse--Bott property of $\YM_+$ and Theorem \ref{mainthm:Lojasiewicz-Simon_gradient_inequality_Morse-Bott} or directly using standard arguments in Yang--Mills gauge theory. We shall provide both arguments.

\begin{proof}[Proof of Inequality \eqref{eq:Lojasiewicz-Simon_gradient_inequality_self-dual_Yang-Mills_energy} using the Morse--Bott property of $\YM_+$ under the condition \eqref{eq:Lojasiewicz-Simon_gradient_inequality_self-dual_Yang-Mills_energy_W12_neighborhood}]
We seek to apply Corollary \ref{maincor:Lojasiewicz-Simon_gradient_inequality_Morse-Bott_YisXdual} with $\sX = W_{A_\infty}^{1,2}(X;T^*X\otimes\ad P)$. Our \cite[Proposition 3.1.1]{Feehan_Maridakis_Lojasiewicz-Simon_coupled_Yang-Mills_v6}, giving analyticity of the boson coupled Yang--Mills energy function carries over \mutatis (when $d=4$ and $p=2$) for the self-dual Yang--Mills energy function and, indeed, is easier since $X$ is restricted to have dimension $d=4$ and the structure of the energy function is much simpler. Hence, the map
\[
  \YM_+:A_\infty + W_{A_\infty}^{1,2}(X;T^*X\otimes\ad P) \to \RR
\]
is at least $C^2$. Moreover, when $\Coker d_{A_\infty}^+ = 0$, we verified that $\YM_+$ has the Morse--Bott property (in the sense of Definition \ref{defn:Morse-Bott_function}) at $A_\infty$ in Lemma \ref{lem:Morse-Bott_property_self-dual_Yang-Mills_energy}. Therefore, Inequality
\eqref{eq:Lojasiewicz-Simon_gradient_inequality_self-dual_Yang-Mills_energy} now follows from Corollary \ref{maincor:Lojasiewicz-Simon_gradient_inequality_Morse-Bott_YisXdual}, with one caveat: In order to apply Corollary \ref{maincor:Lojasiewicz-Simon_gradient_inequality_Morse-Bott_YisXdual}, we must strengthen the hypothesis \eqref{eq:Lojasiewicz-Simon_gradient_inequality_self-dual_Yang-Mills_energy_L4_neighborhood} to
\begin{equation}
\label{eq:Lojasiewicz-Simon_gradient_inequality_self-dual_Yang-Mills_energy_W12_neighborhood}
\|A - A_\infty\|_{W_{A_\infty}^{1,2}(X)} < \sigma,
\end{equation}
corresponding to the {\L}ojasiewicz--Simon neigborhood condition \eqref{eq:Lojasiewicz-Simon_gradient_inequality_neighborhood_Morse-Bott_YisXdual} in Corollary \ref{maincor:Lojasiewicz-Simon_gradient_inequality_Morse-Bott_YisXdual}. 
\end{proof}

\begin{proof}[Proof of Theorem \ref{mainthm:Optimal_Lojasiewicz-Simon_inequalities_self-dual_Yang-Mills_energy_function}, including direct proof of Inequality \eqref{eq:Lojasiewicz-Simon_gradient_inequality_self-dual_Yang-Mills_energy}]
The first and final assertions regarding the Morse--Bott properties of $\YM_+:\sA(P)\to\RR$ and $\YM_+:\sB^*(P)\to\RR$ both follow from Lemma \ref{lem:Morse-Bott_property_self-dual_Yang-Mills_energy}.

In the remainder of the proof, we may assume without loss of generality that $A_\infty$ is a $C^\infty$ connection by choosing a $W^{2,q}$ gauge transformation $u \in \Aut(P)$ such that $u(A_\infty)$ is a $C^\infty$ anti-self-dual connection. To see this, we observe that the hypothesis \eqref{eq:Lojasiewicz-Simon_gradient_inequality_self-dual_Yang-Mills_energy_L4_neighborhood} is equivalent to
\[
\|u(A) - u(A_\infty)\|_{L^4(X)} < \sigma
\]
and the inequalities \eqref{eq:Lojasiewicz-Simon_distance_inequality_self-dual_Yang-Mills_energy},
\eqref{eq:Lojasiewicz-Simon_gradient_inequality_self-dual_Yang-Mills_energy} are also equivalent to their analogues with $u(A)$ and $u(A_\infty)$. The existence of $u$ follows from standard arguments; see Uhlenbeck \cite[p. 33]{UhlLp} or Wehrheim \cite[Theorem 9.4 (i)]{Wehrheim_2004}.

Because $F_{A_\infty}^+ = 0$, we have an elliptic complex \cite[Equation (4.2.26)]{DK},
\[
\Omega^0(X;\ad P) \xrightarrow{d_{A_\infty}} \Omega^1(X;\ad P) \xrightarrow{d_{A_\infty}^+} \Omega^{2,+}(X;\ad P)
\]
and an $L^2$-orthogonal Hodge decomposition \cite[Theorem 1.5.2]{Gilkey2}
\[
W_{A_\infty}^{1,q}(X;T^*X\otimes\ad P)
=
\Ker \left(d_{A_\infty}^+ + d_{A_\infty}^*\right)
\oplus
\Ran d_{A_\infty}
\oplus
\Ran d_{A_\infty}^{+,*}.
\]
Note that $\Ran d_{A_\infty} \subset \Ker d_{A_\infty}^+$. We now write $A = A_\infty + a$ for $a \in W_{A_\infty}^{1,q}(X;T^*X\otimes\ad P)$ and split $a = a_\perp + a_\parallel$, where $a_\perp, a_\parallel \in W_{A_\infty}^{1,q}(X;T^*X\otimes\ad P)$ and $a_\perp$ is $L^2$-orthogonal to $\Ker d_{A_\infty}^+$ while $a_\parallel \in \Ker d_{A_\infty}^+$.

We first consider the case where $a_\parallel = 0$ and observe that $a = a_\perp = d_{A_\infty}^{+,*}v$ for
\[
  v \in W_{A_\infty}^{2,q}(X;\wedge^2(T^*X)\otimes\ad P)
\]
by the Hodge decomposition. Because $F_{A_\infty}^+ = 0$, we have
\begin{equation}
\label{eq:ASD_near_Ainfty}
F_A^+ = F_{A_\infty + a}^+ = d_{A_\infty}^+a + (a\wedge a)^+.
\end{equation}
We claim that $a$ obeys the following \apriori estimate, with $p \in (1,\infty)$ obeying $p\leq q$ and a constant $C=C(A_\infty,g,G,p)\in[1,\infty)$:
\begin{equation}
\label{eq:Apriori_W1p_estimate_a_dAinfty+a}
\|a\|_{W_{A_\infty}^{1,p}(X)} \leq C\|d_{A_\infty}^+ a\|_{L^p(X)}.
\end{equation}
To see this, we observe that
\[
\|d_{A_\infty}^{+,*}v\|_{W_{A_\infty}^{1,p}(X)}
\leq
c\|v\|_{W_{A_\infty}^{2,p}(X)}
\leq
C\|d_{A_\infty}^+d_{A_\infty}^{+,*}v\|_{L^p(X)}
\]
for constants $c=c(g,G)$ and $C=C(A_\infty,g,G,p)$ in $[1,\infty)$. Ellipticity of the second-order operator $d_{A_\infty}^+d_{A_\infty}^{+,*}$ follows from its Bochner--Weitzenb\"ock formula \cite[Equation (6.26)]{FU}, as that implies that its principal symbol coincides with that of the covariant Laplace operator $\nabla_{A_\infty}^*\nabla_{A_\infty}$ and thus a scalar multiple (the Riemannian metric on $T^*X$) of the identity. The \apriori $W^{2,p}$ elliptic estimate for $v$ follows from \cite[Theorem 9.14]{GilbargTrudinger} or \cite[Theorem 14.60]{Feehan_yang_mills_gradient_flow_v4} for $d_{A_\infty}^+d_{A_\infty}^{+,*}$ and an argument exactly analogous to the proof of \cite[Lemma 9.17]{GilbargTrudinger} to eliminate the term $\|v\|_{L^p(X)}$ from the right-hand side. Hence, the claim \eqref{eq:Apriori_W1p_estimate_a_dAinfty+a} follows.

Because $1/p = 1/p^*+1/4$ with $p^* = 4p/(4-p) \in (4,\infty)$, we have
\[
\|(a\wedge a)^+\|_{L^p(X)} \leq c\|a\|_{L^{p^*}(X)}\|a\|_{L^4(X)}  \leq C\|a\|_{W_{A_\infty}^{1,p}(X)}\|a\|_{L^4(X)},
\]
for a constant $c=c(g,G)\in[1,\infty)$ and $C=C(g,G,p) \in [1,\infty)$. Consequently,
\begin{align*}
\|a\|_{W_{A_\infty}^{1,p}(X)} &\leq C\|d_{A_\infty}^+ a\|_{L^p(X)}
\quad\text{(by \eqref{eq:Apriori_W1p_estimate_a_dAinfty+a})}
\\
&\leq C\|F_A^+\|_{L^p(X)} + C\|(a\wedge a)^+\|_{L^p(X)}
\quad\text{(by \eqref{eq:ASD_near_Ainfty})}
\\
&\leq C\|F_A^+\|_{L^p(X)} + C\|a\|_{W_{A_\infty}^{1,p}(X)}\|a\|_{L^4(X)}.
\end{align*}
Since $\|a\|_{L^4(X)} < \sigma$ by \eqref{eq:Lojasiewicz-Simon_gradient_inequality_self-dual_Yang-Mills_energy_L4_neighborhood}, then rearrangement, for small enough $\sigma = \sigma(A_\infty,g,G,p) \in (0,1]$, yields
\begin{equation}
\label{eq:W1p_a_leq_LpFA+_when_a_perp_ker_dA+}
\|a\|_{W_{A_\infty}^{1,p}(X)} \leq C\|F_A^+\|_{L^p(X)},
\end{equation}
and thus for $p=2$ we obtain \eqref{eq:Lojasiewicz-Simon_distance_inequality_self-dual_Yang-Mills_energy}.

To prove \eqref{eq:Lojasiewicz-Simon_gradient_inequality_self-dual_Yang-Mills_energy}, write $d_A^+a = d_{A_\infty}^+a + 2(a\wedge a)^+ = F_A^+ + (a\wedge a)^+$ and note that
\begin{align*}
\|d_A^{+,*}F_A^+\|_{W_{A_\infty}^{-1,2}(X)}
&=
\sup_{b \in W_{A_\infty}^{1,2}(X;T^*X\otimes \ad P) \less \{0\}}
\frac{(d_A^{+,*}F_A^+, b)_{L^2(X)}}{\|b\|_{W_{A_\infty}^{1,2}(X)}}
\\
&\geq \frac{(d_A^{+,*}F_A^+, a)_{L^2(X)}}{\|a\|_{W_{A_\infty}^{1,2}(X)}}
=
\frac{(F_A^+, d_A^+a)_{L^2(X)}}{\|a\|_{W_{A_\infty}^{1,2}(X)}}
=
\frac{(F_A^+, F_A^+ + (a\wedge a)^+)_{L^2(X)}}{\|a\|_{W_{A_\infty}^{1,2}(X)}}.
\end{align*}
Therefore,
\begin{equation}
\label{eq:W12dual_norm_Gradient_self-dual_Yang-Mills_geq_pre-energy}
\|d_A^{+,*}F_A^+\|_{W_{A_\infty}^{-1,2}(X)}
\geq \frac{\|F_A^+\|_{L^2(X)}^2}{\|a\|_{W_{A_\infty}^{1,2}(X)}} +
\frac{(F_A^+,(a\wedge a)^+)_{L^2(X)}}{\|a\|_{W_{A_\infty}^{1,2}(X)}}.
\end{equation}
The gradient inequality \eqref{eq:Lojasiewicz-Simon_gradient_inequality_self-dual_Yang-Mills_energy} now follows. Indeed,
\[
\|(a \wedge a)^+\|_{L^2(X)} \leq c\|a\|_{L^4(X)}^2 \leq C\|a\|_{L^4(X)}\|a\|_{W_{A_\infty}^{1,2}(X)},
\]
for constants $c$ and $C$ with the same dependencies as above, and
\begin{align*}
\|d_A^{+,*}F_A^+\|_{W_{A_\infty}^{-1,2}(X)}
&\geq
\frac{\|F_A^+\|_{L^2(X)}^2}{\|a\|_{W_{A_\infty}^{1,2}(X)}}
- \frac{\|F_A^+\|_{L^2(X)}\|(a \wedge a)^+\|_{L^2(X)} }{\|a\|_{W_{A_\infty}^{1,2}(X)}}
\quad\text{(by \eqref{eq:W12dual_norm_Gradient_self-dual_Yang-Mills_geq_pre-energy})}
\\
&\geq
\frac{\|F_A^+\|_{L^2(X)}^2}{\|a\|_{W_{A_\infty}^{1,2}(X)}}
- C\frac{\|F_A^+\|_{L^2(X)}\|a\|_{L^4(X)}\|a\|_{W_{A_\infty}^{1,2}(X)} }{\|a\|_{W_{A_\infty}^{1,2}(X)}}
\\
&\geq
C^{-1}\|F_A^+\|_{L^2(X)} - C\sigma\|F_A^+\|_{L^2(X)}
\quad\text{(by \eqref{eq:Lojasiewicz-Simon_gradient_inequality_self-dual_Yang-Mills_energy_L4_neighborhood}
and \eqref{eq:Lojasiewicz-Simon_distance_inequality_self-dual_Yang-Mills_energy})}.
\end{align*}
Now choose $\sigma$ small enough that $\sigma \leq 1/(2C^2)$ to give \eqref{eq:Lojasiewicz-Simon_gradient_inequality_self-dual_Yang-Mills_energy}. This completes the proof of the optimal {\L}ojasiewicz--Simon inequalities when $a_\parallel = 0$.

When $a_\parallel \neq 0$, we instead choose a $W^{1,q}$ anti-self-dual connection $\tilde A_\infty$ on $P$ such that $A = \tilde A_\infty + \tilde a$, where $\tilde a \in W_{A_\infty}^{1,q}(X;T^*X\otimes\ad P)$ is $L^2$-orthogonal to $\Ker d_{\tilde A_\infty}^+$ and obeys $\|\tilde a\|_{L^4(X)} < 2\sigma$. The existence of $\tilde a$ follows because an open neighborhood of $A_\infty$ in $\widetilde M_+(P,g) \subset \sA(P)$ is a smooth submanifold by our hypothesis that $\Coker d_{A_\infty}^+ = 0$ and so has an $L^2$-normal tubular neighborhood in $\sA(P)$ (compare \cite[Theorem 4.5.2]{Hirsch} in the case of finite-dimensional manifolds). To see this explicitly, we note that by \cite[Section 4.2.5]{DK} for small enough $\sigma=\sigma(A_\infty,g,G) \in (0,1]$,
\[
\sU := \{b \in W_{A_\infty}^{1,q}(X;T^*X\otimes\ad P): F_{A_\infty + b}^+ = 0 \text{ and } \|b\|_{L^4(X)} < \sigma\}
\]
is an open, smooth submanifold of $W_{A_\infty}^{1,q}(X;T^*X\otimes\ad P)$, with
\[
T_b := \Ker d_{A_\infty+b}^+ \cap W_{A_\infty}^{1,q}(X;T^*X\otimes\ad P),
\]
as tangent space at $b$ and smooth normal bundle, $\sN$, with fiber over $b$,
\[
N_b := \left(\Ker d_{A_\infty+b}^+\right)^\perp \cap W_{A_\infty}^{1,q}(X;T^*X\otimes\ad P),
\]
where $(\Ker d_{A_\infty+b}^+)^\perp$ is the $L^2$-orthogonal complement of $\Ker d_{A_\infty+b}^+ \cap W_{A_\infty}^{1,q}(X;T^*X\otimes\ad P)$. The differential of the smooth map,
\[
\sN \ni (b,\eta) \mapsto b+\eta \in W_{A_\infty}^{1,q}(X;T^*X\otimes\ad P) = T_0\oplus N_0,
\]
is the identity at the origin $(0,0)$ and so the existence of an $L^2$-normal tubular neighborhood now follows from the Implicit Function Theorem for smooth maps on Banach spaces. Because $F_{\tilde A_\infty}^+ = 0$, we have
\[
F_A^+ = F_{\tilde A_\infty + \tilde a}^+ = d_{\tilde A_\infty}^+\tilde a + (\tilde a\wedge \tilde a)^+,
\]
and so the inequalities \eqref{eq:Lojasiewicz-Simon_distance_inequality_self-dual_Yang-Mills_energy} and \eqref{eq:Lojasiewicz-Simon_gradient_inequality_self-dual_Yang-Mills_energy} now follow almost exactly as before, noting that $\|\tilde A_\infty - A_\infty\|_{L^4(X)} < \sigma$. This completes the proof of Theorem \ref{mainthm:Optimal_Lojasiewicz-Simon_inequalities_self-dual_Yang-Mills_energy_function}.
\end{proof}

\subsection{Yang--Mills energy function near flat connections}
\label{subsec:Yang-Mills_energy_function_near_flat_connections}
We shall proceed by analogy with our development in Section \ref{subsec:Self-dual_Yang-Mills_energy_function_near_anti-self-dual_connections} but return to the general case where $X$ may have any dimension $d \geq 2$. If $F_A=0$, then $\YM'(A) \equiv 0$ by \eqref{eq:Gradient_Yang-Mills_energy_function} and $A$ is a critical point of $\YM:\sA(P) \to \RR$, so that
\[
\widetilde M_0(P,g) \subset \widetilde{\Crit}\YM \cap \sA(P),
\]
where $\widetilde{\Crit}\YM$ denotes the critical set of $\YM:\sA(P) \to \RR$. Conversely, suppose $A \in \widetilde{\Crit}\YM$. The Bianchi Identity \cite[Equation (2.1.21)]{DK} implies that $d_AF_A = 0$, so $F_A \in \Ker d_A \cap L^q(X;\wedge^2(T^*X)\otimes\ad P)$ and if $A$ is a \emph{regular point} of the map $\sA(P) \ni A \mapsto F_A \in L^q(X;\wedge^2(T^*X)\otimes\ad P)$ in the sense that
\[
\Ker d_A \cap L^q(X;\wedge^2(T^*X)\otimes\ad P)
=
\Ran d_A \cap L^q(X;\wedge^2(T^*X)\otimes\ad P),
\]
then \eqref{eq:Gradient_Yang-Mills_energy_function} implies that $F_A=0$ and $A \in \widetilde M_0(P,g)$. Of course, in the absence of an assumption that $A$ is regular in the preceding sense, then $A$ is (by definition) a \emph{Yang--Mills connection} as in \eqref{eq:Yang_Mills},
\[
d_A^*F_A = 0,
\]
and of course need not be flat. However, if we require in addition to \eqref{eq:Yang_Mills} that
\begin{equation}
\label{eq:Curvature_Ldover2_small}
\|F_A\|_{L^{d/2}(X)} \leq \eps,
\end{equation}
for $\eps = \eps(g,G) \in (0,1]$, then $A$ is necessarily flat by Feehan \cite[Theorem 1]{Feehan_yangmillsenergygapflat_aim}, \cite{Feehan_yangmillsenergygapflat_corrigendum} and thus we obtain the reverse inclusion,
\[
\widetilde{\Crit}\YM \cap \sA_\eps(P) \subset \widetilde M_0(P,g),
\]
where $\sA_\eps(P) := \{A \in \sA(P): A \text{ obeys } \eqref{eq:Curvature_Ldover2_small}\}$. 

If $\Gamma$ is a flat connection on $P$, then its exterior covariant derivative defines an elliptic complex,
\[
\cdots \Omega^i(X;\ad P) \xrightarrow{d_\Gamma} \Omega^{i+1}(X;\ad P) \xrightarrow{d_\Gamma} \Omega^{i+2}(X;\ad P) \cdots
\]
for $i \geq 0$, since $d_\Gamma^2 = F_\Gamma = 0$. By analogy with their definitions based on the deformation complex for an anti-self-dual connection \cite[Section 4.2.5]{DK} on a principal $G$-bundle $P$ over a four-dimensional Riemannian manifold, one defines
\[
H_\Gamma^i(X;\ad P) := \Ker d_\Gamma \cap \Omega^i(X;\ad P)/\Ran d_\Gamma, \quad i \geq 0.
\]
By analogy with the construction in \cite[Section 4.2.5]{DK} of a local Kuranishi model for an open neighborhood of a point $[A] \in M_+(P,g) \subset \sB(P)$ when $X$ has dimension four, we observe that if $H_\Gamma^2(X;\ad P) = 0$, then there is an open neighborhood $\tilde\sU_\Gamma \subset \sA(P)$ of a flat connection $\Gamma$ on $P$ such that
\[
\tilde\sU_\Gamma \cap \widetilde M_0(P) \subset \sA(P)
\]
is an open, smooth submanifold. (See Ho, Wilkin, and Wu \cite[Proposition 2.4]{Ho_Wilkin_Wu_2019} for a detailed proof.) Moreover, if the isotropy group of $\Gamma$ in $\Aut(P)$ is the center of $G$, then the quotient,
\[
\sU_\Gamma \cap M_0(P) \subset \sB^*(P),
\]
is an open, smooth submanifold. In general, the moduli space $M_0(P)$ will not be a smooth submanifold but rather a finite-dimensional, real analytic subvariety (compare \cite[p. 139]{DK}).

By gauge invariance, the Yang--Mills energy function is well-defined on the quotient, $\YM:\sB^*(P) \to \RR$ (with $q > d/2$ for $d\geq 4$ and $q=2$ for $d=2,3$), and we have the equality,
\[
M_0^*(P) = \Crit\YM \cap \sB_\eps^*(P),
\]
where $\Crit\YM$ denotes the critical set of $\YM:\sB^*(P) \to \RR$, and $\sB_\eps(P) := \{[A] \in \sB(P): A \text{ obeys } \eqref{eq:Curvature_Ldover2_small}\}$, and $\sB_\eps^*(P) := \sB_\eps(P) \cap \sB^*(P)$, and $M_0^*(P) := M_0(P) \cap \sB^*(P)$.

Given the preceding remarks, the proof of Lemma \ref{lem:Morse-Bott_property_self-dual_Yang-Mills_energy} adapts\footnote{But see Feehan \cite[Lemma A.5]{Feehan_nonlinear_uhlenbeck_estimate} for a detailed proof.} \mutatis to give the

\begin{lem}[Morse--Bott property of the Yang--Mills energy function at regular flat connections]
\label{lem:Morse-Bott_property_Yang-Mills_energy_near_flat_connection}
Let $(X,g)$ be a closed, smooth Riemannian manifold of dimension $d\geq 2$, and $G$ be a compact Lie group, $P$ be a smooth principal $G$-bundle over $X$, and $q > d/2$ for $d\geq 4$ and $q=2$ for $d=2,3$. If $\Gamma$ is a $W^{1,q}$ flat connection on $P$ such that $H_\Gamma^2(X;\ad P) = 0$, then $\YM:\sA(P) \to \RR$ is a Morse--Bott function at $\Gamma$ in the sense of Definition \ref{defn:Morse-Bott_function}. Moreover, if in addition the isotropy group of $\Gamma$ in $\Aut(P)$ is the center of $G$, then $\YM:\sB^*(P) \to \RR$ is a Morse--Bott function at $[\Gamma]$.
\end{lem}

When $H_\Gamma^2(X;\ad P)=0$, we shall prove the {\L}ojasiewicz--Simon gradient inequality \eqref{eq:Lojasiewicz-Simon_gradient_inequality_Yang-Mills_energy_flat_local} in
Theorem \ref{mainthm:Lojasiewicz-Simon_inequalities_Yang-Mills_energy_flat_Morse-Bott} using the Morse--Bott property of $\YM$ at $\Gamma$ from Lemma \ref{lem:Morse-Bott_property_Yang-Mills_energy_near_flat_connection} and Theorem \ref{mainthm:Lojasiewicz-Simon_gradient_inequality_Morse-Bott}. We shall also give a direct proof of \eqref{eq:Lojasiewicz-Simon_gradient_inequality_Yang-Mills_energy_flat_local} using arguments in Yang--Mills gauge theory. To verify the preceding results, we outline the modifications required to the corresponding proofs in Section \ref{subsec:Self-dual_Yang-Mills_energy_function_near_anti-self-dual_connections} for $\YM_+$ when $X$ has dimension four.

\begin{proof}[Proof of Inequality \eqref{eq:Lojasiewicz-Simon_gradient_inequality_Yang-Mills_energy_flat_local} using the Morse--Bott property of $\YM$ under the condition \eqref{eq:Lojasiewicz-Simon_gradient_inequality_flat_Yang-Mills_energy_W1p_neighborhood}]
We\hfill\break shall apply Theorem \ref{mainthm:Lojasiewicz-Simon_gradient_inequality_Morse-Bott} with
\begin{multline*}
  \sX = W_\Gamma^{1,p}(X;T^*X\otimes\ad P), \quad \sG = W_\Gamma^{1,2}(X;T^*X\otimes\ad P),
  \\
  \sY = W_\Gamma^{-1,p}(X;T^*X\otimes\ad P), \quad \sH = W_\Gamma^{-1,2}(X;T^*X\otimes\ad P).
\end{multline*}
Thus, $\sH=\sG^*$ and $\sG^* \subset \sX^*$. Our \cite[Proposition 3.1.1]{Feehan_Maridakis_Lojasiewicz-Simon_coupled_Yang-Mills_v6}, giving analyticity of the boson coupled Yang--Mills energy function (when $d\geq 2$ and $p\in(d/2,\infty)$ obeys $p\geq 2$) implies that the Yang--Mills energy function 
\[
  \YM:\Gamma + W_\Gamma^{1,p}(X;T^*X\otimes\ad P) \to \RR
\]
is at least $C^2$. Moreover, the gradient map
\[
  \YM':\Gamma + W_\Gamma^{1,p}(X;T^*X\otimes\ad P) \to W_\Gamma^{-1,p}(X;T^*X\otimes\ad P)
\]
is at least $C^1$ and by \cite[Lemma 4.1.1]{Feehan_Maridakis_Lojasiewicz-Simon_coupled_Yang-Mills_v6} the Hessian operator
\[
  \YM''(A) \in \sL\left( W_\Gamma^{1,p}(X;T^*X\otimes\ad P), W_\Gamma^{-1,p}(X;T^*X\otimes\ad P) \right)
\]
for each $A \in \Gamma + W_\Gamma^{1,p}(X;T^*X\otimes\ad P)$ has a bounded extension
\[
  \sM_1(A) \in \sL\left( W_\Gamma^{1,2}(X;T^*X\otimes\ad P), W_\Gamma^{-1,2}(X;T^*X\otimes\ad P) \right)
\]
such that the following map is continuous:
\begin{multline*}
  \Gamma + W_\Gamma^{1,p}(X;T^*X\otimes\ad P) \ni A
  \\
  \mapsto \sM_1(A) \in \sL\left( W_\Gamma^{1,2}(X;T^*X\otimes\ad P), W_\Gamma^{-1,2}(X;T^*X\otimes\ad P) \right).
\end{multline*}
Since $H_\Gamma^2(X;\ad P)=0$, then $\YM$ has the Morse--Bott property (in the sense of Definition \ref{defn:Morse-Bott_function}) at $\Gamma$ by Lemma \ref{lem:Morse-Bott_property_Yang-Mills_energy_near_flat_connection}. Therefore, Inequality
\eqref{eq:Lojasiewicz-Simon_gradient_inequality_Yang-Mills_energy_flat_local} now follows from Theorem \ref{mainthm:Lojasiewicz-Simon_gradient_inequality_Morse-Bott}, with one caveat: In order to apply Theorem \ref{mainthm:Lojasiewicz-Simon_gradient_inequality_Morse-Bott}, we must strengthen the hypothesis \eqref{eq:Lojasiewicz-Simon_gradient_inequality_Yang-Mills_energy_flat_nbhd} to
\begin{equation}
\label{eq:Lojasiewicz-Simon_gradient_inequality_flat_Yang-Mills_energy_W1p_neighborhood}
\|A - \Gamma\|_{W_\Gamma^{1,p}(X)} < \sigma,
\end{equation}
corresponding to the {\L}ojasiewicz--Simon neigborhood condition \eqref{eq:Lojasiewicz-Simon_gradient_inequality_neighborhood_Morse-Bott} in Theorem \ref{mainthm:Lojasiewicz-Simon_gradient_inequality_Morse-Bott}. 
\end{proof}

\begin{proof}[Proof of Theorem \ref{mainthm:Lojasiewicz-Simon_inequalities_Yang-Mills_energy_flat_Morse-Bott}, including direct proof of Inequality \eqref{eq:Lojasiewicz-Simon_gradient_inequality_Yang-Mills_energy_flat_local}]
The first and final assertions regarding the Morse--Bott properties of $\YM:\sA(P)\to\RR$ and $\YM:\sB^*(P)\to\RR$ both follow from Lemma \ref{lem:Morse-Bott_property_Yang-Mills_energy_near_flat_connection}. For the remainder of the proof, we highlight the modifications required to the proof of Theorem \ref{mainthm:Optimal_Lojasiewicz-Simon_inequalities_self-dual_Yang-Mills_energy_function}. 

As before, we may assume without loss of generality that $\Gamma$ is a $C^\infty$ connection by choosing a $W^{2,q}$ gauge transformation $u \in \Aut(P)$ such that $u(\Gamma)$ is a $C^\infty$ flat connection. Similarly, we have an $L^2$-orthogonal Hodge decomposition \cite[Theorem 1.5.2]{Gilkey2},
\[
W_\Gamma^{1,q}(X;T^*X\otimes\ad P)
=
\Ker \left(d_\Gamma + d_\Gamma^*\right)
\oplus
\Ran d_\Gamma
\oplus
\Ran d_\Gamma^*.
\]
Note that $\Ran d_\Gamma \subset \Ker d_\Gamma$ and write $A = \Gamma + a$ for $a \in W_\Gamma^{1,q}(X;T^*X\otimes\ad P)$ and split $a = a_\perp + a_\parallel$, where $a_\perp, a_\parallel \in W_\Gamma^{1,q}(X;T^*X\otimes\ad P)$ and $a_\perp$ is $L^2$-orthogonal to $\Ker d_\Gamma$ while $a_\parallel \in \Ker d_\Gamma$.

We first consider the case where $a_\parallel = 0$ and observe that $a = a_\perp = d_\Gamma^*v$ for $v \in W_\Gamma^{2,q}(X;\wedge^2(T^*X)\otimes\ad P)$ by the Hodge decomposition. Because $F_\Gamma = 0$, we have
\begin{equation}
\label{eq:Flat_near_Gamma}
F_A = F_{\Gamma + a} = d_\Gamma a + a\wedge a.
\end{equation}
The proof of \eqref{eq:Apriori_W1p_estimate_a_dAinfty+a} carries over without change to show that $a$ obeys the following \apriori estimate, with $p \in (1,\infty)$ obeying $p\leq q$ and a constant $C=C(\Gamma,g,G,p)\in[1,\infty)$:
\begin{equation}
\label{eq:Apriori_W1p_estimate_a_dGammaa}
\|a\|_{W_\Gamma^{1,p}(X)} \leq C\|d_\Gamma a\|_{L^p(X)}.
\end{equation}
Moreover, the proof of \eqref{eq:W1p_a_leq_LpFA+_when_a_perp_ker_dA+} adapts to show that, for small enough $\sigma = \sigma(g,G,p,\Gamma) \in (0,1]$,
\begin{equation}
\label{eq:W1p_a_leq_LpFA_when_a_perp_ker_dA}
\|a\|_{W_\Gamma^{1,p}(X)} \leq C\|F_A\|_{L^p(X)},
\end{equation}
where $C = G(g,G,p,\Gamma) \in [1,\infty)$ and for $p \in (1,d)$ or $p=2$ when $d=2$ obeying $p\leq q$.

The only change in the proof of \eqref{eq:W1p_a_leq_LpFA+_when_a_perp_ker_dA+} is that we now use the continuous Sobolev multiplication $L^d(X)\times L^{p^*}(X) \to L^p(X)$ and continuous Sobolev embedding $W^{1,p}(X) \subset L^{p^*}(X)$, for $p\in(1,d)$ and $p^* = dp/(d-p) \in (d,\infty)$, to estimate,
\[
\|a\wedge a\|_{L^p(X)} \leq C\|a\|_{L^d(X)}\|a\|_{W_\Gamma^{1,p}(X)}.
\]
For $d=2$ and $p=2$, which is excluded by the preceding requirement that $p\in(1,d)$, we recall that $r_0>2$ and choose $t_0\in (2,\infty)$ by writing $1/2=1/r_0+1/t_0$ and use the continuous Sobolev multiplication $L^{r_0}(X)\times L^{t_0}(X) \to L^2(X)$ and continuous Sobolev embedding $W^{1,p}(X) \subset L^{t_0}(X)$ to estimate
\[
\|a\wedge a\|_{L^2(X)} \leq C\|a\|_{L^{r_0}(X)}\|a\|_{W_\Gamma^{1,2}(X)}.
\]
For all $d \geq 2$, we thus obtain \eqref{eq:Apriori_W1p_estimate_a_dGammaa}, now using the condition \eqref{eq:Lojasiewicz-Simon_gradient_inequality_Yang-Mills_energy_flat_nbhd} in place of the condition \eqref{eq:Lojasiewicz-Simon_gradient_inequality_self-dual_Yang-Mills_energy_L4_neighborhood} used to obtain \eqref{eq:W1p_a_leq_LpFA+_when_a_perp_ker_dA+}.

By choosing $p=2$ in \eqref{eq:W1p_a_leq_LpFA_when_a_perp_ker_dA} we obtain \eqref{eq:Lojasiewicz-Simon_distance_inequality_Yang-Mills_energy_flat_local}. To establish \eqref{eq:Lojasiewicz-Simon_gradient_inequality_Yang-Mills_energy_flat_local}, we write $d_Aa = d_\Gamma a + 2a\wedge a = F_A + a\wedge a$ and adapt the argument in the proof of Theorem \ref{mainthm:Optimal_Lojasiewicz-Simon_inequalities_self-dual_Yang-Mills_energy_function} used to prove \eqref{eq:Lojasiewicz-Simon_gradient_inequality_self-dual_Yang-Mills_energy}. The only significant change is that, for $d \geq 3$, we now use the continuous Sobolev multiplication $L^d(X)\times L^{2^*}(X) \to L^2(X)$ and continuous Sobolev embedding $W^{1,2}(X) \subset L^{2^*}(X)$ for $2^* = 2d/(d-2) \in (d,\infty)$. For $d=2$ and $r_0>2$, we use the continuous Sobolev multiplication $L^{r_0}(X)\times L^{t_0}(X) \to L^2(X)$ and continuous Sobolev embedding $W^{1,2}(X) \subset L^{t_0}(X)$, as discussed above. This completes the proof of the optimal {\L}ojasiewicz--Simon inequalities when $a_\parallel = 0$.

When $a_\parallel \neq 0$, we instead choose a $W^{1,q}$ flat connection $\tilde \Gamma$ on $P$ such that $A = \tilde \Gamma + \tilde a$, where $\tilde a \in W_\Gamma^{1,q}(X;T^*X\otimes\ad P)$ is $L^2$-orthogonal to $\Ker d_{\tilde \Gamma}$ and obeys $\|\tilde a\|_{L^{r_0}(X)} < 2\sigma$. The existence of $\tilde a$ follows because an open neighborhood of $\Gamma$ in $\widetilde M_0(P,g) \subset \sA(P)$ is a smooth submanifold by our hypothesis that $H_\Gamma^2(X;\ad P) = 0$ and so has an $L^2$-normal tubular neighborhood in $\sA(P)$, by the same argument as used in the proof of Theorem \ref{mainthm:Optimal_Lojasiewicz-Simon_inequalities_self-dual_Yang-Mills_energy_function}. In the present context, we recall that
\[
T_\Gamma := \Ker d_\Gamma \cap W_\Gamma^{1,q}(X;T^*X\otimes\ad P)
\]
is the tangent space at $\Gamma$ to $\{A \in \sA(P): F_A = 0\}$ and
\[
N_\Gamma := \left(\Ker d_\Gamma\right)^\perp \cap W_\Gamma^{1,q}(X;T^*X\otimes\ad P)
\]
is the corresponding normal space. Because $F_{\tilde \Gamma} = 0$, we have
\[
F_A = F_{\tilde \Gamma + \tilde a}^+ = d_{\tilde \Gamma}\tilde a + \tilde a\wedge \tilde a,
\]
and so the inequalities \eqref{eq:Lojasiewicz-Simon_distance_inequality_Yang-Mills_energy_flat_local} and \eqref{eq:Lojasiewicz-Simon_gradient_inequality_Yang-Mills_energy_flat_local} now follow almost exactly as before, noting that $\|\tilde \Gamma - \Gamma\|_{L^{r_0}(X)} < \sigma$. This completes the proof of Theorem \ref{mainthm:Lojasiewicz-Simon_inequalities_Yang-Mills_energy_flat_Morse-Bott}.
\end{proof}

\subsection{Yang--Mills energy function near arbitrary critical points}
\label{subsec:Yang-Mills_energy_function_near_yang-mills_connections}
It remains to give the short

\begin{proof}[Proof of Theorem \ref{mainthm:Lojasiewicz-Simon_inequalities_Yang-Mills_energy_Morse-Bott}]
The argument is virtually identical to our proof in Section \ref{subsec:Yang-Mills_energy_function_near_flat_connections} of Inequality \eqref{eq:Lojasiewicz-Simon_gradient_inequality_Yang-Mills_energy_flat_local} in Theorem \ref{mainthm:Lojasiewicz-Simon_inequalities_Yang-Mills_energy_flat_Morse-Bott} using the Morse--Bott property of $\YM$ near a regular flat connection $\Gamma$ under the condition \eqref{eq:Lojasiewicz-Simon_gradient_inequality_flat_Yang-Mills_energy_W1p_neighborhood}. The only difference is that we now assume as a hypothesis that $\YM$ has the Morse--Bott property (in the sense of Definition \ref{defn:Morse-Bott_function}) at the critical point $A_\infty$.
\end{proof}

\subsection{Yang--Mills energy function over Riemann surfaces and arbitrary critical points}
\label{subsec:Yang-Mills_energy_function_near_yang-mills_connections_over_Riemann_surfaces}
We begin with the

\begin{proof}[First proof of Theorem  \ref{mainthm:Lojasiewicz-Simon_inequalities_Yang-Mills_energy_irreducible_Riemann_surface}]
By R\r{a}de \cite[Proposition 7.2]{Rade_1992}, the {\L}ojasiewicz exponent is $1/2$ under the hypotheses of Theorem \ref{mainthm:Lojasiewicz-Simon_inequalities_Yang-Mills_energy_irreducible_Riemann_surface}. Moreover, a general result due to the author \cite[Theorem 2]{Feehan_lojasiewicz_inequality_all_dimensions_morse-bott} asserts that when an analytic function has {\L}ojasiewicz exponent equal to $1/2$ at a critical point and a Hessian operator that is suitably Fredholm at that point, then the analytic function must be Morse--Bott. Indeed, we can apply \cite[Theorem 1]{Feehan_lojasiewicz_inequality_all_dimensions_morse-bott} to the analytic function $f(a) := \YM(A_\infty + a)-\YM(A_\infty)$, where $a$ belongs to an open neighborhood $\sU$ of the origin in the Hilbert space
\[
  \sX := \Ker d_{A_\infty}^* \cap W_{A_\infty}^{1,2}(X;T^*X\otimes\ad P),
\]
noting that the continuous dual space of $\sX$ is given by
\[
  \sX^*
  =
  \left(\Ker d_{A_\infty}^* \cap W_{A_\infty}^{1,2}(X;T^*X\otimes\ad P)\right)^*
  \cong 
  \Ker d_{A_\infty}^* \cap W_{A_\infty}^{-1,2}(X;T^*X\otimes\ad P).
\]
The Hessian $f''(0) \in \Hom(\sX,\sX^*)$ is a Fredholm operator with index zero by Feehan and Maridakis \cite[Proposition 3.1.6]{Feehan_Maridakis_Lojasiewicz-Simon_coupled_Yang-Mills_v6}. According to R\r{a}de \cite[Proposition 7.2]{Rade_1992}, the energy function \eqref{eq:Yang-Mills_energy_function}, that is,
\[
  \YM:A_\infty + W_{A_\infty}^{1,2}(X;T^*X\otimes\ad P) \to \RR,
\]
obeys the optimal {\L}ojasiewicz--Simon gradient inequality \eqref{eq:Lojasiewicz-Simon_gradient_inequality_Yang-Mills_energy_local} on the open ball $B_\sigma(A_\infty)$ of radius $\sigma\in(0,1]$ around $A_\infty$ in the affine Hilbert space $A_\infty + W_{A_\infty}^{1,2}(X;T^*X\otimes\ad P)$ defined by \eqref{eq:Lojasiewicz-Simon_gradient_inequality_Yang-Mills_energy_nbhd}. Hence,
\[
  \|f'(a)\|_{\sX^*} \geq Z|f(a)|^{1/2}, \quad\text{for all } a \in \sU,
\]
where $Z\in(0,\infty)$ is a constant and thus $f:\sU\to\RR$ is a Morse--Bott function in the sense of Definition \ref{defn:Morse-Bott_function}. Thus, $\sU\cap\Crit f$ is a smooth submanifold (possibly after shrinking $\sU$) and $\Ker f''(0) = T_0\Crit f$.

We have an $L^2$-orthogonal direct sum
\[
  W_{A_\infty}^{1,2}(X;T^*X\otimes\ad P) = \Ran d_{A_\infty} \oplus \Ker d_{A_\infty}^*,
\]
where
\begin{align*}
  \Ran d_{A_\infty}
  &= \Ran \left(d_{A_\infty}: W_{A_\infty}^{2,2}(X;\ad P) \to  W_{A_\infty}^{1,2}(X;T^*X\otimes\ad P) \right),
  \\
  \Ker d_{A_\infty}^*
  &= \Ker \left(d_{A_\infty}^*: W_{A_\infty}^{1,2}(X;T^*X\otimes\ad P) \to  L^2(X;\ad P) \right),
\end{align*}  
and the $L^2$-orthogonal projection
\[
  \pi:\sA(P) \to A_\infty + \Ker d_{A_\infty}^* \cap W_{A_\infty}^{1,2}(X;T^*X\otimes\ad P)
\]
is a smooth submersion with
\[
  \pi^{-1}(A_\infty + \sU)\cap\Crit\YM = A_\infty + \pi^{-1}(\sU\cap\Crit f).
\]
In particular, $\Crit\YM \subset \sA(P)$ is a smooth submanifold near $A_\infty$ and 
\[
  T_{A_\infty}\Crit\YM
  = T_{A_\infty}(\YM')^{-1}(0))
  = (\YM''(A_\infty))^{-1}(0)
  = \Ker\YM''(A_\infty).
\]
Therefore, $\YM$ is Morse--Bott at $A_\infty$, as claimed.
\end{proof}

R\r{a}de proves the particular case of \cite[Proposition 7.2]{Rade_1992} that we just applied in \cite[Section 10]{Rade_1992}, where the main technical ingredient is his \cite[Lemma 10.1]{Rade_1992}. While there is a slight overlap between his arguments and those of Atiyah and Bott \cite{Atiyah_Bott_1983}, his proof is essentially independent of \cite{Atiyah_Bott_1983} and does \emph{not} proceed by first showing that $\YM$ is Morse--Bott at a Yang--Mills $\U(n)$ connection $A_\infty$ with trivial isotropy subgroup over a Riemann surface.

The second (and far more direct) proof below was suggested to the author by the referee. Before proceeding to the proof proper, we recall some facts about Yang--Mills connections over Riemann surfaces. If $G$ is a Lie group and $P$ is a principal $G$-bundle over a smooth manifold $X$, then there is a bijection between conjugacy classes of representations $\rho:\pi_1(X)\to G$ of the fundamental group of $X$ and gauge-equivalence classes of flat connections on $P$ (see Atiyah and Bott \cite[p. 563]{Atiyah_Bott_1983}, Donaldson and Kronheimer \cite[Proposition 2.2.3]{DK} or Kobayashi \cite[Proposition 1.2.6]{Kobayashi_differential_geometry_complex_vector_bundles}). When $X$ is a closed, connected Riemann surface with genus greater than or equal to one, solutions $A$ to the Yang--Mills equation \eqref{eq:Yang_Mills} that are not necessarily flat can be described in terms of \emph{universal central extensions} of the fundamental group
\[
  1 \xrightarrow{} \ZZ \xrightarrow{} \Gamma \xrightarrow{} \pi_1(X) \xrightarrow{} 1,
\]
and the central extension \cite[Equation (6.5)]{Atiyah_Bott_1983} obtained from $\Gamma$ by extending the center to $\RR$,
\[
  1 \xrightarrow{} \RR \xrightarrow{} \Gamma_\RR \xrightarrow{} \pi_1(X) \xrightarrow{} 1.
\]
(See also Diez and Huebschmann \cite{Diez_Huebschmann_2018}.) Given any homomorphism $\rho:\Gamma_\RR\to G$, there is an induced $G$-connection $A_\rho$ that satisfies the Yang-Mills equation \eqref{eq:Yang_Mills} and one has the

\begin{thm}
\label{thm:Atiyah_Bott_6-7}  
(See Atiyah and Bott \cite[Theorem 6.7]{Atiyah_Bott_1983}.)
Let $G$ be compact Lie group and $X$ be a closed, connected Riemann surface with genus greater than or equal to one. Then the mapping $\rho\to A_\rho$ induces a bijection between conjugacy classes of homomorphisms $\rho:\Gamma_\RR\to G$ and equivalence classes of Yang--Mills connections over $X$.  
\end{thm}

The curvature $F_A$ of a Yang--Mills connection $A$ on $P$ is given by $F_A = \xi\otimes\vol \in \Omega^2(X;\fg)$, where $\fg$ is the Lie algebra of $G$ \cite[Equation (6.10)]{Atiyah_Bott_1983} and one calls $A$ \emph{central} if $\xi$ is in the center of $\fg$. One has the

\begin{prop}
\label{prop:Atiyah_Bott_6-16}  
(See Atiyah and Bott \cite[Proposition 6.16]{Atiyah_Bott_1983}.)
Continue the hypotheses of Theorem \ref{thm:Atiyah_Bott_6-7}. Then every smooth principal $G$-bundle $P$ over $X$ has a central Yang--Mills connection. 
\end{prop}

The value of the function $\YM(A)$ at a central Yang--Mills connection $A$ is determined by the characteristic classes of $P$ and this value is the \emph{absolute minimum} for $P$ (see \cite[p. 562]{Atiyah_Bott_1983} and \cite[Section 12]{Atiyah_Bott_1983}). For $G=\U(n)$, a homomorphism $\Gamma_\RR \to \U(n)$ is a unitary representation of $\Gamma_\RR$. If a representation is irreducible, then $\xi$ is necessarily central; Yang--Mills $\U(n)$-connections for which $\xi$ is diagonal (with entries $-2\pi ik/n$) give rise to the absolute minimum $4\pi^2k^2/n$ for the Yang-Mills function, where $k$ is the first Chern number of the principal $\U(n)$-bundle $P$ \cite[pp. 564, 570]{Atiyah_Bott_1983}. We can now proceed to give the

\begin{proof}[Second proof of Theorem  \ref{mainthm:Lojasiewicz-Simon_inequalities_Yang-Mills_energy_irreducible_Riemann_surface}]
We may assume without loss of generality that $A_\infty$ is $C^\infty$-smooth since, if not, we can appeal to Wehrheim \cite[Theorem 9.4]{Wehrheim_2004} and find a $W^{2,2}$ gauge transformation $u_\infty\in\Aut(P)$ such that $u_\infty(A_\infty)$ is $C^\infty$-smooth and replace $u_\infty(A_\infty)$ by $A_\infty$. While we denote Sobolev spaces of $W^{k,2}$ sections of $\Lambda^l(T^*X)\otimes\ad P$ by $W^{k,2}(X;\Lambda^l(T^*X)\otimes\ad P)$ elsewhere in this article for integers $k\geq 0$, it will be convenient here to suppress the Sobolev notation and denote these spaces simply by $\Omega^l(\ad P)$, with the appropriate Sobolev regularity being understood.
  
From \eqref{eq:Hessian_Yang-Mills_energy_function}, the Hessian at $A_\infty$ with respect to the $L^2$ metric,
\[
  (\Hess\YM(A_\infty)a, b)_{L^2(X)} := \YM''(A_\infty)(a,b), \quad\text{for all } a,b \in \Omega^1(\ad P),
\]  
is given by (see also \cite[Proposition 4.1]{AtiyahBott})
\[
  \Hess\YM(A_\infty)a = d_{A_\infty}^*d_{A_\infty}a + \star [\star F_{A_\infty},a],
  \quad\text{for all } a \in \Omega^1(\ad P).
\]
At a Yang--Mills $\U(n)$-connection $A_\infty$ with trivial isotropy subgroup, the curvature $F_{A_\infty}$ is central as noted in the preceding paragraphs. Therefore, the term $[\star F_{A_\infty}, a]$ is zero for all $a \in \Omega^1(\ad P)$ and we obtain
\[
  \YM''(A_\infty) = d_{A_\infty}^*d_{A_\infty}.
\]
We may view the gradient of the Yang--Mills energy function, determined by the differential  \eqref{eq:Gradient_Yang-Mills_energy_function} and the $L^2$ metric by the following relation for any connection $A\in\sA(P)$,
\[
  (\grad\YM(A), a)_{L^2(X)} := \YM'(A)a, \quad\text{for all } a \in \Omega^1(\ad P),
\]  
as defining a section of a vector bundle over $\sA(P)$,
\[
  \sA(P) \ni A \mapsto s(A) := (A,\grad\YM(A)) = (A,d_A^*F_A) \in \sR(P), 
\]
where
\[
  \sR(P) := \left\{\left(A,\Ran \left(d_A^*:\Omega^2(\ad P) \to \Omega^1(\ad P)\right)\right): A \in \sA(P)\right\} \subset \sA(P)\times \Omega^1(\ad P).
\]
We claim that the section $s$ is \emph{transverse} at $A_\infty$ to the zero section $\sA(P)\times\{0\}\subset\sR(P)$, with
\[
  Ds(A_\infty)a = \Hess\YM(A_\infty)a = d_{A_\infty}^*d_{A_\infty}a \in \Ran d_{A_\infty}^* = \sR(P)_{A_\infty}.
\]
Given the claim, $\sU_{A_\infty}\cap\Crit\YM = \{A\in\sU_{A_\infty}:s(A)=0\}$ will be a smooth submanifold of $\sA(P)$ for a small enough open neighborhood $\sU_{A_\infty}$ of $\sA_\infty$ in $\sA(P)$.

We observe that $s\transv (\sA(P)\times\{0\}) \subset \sR(P)$ at $A_\infty$ if and only if
\[
  \Ran Ds(A_\infty) = \Ran \left(d_{A_\infty}^*:\Omega^2(\ad P) \to \Omega^1(\ad P)\right),
\]
that is, if and only if
\[
  \Ran \left(d_{A_\infty}^*d_{A_\infty}:\Omega^1(\ad P) \to \Omega^1(\ad P)\right)
  = \Ran \left(d_{A_\infty}^*:\Omega^2(\ad P) \to \Omega^1(\ad P)\right),
\]
which is implied by the condition
\[
  \Ran \left(d_{A_\infty}:\Omega^1(\ad P) \to \Omega^2(\ad P)\right) = \Omega^2(\ad P).
\]
Hence, abusing notation slightly\footnote{Since $(\Omega^\bullet(\ad P), d_A)$ forms an elliptic complex if and only if $A$ is a flat connection.}, it suffices to check whether
\[
  \bH_{A_\infty}^2
  := \Ran \left(d_{A_\infty}:\Omega^1(\ad P) \to \Omega^2(\ad P)\right)^\perp
  = \Ker \left(d_{A_\infty}^*:\Omega^2(\ad P) \to \Omega^1(\ad P)\right)
  = \{0\}.
\]
Suppose now that $A_\infty$ has trivial isotropy subgroup (isomorphic to the center of $G$) in $\Aut(P)$ and thus
\[
  \bH_{A_\infty}^0 := \Ker \left(d_{A_\infty}:\Omega^0(\ad P) \to \Omega^1(\ad P)\right) = \{0\}.
\]
Because $\dim X = 2$, we have $d_{A_\infty}^* = -\star d_{A_\infty}\star:\Omega^2(\ad P) \to \Omega^1(\ad P)$ by Warner \cite[Equation 6.1 (2)]{Warner} and linear isomorphisms $\star:\Omega^p(\ad P) \cong \Omega^{2-p}(\ad P)$ for $p=0,1,2$. Hence,
\begin{multline*}
  \bH_{A_\infty}^2
  = \Ker \left(\star d_{A_\infty}\star:\Omega^2(\ad P) \to \Omega^1(\ad P)\right)
  \\
  \cong \Ker \left(\star\,d_{A_\infty}:\Omega^0(\ad P) \to \Omega^1(\ad P)\right)
  \\
  = \Ker \left(d_{A_\infty}:\Omega^0(\ad P) \to \Omega^1(\ad P)\right)
  = \bH_{A_\infty}^0
  = \{0\}.
\end{multline*}
Thus, $\Crit\YM$ is a smooth submanifold of $\sA(P)$ in an open neighborhood of $A_\infty$ in $\sA(P)$ and, as in our first proof of Theorem  \ref{mainthm:Lojasiewicz-Simon_inequalities_Yang-Mills_energy_irreducible_Riemann_surface}, its tangent space at $A_\infty$ is given by
\[
  T_{A_\infty}\Crit\YM
  = T_{A_\infty}(\YM')^{-1}(0))
  = (\YM''(A_\infty))^{-1}(0)
  = \Ker\YM''(A_\infty).
\]
Therefore, $\YM$ is Morse--Bott at $A_\infty$, as claimed.
\end{proof}

%
%

\bibliography{/Users/pfeehan/Dropbox/LATEX/Bibinputs/master,/Users/pfeehan/Dropbox/LATEX/Bibinputs/mfpde}

\def\cprime{$'$} \def\cprime{$'$}
  \def\ocirc#1{\ifmmode\setbox0=\hbox{$#1$}\dimen0=\ht0 \advance\dimen0
  by1pt\rlap{\hbox to\wd0{\hss\raise\dimen0
  \hbox{\hskip.2em$\scriptscriptstyle\circ$}\hss}}#1\else {\accent"17 #1}\fi}
  \def\cprime{$'$} \def\cprime{$'$} \def\cprime{$'$} \def\cprime{$'$}
  \def\polhk#1{\setbox0=\hbox{#1}{\ooalign{\hidewidth
  \lower1.5ex\hbox{`}\hidewidth\crcr\unhbox0}}} \def\cprime{$'$}
  \def\cprime{$'$} \def\cprime{$'$}
  \def\lfhook#1{\setbox0=\hbox{#1}{\ooalign{\hidewidth
  \lower1.5ex\hbox{'}\hidewidth\crcr\unhbox0}}} \def\cprime{$'$}
  \def\cprime{$'$} \def\cprime{$'$} \def\cprime{$'$} \def\cprime{$'$}
\providecommand{\bysame}{\leavevmode\hbox to3em{\hrulefill}\thinspace}
\providecommand{\MR}{\relax\ifhmode\unskip\space\fi MR }
\providecommand{\MRhref}[2]{%
  \href{http://www.ams.org/mathscinet-getitem?mr=#1}{#2}
}
\providecommand{\href}[2]{#2}
\begin{thebibliography}{10}

\bibitem{Adams_Simon_1988}
David Adams and Leon Simon, \emph{Rates of asymptotic convergence near isolated
  singularities of geometric extrema}, Indiana Univ. Math. J. \textbf{37}
  (1988), 225--254. \MR{963501 (90b:58046)}

\bibitem{AdamsFournier}
Robert~A. Adams and John J.~F. Fournier, \emph{Sobolev spaces}, second ed.,
  Elsevier/Academic Press, Amsterdam, 2003. \MR{2424078 (2009e:46025)}

\bibitem{Atiyah_1982}
Michael~F. Atiyah, \emph{Convexity and commuting {H}amiltonians}, Bull. London
  Math. Soc. \textbf{14} (1982), no.~1, 1--15. \MR{642416}

\bibitem{Atiyah_Bott_1983}
Michael~F. Atiyah and Raoul~H. Bott, \emph{The {Y}ang--{M}ills equations over
  {R}iemann surfaces}, Philos. Trans. Roy. Soc. London Ser. A \textbf{308}
  (1983), 523--615. \MR{702806 (85k:14006)}

\bibitem{AtiyahBott}
Michael~F. Atiyah and Raoul~H. Bott, \emph{The moment map and equivariant
  cohomology}, Topology \textbf{23} (1984), no.~1, 1--28. \MR{721448}

\bibitem{ADHM}
Michael~F. Atiyah, Nigel~J. Hitchin, Vladimir~G. Drinfel{\cprime}d, and Yuri~I.
  Manin, \emph{Construction of instantons}, Phys. Lett. A \textbf{65} (1978),
  185--187. \MR{598562 (82g:81049)}

\bibitem{AHS}
Michael~F. Atiyah, Nigel~J. Hitchin, and Isadore~M. Singer, \emph{Self-duality
  in four-dimensional {R}iemannian geometry}, Proc. Roy. Soc. London Ser. A
  \textbf{362} (1978), no.~1711, 425--461. \MR{506229 (80d:53023)}

\bibitem{Aubin}
Thierry Aubin, \emph{Nonlinear analysis on manifolds. {M}onge-{A}mp\`ere
  equations}, Springer, New York, 1982. \MR{681859 (85j:58002)}

\bibitem{Austin_Braam_1995}
David~M. Austin and Peter~J. Braam, \emph{Morse--{B}ott theory and equivariant
  cohomology}, The {F}loer memorial volume, Progr. Math., vol. 133,
  Birkh\"auser, Basel, 1995, pp.~123--183. \MR{1362827 (96i:57037)}

\bibitem{Banyaga_Hurtubise_lectures_morse_homology}
Augustin Banyaga and D.~Hurtubise, \emph{Lectures on {M}orse homology}, Kluwer
  Texts in the Mathematical Sciences, vol.~29, Kluwer Academic Publishers
  Group, Dordrecht, 2004. \MR{2145196}

\bibitem{Banyaga_Hurtubise_2009}
Augustin Banyaga and David~E. Hurtubise, \emph{The {M}orse--{B}ott inequalities
  via a dynamical systems approach}, Ergodic Theory Dynam. Systems \textbf{29}
  (2009), no.~6, 1693--1703. \MR{2563088}

\bibitem{Banyaga_Hurtubise_2010}
Augustin Banyaga and David~E. Hurtubise, \emph{Morse--{B}ott homology}, Trans.
  Amer. Math. Soc. \textbf{362} (2010), no.~8, 3997--4043. \MR{2608393}

\bibitem{Banyaga_Hurtubise_2013}
Augustin Banyaga and David~E. Hurtubise, \emph{Cascades and perturbed
  {M}orse--{B}ott functions}, Algebr. Geom. Topol. \textbf{13} (2013), no.~1,
  237--275. \MR{3031642}

\bibitem{Berger_1977}
Marcel Berger, \emph{Nonlinearity and functional analysis}, Academic Press, New
  York, 1977. \MR{0488101 (58 \#7671)}

\bibitem{BierstoneMilman}
Edward Bierstone and Pierre~D. Milman, \emph{Semianalytic and subanalytic
  sets}, Inst. Hautes \'Etudes Sci. Publ. Math. (1988), no.~67, 5--42.
  \MR{972342 (89k:32011)}

\bibitem{Bott_1954}
Raoul~H. Bott, \emph{Nondegenerate critical manifolds}, Ann. of Math. (2)
  \textbf{60} (1954), 248--261. \MR{0064399 (16,276f)}

\bibitem{Brezis}
Haim Br{\'e}zis, \emph{Functional analysis, {S}obolev spaces and partial
  differential equations}, Universitext, Springer, New York, 2011. \MR{2759829
  (2012a:35002)}

\bibitem{Brezis_Wainger_1980}
Haim Br{\'e}zis and Stephen Wainger, \emph{A note on limiting cases of
  {S}obolev embeddings and convolution inequalities}, Comm. Partial
  Differential Equations \textbf{5} (1980), no.~7, 773--789. \MR{579997}

\bibitem{Carlotto_Chodosh_Rubinstein_2015}
Alessandro Carlotto, Otis Chodosh, and Yanir~A. Rubinstein, \emph{Slowly
  converging {Y}amabe flows}, Geom. Topol. \textbf{19} (2015), no.~3,
  1523--1568, arXiv:1401.3738. \MR{3352243}

\bibitem{Chill_2003}
Ralph Chill, \emph{On the {{\L}}ojasiewicz--{S}imon gradient inequality}, J.
  Funct. Anal. \textbf{201} (2003), 572--601. \MR{1986700 (2005c:26019)}

\bibitem{Colding_Minicozzi_2014sdg}
Tobias~H. Colding and William~P. Minicozzi, II, \emph{{\L}ojasiewicz
  inequalities and applications}, Surveys in Differential Geometry \textbf{XIX}
  (2014), 63--82, arXiv:1402.5087.

\bibitem{Colding_Minicozzi_2015}
Tobias~H. Colding and William~P. Minicozzi, II, \emph{Uniqueness of blowups and
  {{\L}}ojasiewicz inequalities}, Ann. of Math. (2) \textbf{182} (2015), no.~1,
  221--285. \MR{3374960}

\bibitem{Colding_Minicozzi_Pedersen_2015bams}
Tobias~H. Colding, William~P. Minicozzi, II, and E.~K. Pedersen, \emph{Mean
  curvature flow}, Bull. Amer. Math. Soc. (N.S.) \textbf{52} (2015), no.~2,
  297--333. \MR{3312634}

\bibitem{Diez_Huebschmann_2018}
Tobias Diez and Johannes Huebschmann, \emph{Yang-{M}ills moduli spaces over an
  orientable closed surface via {F}r\'{e}chet reduction}, J. Geom. Phys.
  \textbf{132} (2018), 393--414. \MR{3836789}

\bibitem{DK}
Simon~K. Donaldson and Peter~B. Kronheimer, \emph{The geometry of
  four-manifolds}, Oxford University Press, New York, 1990.

\bibitem{Evans2}
Lawrence~C. Evans, \emph{Partial differential equations}, second ed., Graduate
  Studies in Mathematics, vol.~19, American Mathematical Society, Providence,
  RI, 2010. \MR{2597943 (2011c:35002)}

\bibitem{Feehan_yangmillsenergygapflat_corrigendum}
Paul M.~N. Feehan, \emph{Corrigendeum to ``{E}nergy gap for {Y}ang--{M}ills
  connections, {II}: {A}rbitrary closed {R}iemannian manifolds''}, preprint,
  July 15, 2019.

\bibitem{Feehan_yang_mills_gradient_flow_v4}
Paul M.~N. Feehan, \emph{Global existence and convergence of solutions to
  gradient systems and applications to {Y}ang--{M}ills gradient flow},
  arXiv:1409.1525v4, xx+475 pages.

\bibitem{Feehan_nonlinear_uhlenbeck_estimate}
Paul M.~N. Feehan, \emph{Morse theory for the {Y}ang--{M}ills energy function
  near flat connections}, 91 pages, arXiv:1906.03954.

\bibitem{FeehanSlice}
Paul M.~N. Feehan, \emph{Critical-exponent {S}obolev norms and the slice
  theorem for the quotient space of connections}, Pacific J. Math. \textbf{200}
  (2001), no.~1, 71--118, arXiv:dg-ga/9711004. \MR{1863408}

\bibitem{Feehan_yangmillsenergygap}
Paul M.~N. Feehan, \emph{Energy gap for {Y}ang--{M}ills connections, {I}:
  Four-dimensional closed {R}iemannian manifolds}, Adv. Math. \textbf{296}
  (2016), 55--84, arXiv:1412.4114. \MR{3490762}

\bibitem{Feehan_yangmillsenergygapflat_aim}
Paul M.~N. Feehan, \emph{Energy gap for {Y}ang--{M}ills connections, {II}:
  {A}rbitrary closed {R}iemannian manifolds}, Adv. Math. \textbf{312} (2017),
  547--587. \MR{3635819}

\bibitem{Feehan_lojasiewicz_inequality_all_dimensions_morse-bott}
Paul M.~N. Feehan, \emph{On the {M}orse--{B}ott property of analytic functions
  on {B}anach spaces with {{\L}}ojasiewicz exponent one half}, Calc. Var.
  Partial Differential Equations \textbf{59} (2020), no.~2, Paper No. 87, 50,
  arXiv:1803.11319. \MR{4087392}

\bibitem{Feehan_Maridakis_Lojasiewicz-Simon_coupled_Yang-Mills_v6}
Paul M.~N. Feehan and Manousos Maridakis, \emph{{{\L}}ojasiewicz--{S}imon
  gradient inequalities for coupled {Y}ang--{M}ills energy functions}, Memoirs
  of the American Mathematical Society, American Mathematical Society,
  Providence, RI, in press, arXiv:1510.03815v6.

\bibitem{Feehan_Maridakis_Lojasiewicz-Simon_Banach}
Paul M.~N. Feehan and Manousos Maridakis, \emph{{\L}ojasiewicz-{S}imon gradient
  inequalities for analytic and {M}orse-{B}ott functions on {B}anach spaces},
  J. Reine Angew. Math. \textbf{765} (2020), 35--67, arXiv:1510.03817.
  \MR{4129355}

\bibitem{FU}
Daniel~S. Freed and Karen~K. Uhlenbeck, \emph{Instantons and four-manifolds},
  second ed., Mathematical Sciences Research Institute Publications, vol.~1,
  Springer, New York, 1991. \MR{1081321 (91i:57019)}

\bibitem{FrM}
Robert Friedman and John~W. Morgan, \emph{Smooth four-manifolds and complex
  surfaces}, Ergebnisse der Mathematik und ihrer Grenzgebiete (3) [Results in
  Mathematics and Related Areas (3)], vol.~27, Springer--Verlag, Berlin, 1994.
  \MR{1288304}

\bibitem{GilbargTrudinger}
David Gilbarg and Neil~S. Trudinger, \emph{Elliptic partial differential
  equations of second order}, second ed., Grundlehren der Mathematischen
  Wissenschaften [Fundamental Principles of Mathematical Sciences], vol. 224,
  Springer-Verlag, Berlin, 1983. \MR{737190}

\bibitem{Gilkey2}
Peter~B. Gilkey, \emph{Invariance theory, the heat equation, and the
  {A}tiyah--{S}inger index theorem}, second ed., Studies in Advanced
  Mathematics, CRC Press, Boca Raton, FL, 1995. \MR{1396308 (98b:58156)}

\bibitem{Grafakos_classical_fourier_analysis}
Loukas Grafakos, \emph{Classical {F}ourier analysis}, third ed., Graduate Texts
  in Mathematics, vol. 249, Springer, New York, 2014. \MR{3243734}

\bibitem{Haraux_Jendoubi_2007}
Alain Haraux and M.~A. Jendoubi, \emph{On the convergence of global and bounded
  solutions of some evolution equations}, J. Evol. Equ. \textbf{7} (2007),
  449--470. \MR{2328934 (2008k:35480)}

\bibitem{Hirsch}
Morris~W. Hirsch, \emph{Differential topology}, Graduate Texts in Mathematics,
  vol.~33, Springer--Verlag, New York, 1994, Corrected reprint of the 1976
  original. \MR{1336822 (96c:57001)}

\bibitem{Ho_Wilkin_Wu_2019}
Nan-Kuo Ho, Graeme Wilkin, and Siye Wu, \emph{Conditions of smoothness of
  moduli spaces of flat connections and of character varieties}, Math. Z.
  \textbf{293} (2019), no.~1-2, 1--23, arXiv:1610.09987. \MR{4002269}

\bibitem{Huang_2006}
Sen-Zhong Huang, \emph{Gradient inequalities}, Mathematical Surveys and
  Monographs, vol. 126, American Mathematical Society, Providence, RI, 2006.
  \MR{2226672 (2007b:35035)}

\bibitem{Husemoller}
Dale Husemoller, \emph{Fibre bundles}, third ed., Graduate Texts in
  Mathematics, vol.~20, Springer--Verlag, New York, 1994. \MR{1249482
  (94k:55001)}

\bibitem{Isobe_2009}
Takeshi Isobe, \emph{Topological and analytical properties of {S}obolev
  bundles. {I}. {T}he critical case}, Ann. Global Anal. Geom. \textbf{35}
  (2009), no.~3, 277--337. \MR{2495977}

\bibitem{Kirwan_cohomology_quotients_symplectic_algebraic_geometry}
Frances~C. Kirwan, \emph{Cohomology of quotients in symplectic and algebraic
  geometry}, Mathematical Notes, vol.~31, Princeton University Press,
  Princeton, NJ, 1984. \MR{766741}

\bibitem{Kirwan_1987}
Frances~C. Kirwan, \emph{Some examples of minimally degenerate {M}orse
  functions}, Proc. Edinburgh Math. Soc. (2) \textbf{30} (1987), no.~2,
  289--293. \MR{892697}

\bibitem{Kobayashi}
Shoshichi Kobayashi, \emph{Differential geometry of complex vector bundles},
  Publications of the Mathematical Society of Japan, vol.~15, Princeton
  University Press, Princeton, NJ, 1987, Kan{\^o} Memorial Lectures, 5.
  \MR{909698 (89e:53100)}

\bibitem{Kobayashi_differential_geometry_complex_vector_bundles}
Shoshichi Kobayashi, \emph{Differential geometry of complex vector bundles},
  Princeton Legacy Library, Princeton University Press, Princeton, NJ, [2014],
  Reprint of the 1987 edition [ MR0909698]. \MR{3643615}

\bibitem{Kobayashi_Nomizu_v1}
Shoshichi Kobayashi and Katsumi Nomizu, \emph{Foundations of differential
  geometry. {V}ol {I}}, Interscience Publishers, a division of John Wiley \&
  Sons, New York-London, 1963. \MR{0152974 (27 \#2945)}

\bibitem{Koiso_1987}
Norihito Koiso, \emph{Yang-{M}ills connections and moduli space}, Osaka J.
  Math. \textbf{24} (1987), no.~1, 147--171. \MR{881753}

\bibitem{KwonThesis}
Heaseung Kwon, \emph{Asymptotic convergence of harmonic map heat flow}, {Ph.D}.
  thesis, Stanford University, Palo Alto, CA, 2002. \MR{2703296}

\bibitem{Lerman_2005}
Eugene Lerman, \emph{Gradient flow of the norm squared of a moment map},
  Enseign. Math. (2) \textbf{51} (2005), 117--127. \MR{2154623 (2006b:53106)}

\bibitem{Liu_Yang_2010}
Qingyue Liu and Yunyan Yang, \emph{Rigidity of the harmonic map heat flow from
  the sphere to compact {K}\"ahler manifolds}, Ark. Mat. \textbf{48} (2010),
  121--130. \MR{2594589 (2011a:53066)}

\bibitem{Lojasiewicz_1959}
Stanis{\l}aw {\L}ojasiewicz, \emph{Sur le probl\`eme de la division}, Studia
  Math. \textbf{18} (1959), 87--136. \MR{0107168 (21 \#5893)}

\bibitem{Lojasiewicz_1961}
Stanis{\l}aw {\L}ojasiewicz, \emph{Sur le probl\`eme de la division}, Rozprawy
  Mat. \textbf{22} (1961), 1--57. \MR{0126072}

\bibitem{Lojasiewicz_1965}
Stanis{\l}aw {\L}ojasiewicz, \emph{Ensembles semi-analytiques},  (1965), Publ.
  Inst. Hautes Etudes Sci., Bures-sur-Yvette. LaTeX version by M. Coste, August
  29, 2006 based on mimeographed course notes by S. {\L}ojasiewicz, available
  at \url{perso.univ-rennes1.fr/michel.coste/Lojasiewicz.pdf}.

\bibitem{Lorentz_1950}
George~G. Lorentz, \emph{Some new functional spaces}, Ann. of Math. (2)
  \textbf{51} (1950), 37--55. \MR{0033449}

\bibitem{Lorentz_1951}
George~G. Lorentz, \emph{On the theory of spaces {$\Lambda$}}, Pacific J. Math.
  \textbf{1} (1951), 411--429. \MR{0044740}

\bibitem{MilnorStasheff}
John~W. Milnor and James~D. Stasheff, \emph{Characteristic classes}, Princeton
  University Press, Princeton, N. J.; University of Tokyo Press, Tokyo, 1974,
  Annals of Mathematics Studies, No. 76. \MR{0440554}

\bibitem{MMR}
John~W. Morgan, Tomasz~S. Mrowka, and Daniel Ruberman, \emph{The {$L^2$}-moduli
  space and a vanishing theorem for {D}onaldson polynomial invariants},
  Monographs in Geometry and Topology, vol.~2, International Press, Cambridge,
  MA, 1994. \MR{1287851 (95h:57039)}

\bibitem{Mrowka_7-30-2018}
Tomasz~S. Mrowka, \emph{personal communication}, July 30, 2018.

\bibitem{Nicolaescu_morse_theory}
Liviu~I. Nicolaescu, \emph{An invitation to {M}orse theory}, second ed.,
  Universitext, Springer, New York, 2011. \MR{2883440 (2012i:58007)}

\bibitem{ParkerGauge}
Thomas~H. Parker, \emph{Gauge theories on four-dimensional {R}iemannian
  manifolds}, Comm. Math. Phys. \textbf{85} (1982), 563--602. \MR{677998
  (84b:58036)}

\bibitem{Peetre_1966}
Jaak Peetre, \emph{Espaces d'interpolation et th\'eor\`eme de {S}oboleff}, Ann.
  Inst. Fourier (Grenoble) \textbf{16} (1966), no.~fasc. 1, 279--317.
  \MR{0221282}

\bibitem{Peetre_1969}
Jaak Peetre, \emph{On the theory of {$\mathcal{L}_{p,\lambda}$} spaces}, J.
  Functional Analysis \textbf{4} (1969), 71--87. \MR{0241965 (39 \#3300)}

\bibitem{Riviere_2002}
Tristan Rivi{\`e}re, \emph{Interpolation spaces and energy quantization for
  {Y}ang--{M}ills fields}, Comm. Anal. Geom. \textbf{10} (2002), no.~4,
  683--708. \MR{1925499 (2004a:58018)}

\bibitem{Rade_1992}
Johan R\r{a}de, \emph{On the {Y}ang--{M}ills heat equation in two and three
  dimensions}, J. Reine Angew. Math. \textbf{431} (1992), 123--163. \MR{1179335
  (94a:58041)}

\bibitem{Rudin}
Walter Rudin, \emph{Functional analysis}, second ed., International Series in
  Pure and Applied Mathematics, McGraw-Hill, Inc., New York, 1991. \MR{1157815}

\bibitem{Sedlacek}
Steven~B. Sedlacek, \emph{A direct method for minimizing the {Y}ang--{M}ills
  functional over {$4$}-manifolds}, Comm. Math. Phys. \textbf{86} (1982),
  515--527. \MR{679200 (84e:81049)}

\bibitem{SedlacekThesis}
Steven~B. Sedlacek, \emph{A direct method for minimizing the {Y}ang--{M}ills
  functional over four-dimensional manifolds}, {Ph.D}. thesis, Northwestern
  University, Chicago, IL, 1982. \MR{2632187}

\bibitem{Sengupta_1998jgp}
Ambar~N. Sengupta, \emph{The moduli space of flat {${\rm SU}(2)$} and {${\rm
  SO}(3)$} connections over surfaces}, J. Geom. Phys. \textbf{28} (1998),
  no.~3-4, 209--254. \MR{1658751}

\bibitem{Shevchishin_2002}
Vsevolod~V. Shevchishin, \emph{Limit holonomy and extension properties of
  {S}obolev and {Y}ang--{M}ills bundles}, J. Geom. Anal. \textbf{12} (2002),
  no.~3, 493--528. \MR{1901752}

\bibitem{Simon_1983}
Leon Simon, \emph{Asymptotics for a class of nonlinear evolution equations,
  with applications to geometric problems}, Ann. of Math. (2) \textbf{118}
  (1983), 525--571. \MR{727703 (85b:58121)}

\bibitem{Simon_1985}
Leon Simon, \emph{Isolated singularities of extrema of geometric variational
  problems}, Lecture Notes in Math., vol. 1161, Springer, Berlin, 1985.
  \MR{821971 (87d:58045)}

\bibitem{Simon_1996}
Leon Simon, \emph{Theorems on regularity and singularity of energy minimizing
  maps}, Lectures in Mathematics ETH Z\"urich, Birkh{\"a}user, Basel, 1996.
  \MR{1399562 (98c:58042)}

\bibitem{Stein_Weiss_introduction_fourier_analysis}
Elias~M. Stein and Guido Weiss, \emph{Introduction to {F}ourier analysis on
  {E}uclidean spaces}, Princeton University Press, Princeton, N.J., 1971,
  Princeton Mathematical Series, No. 32. \MR{0304972}

\bibitem{Swoboda_2012}
Jan Swoboda, \emph{Morse homology for the {Y}ang--{M}ills gradient flow}, J.
  Math. Pures Appl. (9) \textbf{98} (2012), 160--210, arXiv:1103.0845.
  \MR{2944375}

\bibitem{Tartar_1998}
Luc Tartar, \emph{Imbedding theorems of {S}obolev spaces into {L}orentz
  spaces}, Boll. Unione Mat. Ital. Sez. B Artic. Ric. Mat. (8) \textbf{1}
  (1998), no.~3, 479--500. \MR{1662313}

\bibitem{Tartar_2007}
Luc Tartar, \emph{An introduction to {S}obolev spaces and interpolation
  spaces}, Lecture Notes of the Unione Matematica Italiana, vol.~3, Springer,
  Berlin; UMI, Bologna, 2007. \MR{2328004 (2008g:46055)}

\bibitem{TauSelfDual}
Clifford~H. Taubes, \emph{Self-dual {Y}ang--{M}ills connections on
  non-self-dual {$4$}-manifolds}, J. Differential Geom. \textbf{17} (1982),
  139--170. \MR{658473 (83i:53055)}

\bibitem{TauPath}
Clifford~H. Taubes, \emph{Path-connected {Y}ang--{M}ills moduli spaces}, J.
  Differential Geom. \textbf{19} (1984), 337--392. \MR{755230 (85m:58049)}

\bibitem{TauFrame}
Clifford~H. Taubes, \emph{A framework for {M}orse theory for the
  {Y}ang--{M}ills functional}, Invent. Math. \textbf{94} (1988), 327--402.
  \MR{958836 (90a:58035)}

\bibitem{Topping_1997}
Peter~M. Topping, \emph{Rigidity in the harmonic map heat flow}, J.
  Differential Geom. \textbf{45} (1997), 593--610. \MR{1472890 (99d:58050)}

\bibitem{UhlLp}
Karen~K. Uhlenbeck, \emph{Connections with {$L^{p}$} bounds on curvature},
  Comm. Math. Phys. \textbf{83} (1982), 31--42. \MR{648356 (83e:53035)}

\bibitem{UhlChern}
Karen~K. Uhlenbeck, \emph{The {C}hern classes of {S}obolev connections}, Comm.
  Math. Phys. \textbf{101} (1985), 449--457. \MR{815194 (87f:58028)}

\bibitem{Warner}
Frank~W. Warner, \emph{Foundations of differentiable manifolds and {L}ie
  groups}, Graduate Texts in Mathematics, vol.~94, Springer, New York, 1983.
  \MR{722297 (84k:58001)}

\bibitem{Wehrheim_2004}
Katrin Wehrheim, \emph{Uhlenbeck compactness}, EMS Series of Lectures in
  Mathematics, European Mathematical Society (EMS), Z\"urich, 2004. \MR{2030823
  (2004m:53045)}

\end{thebibliography}
\bibliographystyle{amsplain-nodash}

\end{document}